\documentclass{article}
\usepackage{amsmath,amssymb,amsthm}
\usepackage{graphicx,float,xcolor}
\usepackage{tikz}
\usepackage{mathrsfs}
\usepackage[colorlinks=true]{hyperref}
\usepackage{enumitem}
\usepackage{esint}

\def\R{\mathrm{I\kern-0.21emR}}
\def\N{\mathrm{I\kern-0.21emN}}
\def\Z{\mathbb{Z}}
\newcommand{\C} {\mathbb{C}}

\newcommand{\Lip}{\operatorname{Lip}}

\newcommand{\cl}{\mathrm{cl}}

\renewcommand{\geq}{\geqslant}
\renewcommand{\leq}{\leqslant}

\newtheorem{theorem}{Theorem}[section]
\newtheorem{proposition}{Proposition}[section]
\newtheorem{corollary}{Corollary}[section]
\newtheorem{definition}{Definition}[section]
\newtheorem{lemma}{Lemma}[section]
\newtheorem{example}{Example}[section]
\theoremstyle{definition}\newtheorem{remark}{Remark}[section]

\title{Boundary-compatible interacting approximations of quasilinear PDEs on bounded domains}

\author{
Thierry Paul\footnote{CNRS Laboratoire Ypatia des Sciences Math\'ematiques LYSM, Rome, Italy (\texttt{thierry.paul@sorbonne-universite.fr}).}
\and
Emmanuel Tr\'elat\footnote{Sorbonne Universit\'e, Universit\'e Paris Cit\'e, CNRS, Inria, Laboratoire Jacques-Louis Lions, LJLL, F-75005 Paris, France (\texttt{emmanuel.trelat@sorbonne-universite.fr}).}
}

\date{}

\begin{document}

\maketitle

\begin{abstract}
We develop a general operator-theoretic route that turns Kato-type quasilinear evolution systems on a Banach scale $(Z,X)$ into finite-dimensional interacting approximations. The construction proceeds in two steps. First, one introduces a regularized family $(A_\varepsilon,f_\varepsilon)$ indexed by a scale parameter $\varepsilon>0$, for which the drift $A_\varepsilon[t,z]z+f_\varepsilon[t,z]$ takes values in an output space $Y$ suitable for discretization. Second, one discretizes this regularized dynamics by a sampling-reconstruction pair $(P_N,R_N)$ and obtains an interacting ODE on a finite-dimensional state space $V_N\simeq\R^{dN}$. Our main abstract theorem provides a quantitative estimate of the discrepancy $y_\varepsilon^N-y$ between the lifted discrete solution and the exact one, separating the regularization error $\chi(\varepsilon)$ from the discretization error $(1+L_\varepsilon)N^{-\gamma}$, where $L_\varepsilon$ measures the size of the regularized drift in the output norm. This makes explicit the trade-off between the regularization scale $\varepsilon$, the discretization scale $N$, and the possible deterioration of $L_\varepsilon$ as $\varepsilon\to 0$.

As a running example, we focus on quasilinear PDEs on bounded Lipschitz domains with boundary conditions. We show that Burenkov's variable-step mollifiers provide a boundary-compatible kernelization: they regularize differential operators into explicit integral-interaction operators supported inside the domain and preserve boundary traces of sufficiently regular fields. In this setting one can choose an output space $Y$ for which $L_\varepsilon$ remains uniformly bounded, leading to algebraic convergence rates in $N$ for quasi-uniform discretizations.
\end{abstract}


\section{Introduction}\label{sec_intro}
Particle and meshfree discretizations provide a classical way to represent PDE dynamics through finitely many interacting degrees of freedom. In incompressible fluids, vortex methods build Lagrangian particle representations of the vorticity dynamics, see for instance \cite{CottetKoumoutsakos}. Deterministic particle approximations of scalar conservation laws and nonlocal transport equations have been developed in \cite{dif0, dif1, dif2} and are surveyed in \cite{Chertock_HNA2017}. The weighted particle method of \cite{DegondMasGallic_MC1989} approximates diffusion operators by integral operators before a particle discretization, which is close in spirit to the kernelization used here, and a deterministic particle method for nonlinear diffusion appears in \cite{LionsMas-Gallic}. An early deterministic particle method for diffusion, based on an osmotic velocity and irregular particle meshes, was proposed in \cite{Russo_CPAM1990}, where general boundary conditions were already identified as a difficulty; the boundary-compatible kernelization of the present paper is one way to address this point. Nonlocal and blob-type kernelizations have also been used to approximate diffusion or gradient-flow PDEs by interacting particle systems, see \cite{CarrilloCraigPatacchini_CVPDE2019, CarrilloEspositoWu_CVPDE2024} and references therein. From a complementary viewpoint, \cite{PaulTrelat} studies mean-field, hydrodynamic, and graph limits for interacting finite-dimensional systems, while \cite{Medvedev_SIMA2014} analyzes graph limits for a nonlinear heat equation on dense graphs. In parallel, standard numerical analysis offers a vast arsenal of finite-dimensional discretizations (Galerkin, finite elements, finite volumes), see, e.g., \cite{BrennerScott, EymardGallouetHerbin}. 

The purpose of this paper is not merely to assert that a quasilinear PDE admits some finite-dimensional discretization. Rather, we develop a general operator-theoretic route that produces, from a suitable quasilinear evolution system, an explicit interacting approximation compatible with boundary conditions on bounded domains. The central point is that the interaction is not postulated from the outset: it is created by a boundary-compatible kernelization step. This step converts a differential operator with distributional Schwartz kernel into a genuine integral operator supported inside $\Omega\times\Omega$, while preserving boundary traces. The discretization step then acts on this kernelized dynamics through a sampling-reconstruction pair and yields a finite-dimensional interacting system. In that sense, the present paper adds to the classical passage from a finite interacting system to a continuum limit an upstream passage from an unbounded PDE operator to a boundary-compatible interacting operator.

On a bounded domain $\Omega$ with boundary, naive convolution regularizations typically require extending the unknown outside $\Omega$ and may generate spurious boundary layers or a loss of trace information. To circumvent this, we rely on Burenkov's variable-step mollifiers developed in \cite{Burenkov_1982, Burenkov_1998} and revisited recently in \cite{Hintermuller}. The mollification radius is adapted to the distance to $\partial\Omega$, which keeps the interaction inside $\Omega$ and, for sufficiently regular fields, preserves boundary values and traces of derivatives.

We formalize the above strategy in a Kato-type framework for quasilinear evolution equations (see \cite{Kato_1975, Kato_1993}). The construction proceeds in two steps governed by two scales: a regularization scale $\varepsilon>0$ and a discretization scale $N\in\N^*$. First, we introduce a regularized family $(A_\varepsilon,f_\varepsilon)$ that kernelizes the original operator while producing an output class $Y$ on which discretization is accurate. Second, we discretize the regularized dynamics by a sampling-reconstruction pair $(P_N,R_N)$ and obtain an interacting ODE on a finite-dimensional state space $V_N\simeq\R^{dN}$. The resulting system can be read as a particle-in-cell or blob-type scheme: the unknown is represented by finitely many distinguishable degrees of freedom, and the interaction is mediated by a kernel whose localization is inherited from the regularization. This should be contrasted with the classical meaning of ``particle systems'' in kinetic theory, where particles move in space. In the present work, the sampling sites can be fixed and the interaction occurs in the state space $V_N$.

Our main contributions are the following.
\begin{itemize}[parsep=0.7mm, itemsep=0.7mm, topsep=0.7mm]
\item \textbf{Abstract approximation theorem.}
We prove an approximation theorem (Theorem \ref{thm_approx_abstract}) for Kato-type quasilinear evolution systems, providing an explicit estimate of the discrepancy between the exact solution $y$ and the lifted discrete solution $y_\varepsilon^N$. The estimate separates the regularization error $\chi(\varepsilon)$ from the discretization error $(1+L_\varepsilon)N^{-\gamma}$, where $L_\varepsilon$ measures the size of the regularized drift in the output norm.
\item \textbf{Boundary-compatible kernelization.}
On bounded Lipschitz domains with boundary conditions, we show that variable-step mollifiers yield a boundary-compatible regularization of differential operators: the regularized operators have explicit integral kernels supported in $\Omega\times\Omega$, preserve boundary traces, and map sufficiently regular inputs to an output class with good control.
\item \textbf{Quantitative scale balance.}
We identify explicitly the trade-off between the regularization scale $\varepsilon$, the discretization scale $N$, and the possible deterioration of $L_\varepsilon$ as $\varepsilon\to 0$. Polynomial growth of $L_\varepsilon$ leads to algebraic convergence rates, while more singular regularizations may force logarithmic choices of $\varepsilon_N$.
\end{itemize}

\medskip

These contributions culminate in the main approximation theorem (Theorem \ref{thm_approx_abstract}), which can be informally summarized as follows: under the above Kato-type and regularization assumptions, and for a sampling-reconstruction discretization that approximates the identity on the output class $Y$ at order $N^{-\gamma}$, the lifted discrete solution $y_\varepsilon^N$ satisfies, on a uniform time interval,
$$
\Vert y_\varepsilon^N(t)-y(t) \Vert_X \leq \mathrm{Cst}\left( \chi(\varepsilon) + \frac{1+L_\varepsilon}{N^\gamma} \right) .
$$
Balancing the regularization scale $\varepsilon$ against the discretization scale $N$ then yields an algebraic convergence rate in $N$. In particular, for the running class of quasilinear PDEs on bounded Lipschitz domains, one has $\chi(\varepsilon)=\mathrm{O}(\varepsilon)$ and $L_\varepsilon$ stays uniformly bounded, so the choice $\varepsilon_N\simeq N^{-\gamma}$ gives a rate of order $N^{-\gamma}$ in $L^2(\Omega,\R^d)$ (Corollary \ref{cor_approx_PDE}), with $\gamma=1/n$ for quasi-uniform discretizations of an $n$-dimensional domain. The resulting finite-dimensional systems are explicit interacting systems on $V_N$; two emblematic instances, a transport equation leading to a centered, skew-symmetric (energy-conserving) interaction and a heat equation leading to a graph-Laplacian-type interaction, are worked out in Section \ref{subsec_examples} and discussed in Section \ref{sec_discussion}.

We emphasize that if one only seeks a finite-dimensional approximation of the dynamics, direct discretizations of the unbounded operator $A$ (Galerkin, FEM, finite volumes) may be more economical and are covered by standard numerical analysis. The regularization step of our route is introduced for a different reason: it converts a differential operator with distributional Schwartz kernel and boundary constraints into an explicit integral-interaction operator that remains inside the domain and preserves traces. This kernelization is the key to an interacting-system interpretation that remains meaningful on bounded domains with boundary conditions.

\medskip

The paper is organized as follows.
Section \ref{sec_abstract_setting} recalls Kato's framework for local well-posedness of quasilinear evolution systems on a Banach scale and introduces a running class of quasilinear PDEs on bounded domains.
Section \ref{sec_abstract_approx} presents $\varepsilon$-regularization as a kernelization mechanism and, for the running PDE class, constructs an explicit boundary-compatible regularization based on Burenkov's variable-step mollifiers.
Section \ref{sec_abstract_finite_approx} introduces sampling-reconstruction discretizations on finite-dimensional spaces, proves the main interacting approximation theorem, and instantiates it on the running PDE class to obtain an explicit finite-dimensional interacting system together with quantitative rates.
Section \ref{sec_examples_discretizations} discusses representative discretization templates and points to Appendix \ref{app_discretization_technical} for technical criteria.
Section \ref{sec_discussion} comments on the interacting interpretation, locality patterns, the role of the two-step route, and possible extensions.
The key properties of variable-step mollifiers are recalled in Appendix \ref{app_convolution}, and a general Schur test is recalled in Appendix \ref{app_Schur}.

\paragraph{Notation.}
For $d\in\N^\ast$, we denote by $\Vert\cdot\Vert_{\R^d}$ the Euclidean norm on $\R^d$.
Let $\Omega$ be a measurable subset of $\R^d$.
For any $r\in[1,+\infty]$, we denote by $L^r(\Omega,\R^d)$ the usual Lebesgue space.
For any integer $p\geq 0$, we denote by $W^{p,r}(\Omega,\R^d)$ the Sobolev space of functions whose weak derivatives up to order $p$ belong to $L^r(\Omega,\R^d)$, endowed with
$$
\Vert y\Vert_{W^{p,r}} = \max_{\vert\alpha\vert\leq p}\Vert D^\alpha y\Vert_{L^r}.
$$
For $r=2$, we write $H^p(\Omega,\R^d)=W^{p,2}(\Omega,\R^d)$.

If $E$ and $F$ are Banach spaces, we write $L(E,F)$ for the space of bounded linear maps from $E$ to $F$, and $L(E)=L(E,E)$.

\section{Kato framework and a running PDE class}\label{sec_abstract_setting}
%

Let $T>0$.
We consider the quasilinear evolution equation
\begin{equation}\label{abstract_quasilinear}
\boxed{
\dot{y}(t) = A[t,y(t)] y(t) + f[t,y(t)]
}
\end{equation}
for $t\in[0,T]$, with some initial condition $y(0)=y^0\in X$, where $A[t,z]$ is a linear operator on $X$, of domain $D(A[t,z])$, and $f[t,z]\in X$, for all $(t,z)\in[0,T]\times X$.


When $X=\mathcal{F}(\Omega,\R^d)$ is a space of $\R^d$-valued functions on a given complete metric space $\Omega$, for some $d\in\N^*$, assuming that $t\mapsto y(t)\in X$ is a solution of \eqref{abstract_quasilinear}, since $y(t)$ is a function on $\Omega$, in the sequel we denote indifferently $y(t)(x)=y(t,x)$ for all $t\geq 0$ and $x\in\Omega$.

Note that we use brackets to denote $A[t,y(t)]$ and $f[t,y(t)]$ in \eqref{abstract_quasilinear} in order to underline their possible \emph{nonlocal} dependence with respect to $x\in\Omega$: $A[t,y(t)]$ and $f[t,y(t)]$ do not  necessarily depend only on the value $y(t,x)$ at $x$ of the function $y(t)\in X$, but may involve, for instance, a nonlocal term like $(\int_\Omega y(t,x')^k\, dx' ) y^{k'}(t,x)$ for some $k,k'\geq 1$, or $\int_\Omega \rho(x-x')y(t,x')\, dx' \, y(t,x)^3$ as it is the case in some Vlasov equations.

\paragraph{Objective.}
Our objective is to approximate sufficiently regular solutions $t\mapsto y(t)$ of \eqref{abstract_quasilinear} by lifted solutions of finite-dimensional interacting systems on $V_N\simeq\R^{dN}$, built from sampling and reconstruction operators.

\subsection{Kato's hypotheses and local well-posedness}\label{subsec_kato}


In this section, as well as in the next Section \ref{sec_abstract_approx}, we do not need that $X$ be a space of functions: 
$X$ can be any arbitrary Banach space.

Existence and uniqueness of a solution of \eqref{abstract_quasilinear} are classical and are ensured by proving that the following mapping $\Phi$ is contractive and thus has a fixed point (see \cite{Kato_1975, Kato_1993} or \cite[Section 6.4, Theorem 4.6]{Pazy}): given an appropriate function $t\mapsto z(t)\in X$, $y(\cdot)=\Phi(z(\cdot))$ is defined as the unique solution of $\dot{y}(t) = A[t,z(t)] y(t) + f[t,z(t)]$ such that $y(0)=y^0$. This is done under the following classical assumptions, due to Kato and borrowed from \cite{Kato_1993, Sanekata_1989} (in fact, slightly more general assumptions are done in \cite{Kato_1993}).

\smallskip

\noindent\textbf{Banach space assumptions.}
\begin{enumerate}[label=$(H_{\arabic*})$, leftmargin=*, parsep=0.7mm, itemsep=0.7mm, topsep=0.7mm]
\item\label{H_Z} There exists a Banach subspace $Z$ of $X$, dense in $X$ and continuously embedded in $X$, i.e., there exists $C_1>0$ such that $\Vert z\Vert_X\leq C_1\Vert z\Vert_Z$ for every $z\in Z$. 
\item\label{H_S} There exists an operator $S\in L(Z,X)$ such that $\Vert z\Vert_Z = \Vert z\Vert_X + \Vert Sz\Vert_X$ (graph norm); equivalently, $S$ is a closed operator on $X$ of domain $D(S)=Z$.
\end{enumerate}

\smallskip

Throughout the paper, let $y^0\in Z$ and $T>0$ be fixed.

In what follows, given any $r>0$, we denote by $B_Z(y^0,r) = \{ z\in Z\ \mid\ \Vert z-y^0\Vert_Z\leq r\}$ the closed ball in $Z$ of center $y^0$ and radius $r$, and by $\cl_X({B}_Z(y^0,r))$ its closure in $X$. Note that $\cl_X({B}_Z(y^0,r))=B_Z(y^0,r)$ if $X$ and $Z$ are reflexive.

\smallskip

\noindent\textbf{Assumptions on the operator.}
There exists $r>0$ such that, for all $t\in[0,T]$ and $z\in B_Z(y^0,r)$:
\begin{enumerate}[label=$(H_{\arabic*})$, resume, leftmargin=*, parsep=0.7mm, itemsep=0.7mm, topsep=0.7mm]
\setcounter{enumi}{2}
\item\label{H_stab} (Semigroup and stability) The operator $A[t,z]$ generates a $C_0$ semigroup $(e^{sA[t,z]})_{s\geq 0}$ on $X$, and there exist $M\geq 1$ and $\omega\in\R$ such that, for every $k\in\N^*$, for all $s_1,\ldots,s_k\geq 0$ and all $0\leq t_1\leq\cdots\leq t_k\leq T$, one has
$\Vert e^{s_1A[t_1,z]} \cdots e^{s_kA[t_k,z]} \Vert_{L(X)} \leq M e^{(s_1+\cdots+s_k)\omega}$.\footnote{Note that the latter stability estimate holds true (with $M=1$) if $(e^{sA[t,z]})_{s\geq 0}$ is a semigroup of contractions.
}
\item\label{H_regZ} 
$Z\subset D(A[t,z])$, 
$A[t,z]\in L(Z,X)$ depends continuously on $t$, and there exist $C_4,C'_4\geq 0$ such that $\Vert A[t,z]\Vert_{L(Z,X)}\leq C_4$ and $\Vert A[t,z_1]-A[t,z_2]\Vert_{L(Z,X)}\leq C'_4 \Vert z_1-z_2\Vert_X$ for all $t\in[0,T]$ and $z,z_1,z_2\in B_Z(y^0,r)$.
\item\label{H_intertwining} (Intertwining condition) 
There exist $C_5\geq 0$ and $B\in \mathscr{C}^0([0,T]\times \cl_X({B}_Z(y^0,r)),L(X))$ such that
$\Vert B[t,z]\Vert_{L(X)}\leq C_5$ and
$S A[t,z]  = A[t,z] S + B[t,z] S$ 
on $D(A[t,z]S)\cap D(SA[t,z])$, with $(A[t,z]-\lambda\mathrm{id})^{-1}Z\subset Z$ for every $\lambda>\omega$,
for all $t\in [0,T]$ and $z\in\cl_X({B}_Z(y^0,r))$.\footnote{Equivalently, $SA[t,z]=A[t,z]S+B[t,z]S$ on some subset $E\subset D(A[t,z]S)$ such that $A[t,z]E$ is dense in $Z$ (see \cite[Lemma 1.3]{Kato_1993}). This entails that $Z$ be invariant under the semigroup generated by $A[t,z]$.
The intertwining assumption is actually the hypothesis that allows Kato's theory to build an evolution system that is simultaneously bounded on $X$ and on $Z$.
}
\end{enumerate}

\smallskip

\noindent\textbf{Assumptions on $f$.}
Finally, we assume that:
\begin{enumerate}[label=$(H_{\arabic*})$, resume, leftmargin=*, parsep=0.7mm, itemsep=0.7mm, topsep=0.7mm]
\item\label{H_f_Lip} 
$f\in\mathscr{C}^0([0,T]\times B_Z(y^0,r),Z)$ and there exist $C_6,C'_6\geq 0$ such that $\Vert f[t,z]\Vert_Z\leq C_6$ 
and
$\Vert f[t,z_1]-f[t,z_2]\Vert_X \leq C'_6\Vert z_1-z_2\Vert_X$ for all $t\in[0,T]$ and $z,z_1,z_2\in B_Z(y^0,r)$.
\end{enumerate}

\smallskip

\noindent\textbf{Evolution system.}
As established in \cite[Theorem I]{Kato_1993} (see also \cite{Sanekata_1989} and \cite[Section 6.4]{Pazy}), under Assumptions \ref{H_Z} to \ref{H_intertwining}, given any $z(\cdot)\in\mathscr{C}^0([0,T],X)$ such that $z(t)\in B_Z(y^0,r)$ for every $t\in[0,T]$, there exists an \emph{evolution system} $(U_z(t,s))_{0\leq s\leq t\leq T}$ on $X$, i.e., a family of operators $U_z(t,s)\in L(X)\cap L(Z)$ depending continuously on $(t,s)$ and satisfying, for all $0\leq s\leq\tau\leq t\leq T$:
\begin{enumerate}[label=$(E_{\arabic*})$, leftmargin=*, parsep=0.7mm, itemsep=0.7mm, topsep=0.7mm]
\item\label{E_morphism} $U_z(t,s) = U_z(t,\tau) U_z(\tau,s)$ and $U_z(s,s) = \mathrm{id}_X$.
\item\label{E_UZ} $U_z(t,s)Z\subset Z$; 
\item\label{E_estim_exp} $\Vert U_z(t,s)\Vert_{L(X)} \leq Me^{\omega(t-s)}$ and $\Vert U_z(t,s)\Vert_{L(Z)} \leq Me^{\omega'(t-s)}$ where $\omega'=\omega+MC_5$;
\item\label{E_ode} $\partial_t U_z(t,s) = A[t,z(t)]\, U_z(t,s)$ and $\partial_s U_z(t,s) = -U_z(t,s)A[s,z(s)]$ on $Z$ (the derivatives exist in $L(Z,X)$).
\end{enumerate}

\begin{proposition}\label{prop_existence_uniqueness}
Under Assumptions \ref{H_Z} to \ref{H_f_Lip}, there exist $T'\in(0,T]$ and a unique solution $y(\cdot)\in\mathscr{C}^0([0,T'],Z)\cap\mathscr{C}^1([0,T'],X)$ of \eqref{abstract_quasilinear} such that $y(0)=y^0$. Moreover, $y(t)\in B_Z(y^0,r)$ for every $t\in[0,T']$ and
\begin{equation}\label{duhamel_implicit}
y(t) = U_y(t,0)y^0 + \int_0^t U_y(t,s) f[s,y(s)]\, ds .
\end{equation}
\end{proposition}

Note that, in contrast to the usual Duhamel formula in the classical linear case, the formula \eqref{duhamel_implicit} is implicit in general because of the dependence with respect to $y$, see Remark \ref{rem_explicit} further.

\begin{proof}
The arguments can be found in \cite{Kato_1993, Sanekata_1989} (see also \cite{Kato_1975, Pazy} for a less general result but simpler proof), although not exactly in this form. We give a proof for completeness.

For every $r>0$,
for every $T'\in(0,T]$, let $\mathcal{S}_{T'}$ be the closed convex subset of all $z(\cdot)\in\mathscr{C}^0([0,T'],Z)$ such that $z(0)=y^0$ and $z(t)\in B_Z(y^0,r)$ for every $t\in[0,T']$.
Note that, by \ref{H_Z}, $\mathscr{C}^0([0,T'],Z) \subset \mathscr{C}^0([0,T'],X)$.
Given any $z(\cdot)\in\mathcal{S}_{T'}$, 
we consider the Cauchy problem
\begin{equation}\label{evol_eq_y_z}
\dot{y}(t) = A[t,z(t)] y(t) + f[t,z(t)] , \quad y(0)=y^0 \in Z .
\end{equation}
Using \ref{E_ode} and \eqref{evol_eq_y_z}, we obtain $\frac{d}{ds}(U_z(t,s)y(s))=U_z(t,s)f[s,z(s)]$, that we then integrate on $[0,t]$. Using \ref{H_f_Lip} 
and \ref{E_UZ}, we conclude that \eqref{evol_eq_y_z} 
has a unique solution 
$y(\cdot)\in\mathscr{C}^0([0,T'],X)$, taking its values in $Z$, given by
\begin{equation}\label{sol_y_z}
y(t) = U_z(t,0)y^0 + \int_0^t U_z(t,s) f[s,z(s)]\, ds 
\end{equation}
for every $t\in[0,T']$, and we set $\Phi_{T'}(z(\cdot))=y(\cdot)$. This defines a map $\Phi_{T'}$ on $\mathcal{S}_{T'}$.

Let us prove that, actually, $y(\cdot)\in\mathscr{C}^0([0,T'],Z)\cap\mathscr{C}^1([0,T'],X)$.
Using \ref{H_S}, it suffices to prove that $Sy(\cdot)\in\mathscr{C}^0([0,T'],X)$.
We begin by noting that, using \ref{H_intertwining} and \ref{E_ode}, we have 
$SU_z(t,s) = U_z(t,s)S + \int_s^t U_z(t,\tau) B[\tau,z(\tau)] S U_z(\tau,s)\, d\tau$ 
on $Z$ (actually, this formula implies that $U_z(t,s)\in L(Z)$ depends continuously on $(t,s)$).
Then, using \eqref{sol_y_z}, we have 
\begin{multline}\label{sumoffourterms}
Sy(t) = U_z(t,0)Sy^0 + \int_0^t U_z(t,s) S f[s,z(s)] \, ds 
+ \int_0^t U_z(t,\tau) B[\tau,z(\tau)] S U_z(\tau,0) y^0 \, d\tau \\
+ \int_0^t \int_s^t U_z(t,\tau) B[\tau,z(\tau)] S U_z(\tau,s) f[s,z(s)] \, d\tau \, ds
\end{multline}
for every $t\in[0,T']$, 
and each of these four terms
is an element of $\mathscr{C}^0([0,T'],X)$ thanks to the various assumptions. The claim follows. 
In passing, note that, using the Fubini theorem in the fourth term at the right-hand side of \eqref{sumoffourterms}, the sum of the third and fourth terms is then equal to $\int_0^t U_z(t,\tau) B[\tau,z(\tau)] S y(\tau) \, d\tau$, and thus we get $Sy(\cdot) = \Psi_{T'}(z(\cdot),Sy(\cdot))$ where 
\begin{equation}\label{defPsiT}
\Psi_{T'}(z(\cdot),x(\cdot))(t) = U_z(t,0)Sy^0 + \int_0^t U_z(t,s) S f[s,z(s)] \, ds 
+ \int_0^t U_z(t,\tau) B[\tau,z(\tau)] x(\tau) \, d\tau .
\end{equation}
This remark will be useful at the end of the proof.

Since $y(\cdot)=\Phi_{T'}(z(\cdot))\in\mathscr{C}^0([0,T'],Z)$, it follows that $\Phi_{T'}$ maps $\mathcal{S}_{T'}$ to $\mathcal{S}_{T'}$ if $T'$ is small enough.

Let us prove that $\Phi_{T'}$ is a contraction in $\mathscr{C}^0([0,T'],X)$ if $T'$ is small enough.
Since $U_{z_1}(t,s)-U_{z_2}(t,s) = -\int_s^t \frac{d}{dr} ( U_{z_1}(t,\tau) U_{z_2}(\tau,s) ) \, d\tau = \int_s^t U_{z_1}(t,\tau) ( A[\tau,z_1(\tau)] - A[\tau,z_2(\tau)] ) U_{z_2}(\tau,s) \, d\tau$, we infer from \ref{H_regZ}, \ref{E_estim_exp} and \ref{E_ode} that
\begin{equation}\label{Uz1z2}
\Vert (U_{z_1}(t,s)-U_{z_2}(t,s))\Vert_{L(Z,X)}
\leq C'_4 M^2 (t-s) e^{\omega'(t-s)} \Vert z_1(\cdot)-z_2(\cdot)\Vert_{\mathscr{C}^0([0,T'],X)} 
\end{equation}
for all $0\leq s\leq t\leq T'\leq T$ and all $z_1(\cdot),z_2(\cdot)\in\mathcal{S}_{T'}$. 
Applying \eqref{sol_y_z} to $y_1(\cdot)=\Phi_{T'}(z_1(\cdot))$ and $y_2(\cdot)=\Phi_{T'}(z_2(\cdot))$, we infer from 
\ref{H_f_Lip}, \ref{E_estim_exp} and \eqref{Uz1z2} that
$$
\Vert\Phi_{T'}(z_1(\cdot))-\Phi_{T'}(z_2(\cdot))\Vert_{\mathscr{C}^0([0,T'],X)} \leq T' C \Vert z_1(\cdot)-z_2(\cdot)\Vert_{\mathscr{C}^0([0,T'],X)} .
$$
for all $T'\in(0,T]$ and $z_1(\cdot),z_2(\cdot)\in\mathcal{S}_{T'}$, 
with $C= M e^{\vert\omega'\vert T} ( C'_6 + M C'_4 + M T C'_4 C_6 )$.
The contraction property follows by choosing $T'$ small enough.

In particular, $\Phi_{T'}$ is continuous in $\mathscr{C}^0([0,T'],X)$ norm. Hence, denoting by $\overline{\mathcal{S}}_{T'}$ the closure of $\mathcal{S}_{T'}$ in $\mathscr{C}^0([0,T'],X)$, $\Phi_{T'}$ maps the closed convex set $\overline{\mathcal{S}}_{T'}$ to itself and is a contraction, therefore it has a fixed point $y(\cdot)\in\overline{\mathcal{S}}_{T'}$ (in particular, $y(\cdot)\in\mathscr{C}^0([0,T'],X)$). 

It remains to prove that, actually, $y(\cdot)\in\mathcal{S}_{T'}$ (notably, $y(\cdot)\in\mathscr{C}^0([0,T'],Z)$). 
Note that $\overline{\mathcal{S}}_{T'} = \mathcal{S}_{T'}$ when $X$ and $Z$ are reflexive, so the following argument (developed in \cite{Kato_1993}) is only required in the absence of reflexivity.
Defining $y_0(\cdot)\in \mathscr{C}^0([0,T'],X)$ by $y_0(t)=y^0$ for any $t$, the fixed point $y(\cdot)$ is obtained as the limit in $\mathscr{C}^0([0,T'],X)$ of the sequence $(y_k(\cdot))_{k\in\N}$ of $\overline{\mathcal{S}}_{T'}$ defined by iteration $y_{k+1}(\cdot) = \Phi_{T'}(y_k(\cdot))$. 
Using the map $\Psi_{T'}$ defined by \eqref{defPsiT}, we therefore have $Sy_k(\cdot) = \Psi_{T'}(y_{k-1}(\cdot),Sy_k(\cdot))$ for every $k\in\N^*$.
It follows from \ref{H_intertwining} and \eqref{Uz1z2} that $\Psi_{T'}$ maps continuously $\overline{\mathcal{S}}_{T'}\times E$ to $E$, where $E$ is a closed ball of $\mathscr{C}^0([0,T'],X)$, of center $0$ and of sufficiently large radius, and moreover $\Psi_{T'}$ is contracting with respect to $x(\cdot)$ if $T'$ is chosen small enough. Let us prove that $(Sy_k(\cdot))_{k\in\N}$ is a Cauchy sequence in $\mathscr{C}^0([0,T'],X)$: this is then enough to conclude because it implies that $(y_k(\cdot))_{k\in\N}$ is a Cauchy sequence in $\mathscr{C}^0([0,T'],Z)$, hence it converges and the limit must be $y(\cdot)$.
Let $w(\cdot)\in E$ be such that $w(\cdot) = \Psi_{T'}(y(\cdot),w(\cdot))$ (it exists by the Banach fixed-point theorem). By the triangle inequality, we have
\begin{multline*}
\Vert Sy_k(\cdot)-w(\cdot)\Vert_{\mathscr{C}^0([0,T'],X)}
\leq \Vert \Psi_{T'}(y_{k-1}(\cdot),Sy_k(\cdot)) - \Psi_{T'}(y_{k-1}(\cdot),w(\cdot))\Vert_{\mathscr{C}^0([0,T'],X)} \\
+ \Vert \Psi_{T'}(y_{k-1}(\cdot),w(\cdot)) - \Psi_{T'}(y(\cdot),w(\cdot))\Vert_{\mathscr{C}^0([0,T'],X)} .
\end{multline*}
The first term at the right-hand side is less than $C\Vert Sy_k(\cdot)-w(\cdot)\Vert_{\mathscr{C}^0([0,T'],X)}$ for some $C>0$ because $\Psi_{T'}$ is contracting, and the second term converges to $0$ as $k\rightarrow+\infty$ by continuity of $\Psi_{T'}$. It follows that $Sy_k(\cdot)$ converges to $w(\cdot)$, which finishes the proof.
\end{proof}

\begin{remark}\label{rem_prop_existence_uniqueness1}
The time $T'$ in Proposition \ref{prop_existence_uniqueness} only depends on $y^0$ and on the spaces and various constants considered in Assumptions \ref{H_Z} to \ref{H_f_Lip}: given some Banach spaces $X$ and $Z$ satisfying \ref{H_Z} and \ref{H_S}, some $y^0\in Z$, some $r>0$, $M\geq 1$, $\omega\in\R$ and some nonnegative constants $C_4,C'_4,C_5,C_6,C'_6$, the time $T'$ is uniform with respect to all operators $A$ and functions $f$ satisfying \ref{H_stab} to \ref{H_f_Lip}. 
\end{remark}

\begin{remark}\label{rem_prop_existence_uniqueness_T'uniform}
It also follows from the proof that, taking $T'$ smaller if necessary, for every $\tilde y^0\in B_Z(y^0,r/2)$, there exists a unique solution $y(\cdot)\in\mathscr{C}^0([0,T'],Z)\cap\mathscr{C}^1([0,T'],X)$ of \eqref{abstract_quasilinear} such that $y(0)=\tilde y^0$, which moreover takes its values in $B_Z(y^0,r)$ (similar statements can be found in \cite{HughesKatoMarsden_ARMA1977}).
\end{remark}

\begin{remark}\label{rem_explicit}
Let us comment on the various assumptions done by Kato. 

In the following two cases where the operator does not depend on $y$, the Duhamel formula \eqref{duhamel_implicit} is \emph{explicit}, because $U_y(t,s)$ does not depend on $y$:
\begin{itemize}[topsep=1mm,itemsep=0cm]
\item When $A[t,y]=A$ does not depend on $(t,y)$ and generates a $C_0$ semigroup $(e^{tA})_{t\geq 0}$, we are in the classical framework of semigroup theory (see \cite{EngelNagel,Pazy}). Assumptions \ref{H_Z} to \ref{H_intertwining} are satisfied with $Z=D(A)$, $S=A$, $B=0$, and we have $U_y(t,s)=e^{(t-s)A}$.
\item When $A[t,y]=A(t)$ does not depend on $y$, but depends on $t$, we are in the framework of linear evolution equations, treated for example in \cite[Chapter 5]{Pazy}. 
\end{itemize}
We speak of a \emph{quasilinear} evolution equation when $A[t,y]$ depends on $y$. Then, $U_y(t,s)$ depends on $y(\cdot)$ and the Duhamel formula \eqref{duhamel_implicit} is \emph{implicit} with respect to $y$.

The quasilinear theory for evolution systems, initiated by Kato in the 50s, has been developed by many authors. Here, we have followed the presentation and assumptions done in \cite{Kato_1975, Kato_1993, Sanekata_1989} (see also \cite[Section 6.4]{Pazy}).
Although Assumptions \ref{H_Z} to \ref{H_f_Lip} may seem abstract and difficult to check, they are actually natural, often straightforwardly satisfied. Kato showed that his framework covers most of quasilinear equations encountered in practice, like:
Burgers, Korteweg-de Vries, 
hyperbolic systems of quasilinear partial differential equations of the first order,
Euler and Navier Stokes (incompressible) in $\R^3$, coupled Maxwell-Dirac, quasilinear waves, magnetohydrodynamics (including compressible fluids). 
We also refer to \cite{Muller} for applications to the study of quasilinear Maxwell and wave equations settled on a bounded domain, with either Dirichlet or Neumann-type boundary conditions (see also \cite{DorflerGernerSchnaubelt, SchnaubeltSpitz}).

An even more general theory exists, also initiated by Kato in \cite{Kato_1967} with the notion of nonlinear semigroup, developed in the 70s with the famous Crandall-Liggett generation theorem (see \cite{CrandallLiggett}) or with the theory of maximal monotone operators (see \cite{Brezis_1973}). We refer to the unpublished book manuscript \cite{Benilan} and to the textbook \cite{ItoKappel} for the complete theory. 
For example, nonlinear semigroup theory covers the porous medium equation $\partial_t y = \triangle\varphi(u)$ and Hamilton-Jacobi equations, which are not covered by the quasilinear evolution equation theory.

Our decision to remain within the quasilinear framework is motivated by its simplicity, by the fact that it already covers most usual PDEs (as mentioned above) and, more technically, by the instrumental role of the (implicit) Duhamel formula \eqref{duhamel_implicit}  in deriving our main result, Theorem \ref{thm_approx_abstract}, in Section \ref{sec_abstract_main_result}. 
\end{remark}

\begin{remark}[On refined Kato frameworks]\label{rem_Muller_refinement}
The two-space Banach-scale assumptions \ref{H_Z} to \ref{H_f_Lip} are close to Kato's original semigroup setting.
For quasilinear wave and Maxwell equations, a refined framework due to \cite{Muller} and used later in
\cite{DorichHochbruck_SINUM2022, Dorich_FoCM2025, HochbruckPazur_NM2017, HochbruckPazurSchnaubelt_NM2018}
introduces an intermediate Hilbert space $Z \hookrightarrow Y \hookrightarrow X$,
typically with $Y$ an exact interpolation space between $Z$ and $X$, together with state-dependent positive operators $\Lambda(y)$ defining equivalent norms on $X$ and lower-order terms controlled on balls of $Y$.
This refinement is useful when the nonlinearity is only defined on open sets of $Y$, when pointwise constraints such as $\lambda(u)>0$ are controlled only at the $Y$-level, or when one needs state-dependent energy norms in the numerical analysis.

In the present paper, however, the role of the output space, which we also denote by $Y$, is different: it is the reconstruction/approximation space used in the discretization step, not an intermediate state space of the well-posedness theory.
Since our main objective is boundary-compatible kernelization and spatial interacting approximation, we keep the original Kato framework. 
\end{remark}

\subsection{Running example: a class of quasilinear PDEs on bounded domains}\label{sec_general_class_quasilinear}
We now introduce a running PDE class to which Proposition \ref{prop_existence_uniqueness} applies under standard structural hypotheses. This example will be continued in Sections \ref{sec_abstract_approx} and \ref{sec_abstract_finite_approx}.

Let $\Omega$ be the compact closure of a bounded open subset of $\R^n$ with a Lipschitz boundary, endowed with the induced Euclidean distance and the induced Lebesgue measure. Without loss of generality, we assume that $\vert\Omega\vert=1$. We denote by $\mathring{\Omega}$ the interior of $\Omega$.

Fix $p\in\N^*$, $d\in\N^*$, $T>0$, and an integer $s>p+1+\frac{n}{2}$.\footnote{We take $s$ integer for simplicity, because the variable-step mollifier estimates used later are stated in integer-order Sobolev spaces. A fractional-order version could be developed as well, but this is not needed here.}
Let $y^0\in Z$ and consider the quasilinear PDE
\begin{equation}\label{general_quasilinear_PDE}
\boxed{
\left\{
\begin{array}{ll}
\displaystyle
\partial_t y(t,x)=\sum_{\vert\alpha\vert\leq p} a_\alpha[t,y(t)](x) D^\alpha y(t,x)+f[t,y(t)](x)
& \textrm{in }(0,T)\times\Omega,
\\
F\Big(\big((D^\alpha y(t))_{\vert\partial\Omega}\big)_{\vert\alpha\vert\leq p-1}\Big)=0
& \textrm{on }(0,T)\times\partial\Omega,
\\
y(0,x)=y^0(x)
& \textrm{in }\Omega,
\end{array}
\right.
}
\end{equation}
where $y(t,x)\in\R^d$, $a_\alpha[t,z](x)\in\R^{d\times d}$, and $f[t,z](x)\in\R^d$.
We write
\begin{equation}\label{def_A_PDE}
A[t,z]=\sum_{\vert\alpha\vert\leq p} a_\alpha[t,z] D^\alpha
\end{equation}
with domain
\begin{equation}\label{domain_A}
D(A[t,z])=D=
\Big\{y\in H^p(\Omega,\R^d)\ \Big\vert\ F\Big(\big((D^\alpha y)_{\vert\partial\Omega}\big)_{\vert\alpha\vert\leq p-1}\Big)=0\Big\}.
\end{equation}
The boundary operator $F$ is assumed to be a continuous linear map on the corresponding product of trace spaces. This ensures that $D$ is a linear subspace of $H^p(\Omega,\R^d)$ and that $A[t,z]$ is a linear operator on its common domain $D$ for each fixed $(t,z)$.
We set
$$
X=L^2(\Omega,\R^d),
\qquad
Z=H^s(\Omega,\R^d)\cap D,
$$
and we endow $Z$ with the norm $\Vert y\Vert_Z=\Vert y\Vert_{H^s}$. Since $s>p+1+\frac{n}{2}$, the Sobolev embedding yields
$Z\hookrightarrow W^{p+1,\infty}(\Omega,\R^d)\hookrightarrow W^{1,\infty}(\Omega,\R^d)\hookrightarrow X$.

We assume that $(A,f)$ satisfies Assumptions \ref{H_Z} to \ref{H_f_Lip} on the pair $(X,Z)$. In particular, \eqref{general_quasilinear_PDE} admits a unique solution $y\in \mathscr{C}^0([0,T'],Z)\cap \mathscr{C}^1([0,T'],X)$ on some interval $[0,T']\subset[0,T]$, and $y(t)\in B_Z(y^0,r)$ for every $t\in[0,T']$.

In addition, we assume that there exists $C_a>0$ such that
\begin{equation}\label{coeff_W1inf_PDE}
\max_{\vert\alpha\vert\leq p}
\Vert a_\alpha[t,z]\Vert_{W^{1,\infty}(\Omega,\R^{d\times d})}
\leq C_a
\qquad
\forall (t,z)\in[0,T]\times B_Z(y^0,r).
\end{equation}
This additional coefficient bound will be used only to identify the output space $Y$ in Assumption \ref{H_Y} further.

The intertwining hypothesis \ref{H_intertwining} is regarded here as a structural assumption on the original family $A[t,z]$. For differential operators of order greater than one, it cannot be expected in full generality from the sole form \eqref{def_A_PDE}; it must be checked on a case-by-case basis, as in Kato's original examples (see \cite{Kato_1975}).
Further comments on refinements of Kato's framework, and on wave- and Maxwell-type applications on bounded domains, are gathered in Remark \ref{rem_Muller_refinement}. At the level of the present paper, the intertwining property remains structural.

\paragraph{Schwartz kernel.}
At the level of the original differential operator $A[t,z]$, the natural kernel is a Schwartz kernel supported on the diagonal of $\Omega\times\Omega$, hence distributional rather than integrable. The regularization step introduced in Section \ref{sec_abstract_approx} below replaces this singular diagonal kernel by a genuine localized integral kernel.
This point is one of the main motivations of the two-step construction: kernelization makes the PDE dynamics accessible to interacting approximations while preserving the geometry of the domain and its boundary conditions.

\section{Regularization and kernelization}\label{sec_abstract_approx}
\subsection{General abstract approximation result}
In addition to Assumptions \ref{H_Z} to \ref{H_f_Lip}, we assume that:
\begin{enumerate}[label=$(H_{\arabic*})$,resume] 
\setcounter{enumi}{6}
\item\label{H_Aepsilon} 
There exist a family of linear operators $A_\varepsilon[t,z]$ on $X$ and a family of functions $f_\varepsilon:[0,T]\times X\rightarrow X$, indexed by $\varepsilon\in(0,\varepsilon_0]$ for some $\varepsilon_0>0$, such that, for all $t\in[0,T]$ and $z\in B_Z(y^0,r)$, Assumptions \ref{H_stab} to \ref{H_f_Lip} hold (with $S$, $A$, $B$, $f$ replaced by $S_\varepsilon$, $A_\varepsilon$, $B_\varepsilon$, $f_\varepsilon$) uniformly with respect to $\varepsilon$, i.e., with constants $C_i,M,\omega$ that may be larger if necessary but do not depend on $\varepsilon$. The comparison estimates below and in Proposition \ref{prop_yepsilon} are always measured in the fixed scale $(Z,X)$, even if the intertwining operator $S_\varepsilon$ differs from $S$ (see Remark \ref{rem_intertwining_continuous_not_generic}).
\item\label{H_Aepsilon_CV} There exists 
a continuous function $\chi:[0,\varepsilon_0]\rightarrow[0,+\infty)$ satisfying $\chi(0)=0$, such that
\begin{equation*}
\begin{split}
\Vert (A_\varepsilon[t,z]-A[t,z])y\Vert_X &\leq \chi(\varepsilon) \Vert y\Vert_Z  \qquad
\forall y\in Z , \\ 
\Vert f_\varepsilon[t,z]-f[t,z]\Vert_X &\leq \chi(\varepsilon) ,
\end{split}
\end{equation*}
for all $t\in[0,T]$, $z\in B_Z(y^0,r)$ and $\varepsilon\in(0,\varepsilon_0]$.
%
\end{enumerate}

In \ref{H_Aepsilon}, the most stringent requirement is the uniform semigroup hypothesis \ref{H_stab} for the family $A_\varepsilon[t,z]$. This is classical in Trotter-Kato approximation theory (see \cite{EngelNagel, Pazy}) and is often verified through dissipativity arguments.

Note that we do not assume that $A_\varepsilon[t,z]$ be bounded on $X$, although bounded approximations are of particular interest.

\begin{remark}\label{rem_kernelization}
In many PDE applications, regularization is used not merely to smooth but to convert a differential operator into an integral-interaction operator with an explicit kernel supported in $\Omega\times\Omega$. This kernelization viewpoint is particularly valuable on bounded domains with boundary conditions. The variable-step regularization developed in Section \ref{sec_main_result} is an example: it keeps interactions inside $\Omega$ and preserves boundary traces.
\end{remark}

A canonical bounded approximation is the \emph{Yosida approximant}
$$
A_\varepsilon[t,z] = A[t,z] J_\varepsilon[t,z]
\qquad\textrm{where}\qquad
J_\varepsilon[t,z] = \big(\mathrm{id}-\varepsilon A[t,z]\big)^{-1}.
$$
For every fixed $(t,z)$, $A_\varepsilon[t,z]$ is bounded and therefore generates a uniformly continuous semigroup. In the m-dissipative Hilbert setting, or more generally under the corresponding Hille-Yosida stability estimates, the stability constants can be chosen uniformly in $\varepsilon$. Under the present Kato assumptions, the resolvent invariance and intertwining structure also allow one to verify the analogues of \ref{H_regZ} and \ref{H_intertwining} for the Yosida approximants. Thus, from the abstract semigroup viewpoint, Yosida regularization is compatible with the present framework. 

We nevertheless do not use it later, for two reasons. First, it is non-explicit and does not provide a kernel representation adapted to the boundary geometry of $\Omega$. Second, the main purpose of the present paper is precisely to build a boundary-compatible interacting approximation, for which the regularization step must produce an explicit integral kernel supported inside $\Omega\times\Omega$.
Moreover, $A_\varepsilon[t,z]y\rightarrow A[t,z]y$ for every $y\in D(A[t,z])$, although $\Vert A_\varepsilon[t,z]\Vert_{L(X)}\rightarrow+\infty$ if $A[t,z]$ is unbounded on $X$.

Assumption \ref{H_Aepsilon_CV} refers to convergence estimates, which are often proved by explicit approximation constructions (see also \cite{ItoKappel_MC1998} for finite-dimensional approximations with error estimates), as we will do hereafter. 
For instance when $X=L^2(\Omega,\R^d)$, we often take $Z=H^s(\Omega,\R^d)$ for $s$ large enough. 

%

These approximation assumptions being done, for every $\varepsilon\in(0,\varepsilon_0]$, we consider the quasilinear evolution equation
\begin{equation}\label{abstract_quasilinear_eps}
\boxed{
\dot y_\varepsilon(t) = A_\varepsilon[t,y_\varepsilon(t)] y_\varepsilon(t) + f_\varepsilon[t,y_\varepsilon(t)]
}
\end{equation}

\begin{proposition}\label{prop_yepsilon}
We make Assumptions \ref{H_Z} to \ref{H_Aepsilon_CV}.
Let $y(\cdot)\in\mathscr{C}^0([0,T'],Z)\cap\mathscr{C}^1([0,T'],X)$ be the unique solution of \eqref{abstract_quasilinear} such that $y(0)=y^0$, as given by Proposition \ref{prop_existence_uniqueness}, where $T'$ is as in Remark \ref{rem_prop_existence_uniqueness_T'uniform}.
 
For every $\varepsilon\in(0,\varepsilon_0]$ and every $y_\varepsilon^0\in B_Z(y^0,r/2)$, there exists a unique solution $y_\varepsilon(\cdot)\in\mathscr{C}^0([0,T'],Z)\cap\mathscr{C}^1([0,T'],X)$ of \eqref{abstract_quasilinear_eps} such that $y_\varepsilon(0)=y_\varepsilon^0$. 
Moreover, $y_\varepsilon(t)\in B_Z(y^0,r)$ for every $t\in[0,T']$ and
\begin{equation}\label{CV_yeps_abstract}
\Vert y_\varepsilon(t)-y(t)\Vert_X \leq a_0(t) \Vert y_\varepsilon^0-y^0\Vert_X + a_1(t) \chi(\varepsilon) 
\qquad\forall t\in[0,T'] \qquad\forall\varepsilon\in(0,\varepsilon_0] ,
\end{equation}
where $\beta = M(C'_4 (\Vert y^0\Vert_Z+r) + C'_6)$ and
\begin{equation}\label{def_a1a2}
a_0(t)=M e^{(\beta+\omega)t}\leq \mathrm{Cst},
\qquad
a_1(t)=M(\Vert y^0\Vert_Z+r+1)e^{\beta t}\int_0^t e^{\omega s}\,ds\leq \mathrm{Cst}.
\end{equation}
\end{proposition}



Note that the time $T'$ is uniform in $\varepsilon$, and that Assumption \ref{H_Aepsilon_CV} 
is not required for the first part of Proposition \ref{prop_yepsilon}.
%
This result is similar to \cite[Theorem 7]{Kato_1975} and \cite[Theorem III]{Kato_1993}, where the above convergence result is proved 
without convergence estimate. 
Proposition \ref{prop_yepsilon} can thus be seen as a slight improvement, quantifying the convergence. 

In the sequel, we will apply Proposition \ref{prop_yepsilon} with $y_\varepsilon^0=y^0$.

\begin{proof}
Existence and uniqueness on $[0,T']$ come from Proposition \ref{prop_existence_uniqueness} and Remarks \ref{rem_prop_existence_uniqueness1} and \ref{rem_prop_existence_uniqueness_T'uniform}. This gives the first part of the proposition. 
Let us prove \eqref{CV_yeps_abstract}.
For any $\varepsilon\in(0,\varepsilon_0]$, writing  
\begin{multline*}
\frac{d}{dt}(y_\varepsilon(t)-y(t))
= A_\varepsilon[t,y_\varepsilon(t)](y_\varepsilon(t)-y(t)) + \big(A_\varepsilon[t,y_\varepsilon(t)]-A_\varepsilon[t,y(t)]\big)y(t) \\
+ \big(A_\varepsilon[t,y(t)]-A[t,y(t)]\big)y(t) 
+ f_\varepsilon[t,y_\varepsilon(t)]-f_\varepsilon[t,y(t)] + f_\varepsilon[t,y(t)]-f[t,y(t)],
\end{multline*}
and applying the Duhamel formula \eqref{duhamel_implicit}, we infer that, for every $t\in[0,T']$,
\begin{multline*}
y_\varepsilon(t)-y(t) = U^\varepsilon_{y_\varepsilon}(t,0) (y_\varepsilon^0-y^0) \\ 
+ \int_0^t U^\varepsilon_{y_\varepsilon}(t,s)\, \Big( \big( A_\varepsilon[s,y_\varepsilon(s)] - A_\varepsilon[s,y(s)] \big) y(s) + \big( A_\varepsilon[s,y(s)] - A[s,y(s)] \big) y(s) \\
+ f_\varepsilon[s,y_\varepsilon(s)]-f_\varepsilon[s,y(s)] + f_\varepsilon[s,y(s)]-f[s,y(s)] \Big) \, ds 
\end{multline*}
where, for any $\varepsilon\in(0,\varepsilon_0]$, $(U^\varepsilon_z(t,s))_{0\leq s\leq t\leq T}$ is the evolution system on $X$ corresponding to the quasilinear operator $A_\varepsilon$, for any $z(\cdot)\in\mathscr{C}^0([0,T],X)$ such that $z(t)\in B_Z(y^0,r)$ for every $t\in[0,T]$. It follows from Assumption \ref{H_Aepsilon} that this evolution system satisfies the stability estimates \ref{E_estim_exp} uniformly with respect to $\varepsilon$.
By Proposition \ref{prop_existence_uniqueness}, both $y_\varepsilon(s)$ and $y(s)$ belong to $B_Z(y^0,r)$ for any $s\in[0,T']$.
Using the uniform stability estimates \ref{E_estim_exp} for $U^\varepsilon_{y_\varepsilon}$, the (uniform) Lipschitz properties \ref{H_regZ} for $A_\varepsilon$ and \ref{H_f_Lip} for $f_\varepsilon$, and the convergence estimates \ref{H_Aepsilon_CV}, 
we infer that
\begin{multline*}
\Vert y_\varepsilon(t)-y(t)\Vert_X
\leq Me^{\omega t} \Vert y_\varepsilon^0-y^0\Vert_X \\
+ M \int_0^t e^{\omega(t-s)} \Big( \big( C'_4 (\Vert y^0\Vert_Z+r) + C'_6 \big) \Vert y_\varepsilon(s)-y(s)\Vert_X + (\Vert y^0\Vert_Z+r +1) 
\chi(\varepsilon)  \Big) \, ds
\end{multline*}
and therefore, by the Gronwall lemma\footnote{The Gronwall lemma states that, if $0\leq u(t)\leq \alpha(t)+\beta \int_0^t u(s)\, ds$ with $\alpha$ nondecreasing, then $u(t)\leq\alpha(t) e^{\beta t}$.}
applied to $e^{-\omega t}\Vert y_\varepsilon(t)-y(t)\Vert_X$, we obtain finally \eqref{CV_yeps_abstract}.
\end{proof}

\begin{remark}
When $\omega=0$, we have $\int_0^t e^{\omega s}\, ds=t$; otherwise, this integral is equal to $\frac{e^{\omega t}-1}{\omega}$.
Note that if $C'_4=C'_6=0$ (i.e., if $A[t,z]$ and $f[t,z]$ do not depend on $z$) then $\beta=0$.
\end{remark}

In Section \ref{sec_abstract_finite_approx} hereafter, we will need the following slight additional regularity assumption:

\begin{enumerate}[label=$(H_{\arabic*})$, resume]
\setcounter{enumi}{8}
\item\label{H_Y}
There exist a Banach space $Y$ continuously embedded in $X$ such that, for every $\varepsilon\in(0,\varepsilon_0]$, we have $A_\varepsilon[t,z]z\in Y$ and $f_\varepsilon[t,z]\in Y$ and there exists $L_\varepsilon\geq 0$ such that
$$
\Vert A_\varepsilon[t,z]z+f_\varepsilon[t,z]\Vert_Y \leq L_\varepsilon 
\qquad \forall t\in[0,T]\qquad\forall z\in B_Z(y^0,r) .
$$
\end{enumerate}
In what follows we denote by $C_9>0$ the embedding constant of $Y\hookrightarrow X$, i.e., $\Vert y\Vert_X\leq C_9\Vert y\Vert_Y$ for every $y\in Y$.

In practice, $Y$ will often be chosen as an intermediate Banach space such that $Z\hookrightarrow Y\hookrightarrow X$, so that the condition on $f_\varepsilon$ is automatic. Assumption \ref{H_Y} then means that $A_\varepsilon[t,z]z$ is slightly more regular than merely belonging to $X$.

The size of $L_\varepsilon$ will play a decisive role in the final balance between the regularization parameter $\varepsilon$ and the discretization parameter $N$ in Theorem \ref{thm_approx_abstract}. In the abstract framework, $L_\varepsilon$ is allowed to depend on $\varepsilon$ because many regularizations come with explicit $\varepsilon$-dependent prefactors. In the boundary-compatible variable-step regularization developed later in Section \ref{sec_main_result}, one actually obtains a uniform bound $L_\varepsilon\leq L_0$ for a suitable choice of output space $Y$. In other regularizations, however, $L_\varepsilon$ may deteriorate as $\varepsilon\to 0$, and the scale balance in Theorem \ref{thm_approx_abstract} then becomes nontrivial.

\begin{remark}\label{rem_Leps_growth}
Assumption \ref{H_Y} is the abstract interface between the regularized evolution equation and the discretization step. It isolates an output class $Y$ on which the reconstruction projector $Q_N$ approximates the identity in the pivot norm $X$.
In the abstract theorem, the constant $L_\varepsilon$ is allowed to depend on $\varepsilon$. In the boundary-compatible variable-step regularization used later in Section \ref{sec_main_result}, one actually obtains a uniform bound $L_\varepsilon\leq L_0$.
In other regularizations, however, $L_\varepsilon$ may deteriorate as $\varepsilon\to 0$.

Typical mechanisms are the following.
For Yosida-type approximations, boundedness on $X$ is automatic, but if one asks for control in a stronger output norm $Y$ related to a graph norm or to additional derivatives, then polynomial growth in $\varepsilon^{-1}$ is typical (see for instance \cite{EngelNagel, Pazy}). 
More generally, kernelizations or regularizations that produce outputs with stronger $\varepsilon$-dependent derivative bounds may yield polynomial or even exponential growth of $L_\varepsilon$ (see, e.g., \cite{CarrilloEspositoWu_CVPDE2024}).
The role of Theorem \ref{thm_approx_abstract} further is precisely to separate this possible growth from the purely spatial discretization scale $N$.
\end{remark}

\subsection{Running example: boundary-compatible kernelization by variable-step mollifiers}\label{sec_main_result}
We continue the running PDE class of Section \ref{sec_general_class_quasilinear} and show how the regularization assumptions of Section \ref{sec_abstract_approx} can be verified by an explicit boundary-compatible kernelization on a bounded Lipschitz domain.
The regularization is built from the variable-step mollifier of Appendix \ref{app_convolution}, which is an appropriate modification of the usual convolution near the boundary.
From the viewpoint of interacting approximations, this regularization does three things at once:
\begin{itemize}[parsep=0.7mm, itemsep=0.7mm, topsep=0.7mm]
\item it smooths while keeping the interaction inside $\Omega$ by adapting the mollification radius to the distance to $\partial\Omega$;
\item it converts differential operators into explicit integral-interaction operators with kernels supported in $\Omega\times\Omega$;
\item it preserves boundary traces of sufficiently regular fields, which is crucial to respect boundary conditions.
\end{itemize}
We verify all abstract assumptions that follow directly from this construction.
The two genuinely structural points that are not automatic at this level of generality are the intertwining assumption \ref{H_intertwining}, already present in Kato's theory for the original family $A[t,z]$, and the semigroup-generation part of \ref{H_Aepsilon} for the regularized family $A_\varepsilon[t,z]$.
Both issues are therefore isolated explicitly below.

\subsubsection{Variable-step mollifier and regularized operator}\label{sec_regularized_operator}
Let $\eta\in \mathscr{C}_c^\infty(\R^n)$ be a nonnegative function supported in the unit ball, satisfying $\int_{\R^n} \eta(x)\,dx=1$.
Let $\rho\in \mathscr{C}^s(\Omega)$ be another nonnegative function such that $\rho>0$ on $\mathring{\Omega}$ and $\rho(x)\leq \mathrm{d}_{\R^n}(x,\partial\Omega)$ for every $x\in\Omega$, and such that all derivatives of $\rho$ up to order $s-1$ vanish on $\partial\Omega$.
This stronger boundary flatness is convenient because Appendix \ref{app_convolution} will be used below with the integer $s$ in place of the integer denoted there by $q$. 

For $\varepsilon\in(0,1]$, let $\mathcal{H}_\varepsilon$ be the variable-step mollifier defined in Definition \ref{def_H_eps} in Appendix \ref{app_convolution}, and let $\varepsilon_0>0$ be such that all the bounds of Appendix \ref{app_convolution} hold on $(0,\varepsilon_0]$.

Recall that $X=L^2(\Omega,\R^d)$ and $Z=H^s(\Omega,\R^d)\cap D$.

\begin{lemma}\label{lem_Heps_preserves_D}
For every $\varepsilon\in(0,\varepsilon_0]$, the operator $\mathcal{H}_\varepsilon$ maps $Z$ continuously into $Z$ and preserves the boundary-condition domain $D$, i.e., $\mathcal{H}_\varepsilon(Z)\subset Z$ and $\mathcal{H}_\varepsilon(D)\subset D$.
Moreover, there exists $C_H^Z>0$, independent of $\varepsilon$, such that
$\Vert \mathcal{H}_\varepsilon y\Vert_Z\leq C_H^Z \Vert y\Vert_Z$ for any $y\in Z$ and any $\varepsilon\in(0,\varepsilon_0]$.
\end{lemma}

\begin{proof}
Since $Z\subset H^s(\Omega,\R^d)$, Corollary \ref{H_eps_Wmr} of Appendix \ref{app_convolution} yields $\Vert \mathcal{H}_\varepsilon y\Vert_{H^s}\leq C_H^Z \Vert y\Vert_{H^s}$, with $C_H^Z$ independent of $\varepsilon$.
Next, because all derivatives of $\rho$ up to order $s-1$ vanish on $\partial\Omega$, the traces of $D^\alpha \mathcal{H}_\varepsilon y$ and $D^\alpha y$ coincide on $\partial\Omega$ for every multi-index $\alpha$ with $\vert\alpha\vert\leq p-1$: for $\vert\alpha\vert=0$ this is Lemma \ref{lem_preserv_value_boundary}, and for $1\leq\vert\alpha\vert\leq p-1$ this is Corollary \ref{cor_preserv_deriv_boundary}.
Since the boundary operator $F$ in \eqref{domain_A} only depends on these traces, it follows that if $y\in D$ then $\mathcal{H}_\varepsilon y\in D$.
Therefore $\mathcal{H}_\varepsilon$ maps $Z=H^s(\Omega,\R^d)\cap D$ into itself, and, using the norm $\Vert y\Vert_Z=\Vert y\Vert_{H^s}$, the claimed estimate follows.
\end{proof}

For every $\varepsilon\in(0,\varepsilon_0]$, every $t\in[0,T]$, and every $z\in B_Z(y^0,r)$, we define
\begin{equation}\label{def_Aeps_Heps}
A_\varepsilon[t,z]=\mathcal{H}_\varepsilon^* A[t,z] \mathcal{H}_\varepsilon,
\qquad
f_\varepsilon[t,z]=f[t,z].
\end{equation}
Since $\mathcal{H}_\varepsilon(D)\subset D$, the operator $A_\varepsilon[t,z]$ is well defined on $Z$.

\paragraph{Schwartz kernel.}
It is sometimes useful to write $A_\varepsilon[t,z]$ through its Schwartz kernel. For $y\in Z$,
\begin{equation}\label{def_A_eps_sigma_eps}
(A_\varepsilon[t,z]y)(x)=\int_\Omega \sigma_\varepsilon[t,z](x,x') y(x')\,dx',
\end{equation}
where
\begin{equation*}
\sigma_\varepsilon[t,z](x,x')
=
\sum_{\vert\alpha\vert\leq p}
\int_\Omega
H_\varepsilon(x'',x)
\,a_\alpha[t,z](x'')
\,D_{x''}^\alpha H_\varepsilon(x'',x')
\,dx''.
\end{equation*}
In particular, if $\sigma_\varepsilon[t,z](x,x')\neq 0$, then there exists $x''\in\Omega$ such that
$\Vert x''-x\Vert_{\R^n}\leq \varepsilon \rho(x'')$ and $\Vert x''-x'\Vert_{\R^n}\leq \varepsilon \rho(x'')$,
and therefore
$\Vert x-x'\Vert_{\R^n}\leq 2\varepsilon \Vert \rho\Vert_{L^\infty(\Omega)}$.
Thus the kernel is localized in an $\mathrm{O}(\varepsilon)$-neighborhood of the diagonal.
This localization property will later be used to interpret the lifted finite-dimensional dynamics as a sparse interacting system when the discretization scale is chosen of the same order as $\varepsilon$ (see Section \ref{subsec_examples_locality}).

\subsubsection{Verification of the abstract assumptions}\label{sec_verif}
We separate the verification of the abstract assumptions into two groups.
First, we verify the assumptions that follow directly from the mapping and approximation properties of the variable-step mollifier: \ref{H_regZ}, \ref{H_Y}, and \ref{H_Aepsilon_CV}.
Second, we discuss the more structural issues: dissipativity, semigroup generation, and intertwining.
We keep the dissipativity discussion first because it is the main algebraic reason for choosing the symmetric regularization $A_\varepsilon=\mathcal{H}_\varepsilon^*A\mathcal{H}_\varepsilon$.
The summary Proposition \ref{prop_Aeps_hypotheses_PDE} at the end of this subsection gathers what is automatic and what remains structural.

\paragraph{Dissipativity and the semigroup issue.}

The next lemma is the key algebraic reason for choosing the symmetric regularization $A_\varepsilon=\mathcal{H}_\varepsilon^* A \mathcal{H}_\varepsilon$.

\begin{lemma}\label{lem_Aeps_dissip}
Assume that, for every $t\in[0,T]$ and every $z\in B_Z(y^0,r)$, the operator $A[t,z]-\omega\, \mathrm{id}$ is dissipative on $X=L^2(\Omega,\R^d)$, for some $\omega\in\R$ independent of $(t,z)$.
Then, for every $\varepsilon\in(0,\varepsilon_0]$, every $t\in[0,T]$, every $z\in B_Z(y^0,r)$, and every $y\in Z$,
\begin{equation}\label{Aeps_quasidissipative}
\langle A_\varepsilon[t,z]y,y\rangle_{L^2}\leq \omega_0 \Vert y\Vert_{L^2}^2,
\qquad
\omega_0=\max(0,\omega) \sup_{\varepsilon\in(0,\varepsilon_0]} \Vert \mathcal{H}_\varepsilon\Vert_{L(L^2)}^2.
\end{equation}
In particular, $A_\varepsilon[t,z]-\omega_0\, \mathrm{id}$ is dissipative on $X$, uniformly with respect to $(\varepsilon,t,z)$.
\end{lemma}

\begin{proof}
Let $y\in Z=H^s(\Omega,\R^d)\cap D$. Since $\mathcal{H}_\varepsilon y\in D$, we may write
$$
\langle A_\varepsilon[t,z]y,y\rangle_{L^2} = \langle \mathcal{H}_\varepsilon^* A[t,z]\mathcal{H}_\varepsilon y,y\rangle_{L^2}
= \langle A[t,z]\mathcal{H}_\varepsilon y,\mathcal{H}_\varepsilon y\rangle_{L^2}
\leq \omega \Vert \mathcal{H}_\varepsilon y\Vert_{L^2}^2 ,
$$
where the last inequality follows from the dissipativity of $A[t,z]-\omega\, \mathrm{id}$.
If $\omega\leq 0$, then $\omega \Vert \mathcal{H}_\varepsilon y\Vert_{L^2}^2\leq 0\leq\omega_0\Vert y\Vert_{L^2}^2$. If $\omega>0$, then, by Lemma \ref{lem_Lr-bounds} of Appendix \ref{app_convolution}, $\Vert \mathcal{H}_\varepsilon\Vert_{L(L^2)}$ is uniformly bounded in $\varepsilon$ and $\omega \Vert \mathcal{H}_\varepsilon y\Vert_{L^2}^2\leq\omega \Vert \mathcal{H}_\varepsilon\Vert_{L(L^2)}^2\Vert y\Vert_{L^2}^2\leq\omega_0\Vert y\Vert_{L^2}^2$. In both cases $\langle A_\varepsilon[t,z]y,y\rangle_{L^2}\leq\omega_0\Vert y\Vert_{L^2}^2$, which gives the result.
\end{proof}

\begin{remark}\label{rem_Aeps_semigroup_issue}
If, for any fixed $(t,z)$, the regularized operator $A_\varepsilon[t,z]$ extends to a bounded operator on $X$, then the semigroup-generation part of Assumption \ref{H_Aepsilon} is automatic: one can define $e^{sA_\varepsilon[t,z]}$ by the convergent exponential series in $L(X)$. Moreover, in the Hilbert setting, the quasi-dissipativity estimate \eqref{Aeps_quasidissipative} implies the stability bound $\Vert e^{sA_\varepsilon[t,z]}\Vert_{L(X)}\leq e^{\omega_0 s}$ by the standard energy method applied to $v(s)=e^{sA_\varepsilon[t,z]}y$.

The only nontrivial situation is when $A_\varepsilon[t,z]$ is realized as an unbounded operator on $X$ with domain $D(A_\varepsilon[t,z])\subsetneq X$. In that case, dissipativity alone does not guarantee semigroup generation: one must verify maximal dissipativity, i.e., for some $\lambda>\omega$ the resolvent equation $(\lambda\mathrm{id}-A_\varepsilon[t,z])u=f$ is solvable for every $f\in X$ (equivalently, $A_\varepsilon[t,z]$ admits no proper dissipative extension). A simple counterexample is obtained by restricting the domain of a generator: on $X=L^2(0,1)$, the transport operator $u\mapsto -u'$ with domain $\{u\in H^1(0,1)\mid u(0)=0\}$ generates the right-shift semigroup, while adding the extra constraint $u(1)=0$ yields a dissipative operator that is not maximal and therefore does not generate a $C_0$ semigroup. Finally, if $A[t,z]$ is maximal dissipative on $X$, then its Yosida approximants $A_\varepsilon[t,z]=A[t,z](\mathrm{id}-\varepsilon A[t,z])^{-1}$ are bounded and maximal dissipative, hence generate stable semigroups without further work (see \cite{EngelNagel, Pazy}).
\end{remark}

\paragraph{Verification of \ref{H_regZ}.}

\begin{lemma}\label{lem_Aeps_regZ_PDE}
Assumption \ref{H_regZ} holds for $A_\varepsilon$ uniformly with respect to $\varepsilon$ (taking larger constants $C_4$ and $C'_4$ if necessary).
\end{lemma}

\begin{proof}
By Lemma \ref{lem_Heps_preserves_D}, one has $\mathcal{H}_\varepsilon y\in Z$ and
$\Vert \mathcal{H}_\varepsilon y\Vert_{H^s}\leq C_H^Z \Vert y\Vert_{H^s}$.
Using the $L^2$-boundedness of $\mathcal{H}_\varepsilon^*$ and Assumption \ref{H_regZ} for the original family $A[t,z]$, we obtain
$$
\Vert A_\varepsilon[t,z]y\Vert_{L^2}
= \Vert \mathcal{H}_\varepsilon^* \,A[t,z]\,\mathcal{H}_\varepsilon y\Vert_{L^2}
\leq \Vert \mathcal{H}_\varepsilon^*\Vert_{L(X)} C_4 C_H^Z \Vert y\Vert_{H^s}.
$$
Likewise,
$$
\Vert \big(A_\varepsilon[t,z_1]-A_\varepsilon[t,z_2]\big)y\Vert_{L^2}
= \Vert \mathcal{H}_\varepsilon^* \big(A[t,z_1]-A[t,z_2]\big)\mathcal{H}_\varepsilon y\Vert_{L^2}
\leq \Vert \mathcal{H}_\varepsilon^*\Vert_{L(L^2)} C_4' C_H^Z \Vert z_1-z_2\Vert_X \Vert y\Vert_{H^s}.
$$
Since $\Vert \mathcal{H}_\varepsilon^*\Vert_{L(L^2)}$ is uniformly bounded on $(0,\varepsilon_0]$, this proves the lemma.
\end{proof}

\paragraph{Verification of \ref{H_Y}.} We choose
$$
Y=W^{1,\infty}(\Omega,\R^d)\cap L^2(\Omega,\R^d),
$$
endowed with the norm $\Vert g\Vert_Y=\Vert g\Vert_{L^2}+\Vert g\Vert_{W^{1,\infty}}$.

\begin{lemma}\label{lem_Aepsy_Lip_PDE}
There exists $L_0>0$, independent of $\varepsilon\in(0,\varepsilon_0]$, such that
\begin{equation*}
\Vert A_\varepsilon[t,z]z+f[t,z]\Vert_Y\leq L_0
\qquad
\forall (t,z)\in[0,T]\times B_Z(y^0,r).
\end{equation*}
In particular, Assumption \ref{H_Y} holds for the choice of output space $Y$ above, with $L_\varepsilon\leq L_0$.
\end{lemma}

\begin{proof}
Since $Z\hookrightarrow W^{p+1,\infty}(\Omega,\R^d)$, there exists $C_{Z,\infty}>0$ such that
$\Vert z\Vert_{W^{p+1,\infty}}\leq C_{Z,\infty}(\Vert y^0\Vert_Z+r)$ for every $z\in B_Z(y^0,r)$.
Applying Corollary \ref{H_eps_Wmr} of Appendix \ref{app_convolution} with the integer $s$ in place of the appendix integer $p$, we infer that
$\Vert \mathcal{H}_\varepsilon z\Vert_{W^{p+1,\infty}} \leq \mathrm{Cst} \Vert z\Vert_{W^{p+1,\infty}} \leq \mathrm{Cst}$.
Now let $u_\varepsilon[t,z]=A[t,z](\mathcal{H}_\varepsilon z)$.
Because the coefficients $a_\alpha[t,z]$ belong to $W^{1,\infty}(\Omega)$ uniformly by \eqref{coeff_W1inf_PDE}, and because $\mathcal{H}_\varepsilon z\in W^{p+1,\infty}(\Omega)$ uniformly, the differential expression \eqref{def_A_PDE} gives
$\Vert u_\varepsilon[t,z]\Vert_{W^{1,\infty}}\leq \mathrm{Cst}$,
where $\mathrm{Cst}$ depends on $C_a$, on $p$, and on the $W^{p+1,\infty}$-bound for $\mathcal{H}_\varepsilon z$, but not on $\varepsilon$, $t$, or $z$.
Moreover, by Lemma \ref{lem_Aeps_regZ_PDE},
$\Vert A_\varepsilon[t,z]z\Vert_{L^2}\leq \mathrm{Cst} \Vert z\Vert_{H^s}\leq \mathrm{Cst}$.
Lemma \ref{lem_adjoint_Lip} in Appendix \ref{app_convolution}, together with the $L^2$-boundedness of $\mathcal{H}_\varepsilon^*$, therefore yields
$$
\Vert A_\varepsilon[t,z]z\Vert_Y
= \Vert \mathcal{H}_\varepsilon^* u_\varepsilon[t,z]\Vert_{L^2} + \Vert \mathcal{H}_\varepsilon^* u_\varepsilon[t,z]\Vert_{W^{1,\infty}}
\leq \mathrm{Cst} \Vert u_\varepsilon[t,z]\Vert_{L^2} + \mathrm{Cst} \Vert u_\varepsilon[t,z]\Vert_{W^{1,\infty}} 
\leq \mathrm{Cst}.
$$
Finally, since $f[t,z]\in Z$ by Assumption \ref{H_f_Lip} and $Z\hookrightarrow Y$, one has
$\Vert f[t,z]\Vert_Y\leq \mathrm{Cst} \Vert f[t,z]\Vert_Z\leq \mathrm{Cst}$.
The lemma follows.
\end{proof}

\paragraph{Verification of \ref{H_Aepsilon_CV}.}

\begin{lemma}\label{lem_CV_Aeps}
There exists $C_{A,\eta,\rho}>0$, independent of $\varepsilon\in(0,\varepsilon_0]$, such that
\begin{equation}\label{Aeps_rate_PDE}
\Vert \big(A_\varepsilon[t,z]-A[t,z]\big)y\Vert_{L^2}
\leq C_{A,\eta,\rho}\,\varepsilon \Vert y\Vert_{H^s}
\end{equation}
for all $t\in[0,T]$, $z\in B_Z(y^0,r)$, $y\in Z$ and $\varepsilon\in(0,\varepsilon_0]$.
Consequently, Assumption \ref{H_Aepsilon_CV} holds with
$\chi(\varepsilon)=C_{A,\eta,\rho}\,\varepsilon$.
\end{lemma}

\begin{proof}
Fix $(t,z)\in[0,T]\times B_Z(y^0,r)$ and $y\in Z$.
We decompose
\begin{equation}\label{Aeps_decomp_PDE}
\big(A_\varepsilon[t,z]-A[t,z]\big)y
=
\mathcal{H}_\varepsilon^*\big(A[t,z]\mathcal{H}_\varepsilon y-A[t,z]y\big)
+
\big(\mathcal{H}_\varepsilon^*-\mathrm{id}_X\big)A[t,z]y.
\end{equation}

We first estimate the term $A[t,z]\mathcal{H}_\varepsilon y-A[t,z]y$.
Using \eqref{def_A_PDE},
\begin{align*}
A[t,z]\mathcal{H}_\varepsilon y-A[t,z]y
&=
\sum_{\vert\alpha\vert\leq p} a_\alpha[t,z]
\big(D^\alpha \mathcal{H}_\varepsilon y-D^\alpha y\big).
\end{align*}
Since $y\in Z\hookrightarrow H^{p+1}(\Omega,\R^d)$, Lemmas \ref{lem_CV_deriv_Heps} (for $1\leq\vert\alpha\vert\leq p$) and \ref{lem_CV_Heps} (for $\alpha=0$) of Appendix \ref{app_convolution} give, for every $\vert\alpha\vert\leq p$,
$$
\Vert D^\alpha \mathcal{H}_\varepsilon y-D^\alpha y\Vert_{L^2}
\leq C \varepsilon \Vert y\Vert_{H^{p+1}(\Omega)}
\leq C \varepsilon \Vert y\Vert_{H^s} .
$$
Hence, using the uniform $L^\infty$-bound of the coefficients,
\begin{equation}\label{AH_minus_A_PDE}
\Vert A[t,z]\mathcal{H}_\varepsilon y-A[t,z]y\Vert_{L^2}
\leq C \varepsilon \Vert y\Vert_{H^s} .
\end{equation}

We next estimate the adjoint defect.
Because $y\in Z\hookrightarrow H^{p+1}(\Omega,\R^d)$ and the coefficients belong to $W^{1,\infty}(\Omega)$ uniformly by \eqref{coeff_W1inf_PDE}, the function $A[t,z]y$ belongs to $W^{1,2}(\Omega,\R^d)$ and
$\Vert A[t,z]y\Vert_{W^{1,2}}\leq \mathrm{Cst} \Vert y\Vert_{H^{p+1}}\leq \mathrm{Cst} \Vert y\Vert_{H^s}$.
Applying Lemma \ref{lem_CV_Hepsstar} of Appendix \ref{app_convolution} to $A[t,z]y$, we infer that
\begin{equation}\label{Hstar_minus_id_PDE}
\Vert \big(\mathcal{H}_\varepsilon^*-\mathrm{id}_X\big)A[t,z]y\Vert_{L^2}
\leq C \varepsilon \Vert y\Vert_{H^s} .
\end{equation}
Finally, combining \eqref{Aeps_decomp_PDE}, \eqref{AH_minus_A_PDE}, \eqref{Hstar_minus_id_PDE}, and the uniform $L^2$-boundedness of $\mathcal{H}_\varepsilon^*$ yields \eqref{Aeps_rate_PDE}.
Since $f_\varepsilon=f$, the second part of Assumption \ref{H_Aepsilon_CV} is automatic.
\end{proof}

\begin{proposition}\label{prop_Aeps_hypotheses_PDE}
Under the standing assumptions of Section \ref{sec_general_class_quasilinear}, the explicit regularized family
$(A_\varepsilon,f_\varepsilon)$ defined by \eqref{def_Aeps_Heps}
verifies, uniformly with respect to $\varepsilon\in(0,\varepsilon_0]$:
\begin{itemize}[parsep=0.7mm, itemsep=0.7mm, topsep=0.7mm]
\item Assumption \ref{H_regZ}, by Lemma \ref{lem_Aeps_regZ_PDE};
\item Assumption \ref{H_Y}, for the choice $Y=W^{1,\infty}(\Omega,\R^d)\cap L^2(\Omega,\R^d)$, by Lemma \ref{lem_Aepsy_Lip_PDE};
\item Assumption \ref{H_Aepsilon_CV}, with $\chi(\varepsilon)=C_{A,\eta,\rho}\,\varepsilon$, by Lemma \ref{lem_CV_Aeps};
\item the dissipativity part of \ref{H_stab}, by Lemma \ref{lem_Aeps_dissip}.
\end{itemize}
Therefore, if the semigroup-generation part of Assumption \ref{H_Aepsilon} is additionally known for the chosen realization of $A_\varepsilon[t,z]$ on $X$, then Assumption \ref{H_Aepsilon} is fully verified, except possibly the intertwining condition \ref{H_intertwining}.
\end{proposition}

\begin{remark}\label{rem_intertwining_continuous_not_generic}
The fact that intertwining for $A_\varepsilon$ is not automatic does \emph{not} mean that no intertwining relation can exist for the regularized family.
In principle, one could very well look for a family of operators $S_\varepsilon$ depending on $\varepsilon$ such that $S_\varepsilon A_\varepsilon[t,z] = A_\varepsilon[t,z] S_\varepsilon + B_\varepsilon[t,z] S_\varepsilon$.
What fails in general is the existence of a canonical and uniform construction of such a family from the single operator $S$ used for the original family $A[t,z]$.

Allowing $S_\varepsilon$ would amount to working with $\varepsilon$-dependent graph norms, or even with an $\varepsilon$-dependent Banach scale. This is perfectly conceivable in special situations, in particular for first-order systems or for regularizations tailored to a problem-specific energy. But it lies outside the fixed-scale framework adopted in the present paper, where the same pair $(Z,X)$ is used to compare $y$, $y_\varepsilon$, and $y_\varepsilon^N$.

For the variable-step mollifier regularization considered here, no such family $(S_\varepsilon)_\varepsilon$ is available in general from the construction alone. Establishing it would require a separate problem-dependent analysis. This is why \ref{H_intertwining} remains a structural hypothesis for the continuous regularized family.
\end{remark}

\section{Discretization and interacting approximation systems}\label{sec_abstract_finite_approx}

We now discretize the regularized problem \eqref{abstract_quasilinear_eps}.
The guiding principle is to mirror the continuous pair $(Z,X)$ at the discrete level.
The discrete unknown lives in a finite-dimensional vector space $V_N$, endowed with two norms induced by the same reconstruction operator:
a strong norm $\Vert\cdot\Vert_{Z_N}$ and a pivot norm $\Vert\cdot\Vert_{X_N}$.

\subsection{Discretization assumptions}

We make the following discretization assumptions:
\begin{enumerate}[label=$(H_{\arabic*})$, resume] 
\setcounter{enumi}{9}
%
%
\item\label{H_discretization}
For each $N\in\N^*$, let $V_N$ be a finite-dimensional real vector space (typically, $V_N\simeq\R^{dN}$).
We assume that there exist linear maps $P_N\in L(X,V_N)$ and $R_N\in L(V_N,Z)$ (\emph{a fortiori}, we have $P_N\in L(Z,V_N)$ for the restriction to $Z$, and $R_N\in L(V_N,X)$, because $Z\hookrightarrow X$) satisfying
$$
P_NR_N=\mathrm{id}_{V_N} .
$$
Hence the map $Q_N=R_NP_N$, viewed either as an endomorphism of $X$ or of $Z$, is a projection (we also have $Q_N\in L(X,Z)$).
We assume that there exist $C_{\mathrm{stab}}>0$, $C_{10}>0$, and $\gamma>0$, not depending on $N$, such that
\begin{align}
\Vert Q_N\Vert_{L(X)} \leq C_{\mathrm{stab}}, \qquad \Vert Q_N\Vert_{L(Z)} \leq C_{\mathrm{stab}}, \label{stab_PX_PZ} \\
\Vert Q_Ny-y\Vert_X \leq \frac{C_{10}}{N^\gamma} \Vert y\Vert_Z \qquad\forall y\in Z , \label{CV_X_Z} \\
\Vert Q_Ny-y\Vert_X \leq \frac{C_{10}}{N^\gamma} \Vert y\Vert_Y \qquad\forall y\in Y , \label{CV_X_Y} \\
\lim_{N\to+\infty} \Vert Q_Ny-y\Vert_Z = 0 \qquad\forall y\in Z . \label{CV_Z}
\end{align}
%
%
%
%
\end{enumerate}
We endow the same underlying set $V_N$ with two norms, defining:
\begin{itemize}
\item $X_N=(V_N,\Vert\cdot\Vert_{X_N})$ with the induced norm $\Vert u\Vert_{X_N}=\Vert R_Nu\Vert_X$ for any $u\in V_N$;
\item $Z_N=(V_N,\Vert\cdot\Vert_{Z_N})$ with the induced norm $\Vert u\Vert_{Z_N}=\Vert R_Nu\Vert_Z$ for any $u\in V_N$ .
\end{itemize}
By construction,
$$
\Vert R_N\Vert_{L(X_N,X)}=\Vert R_N\Vert_{L(Z_N,Z)}=1,
\quad
\Vert Q_N\Vert_{L(X)} = \Vert P_N\Vert_{L(X,X_N)},
\quad \Vert Q_N\Vert_{L(Z)} = \Vert P_N\Vert_{L(Z,Z_N)}.
$$
In particular, \eqref{stab_PX_PZ} is the uniform boundedness of the sampling operator on the two discrete scales.

Assumption \ref{H_discretization} mirrors the continuous pair $(Z,X)$ by a discrete pair $(Z_N,X_N)$ carried by the same finite-dimensional space $V_N$.
The discrete $Z_N$-norm controls the strong scale 
while the discrete $X_N$-norm plays the role of the pivot energy norm.

The above sampling-reconstruction viewpoint appears in several neighboring traditions, although not always under the same packaging: finite element quasi-interpolation, Scott-Zhang and Cl\'ement-type operators, smooth finite-volume reconstructions, partition-of-unity and meshfree methods, as well as spectral or orthogonal Galerkin truncations.
We refer to  Section \ref{sec_examples_discretizations} and Appendix \ref{app_discretization_technical} for a precise discussion and representative examples.

Assumption \ref{H_discretization} is primarily a \emph{spatial} discretization hypothesis.
The projector $Q_N=R_NP_N$ acts on the state space and approximates the identity on regular classes of spatial outputs.
Time discretizations are of a different nature: they are usually formulated as one-step or multistage maps on trajectories or stage variables, and their analysis requires stability properties specific to the chosen integrator.
For Kato-type quasilinear equations, see in particular
\cite{DorichHochbruck_SINUM2022, Dorich_FoCM2025, HochbruckPazur_NM2017, KovacsLubich_NM2018, HochbruckPazurSchnaubelt_NM2018}.
One may of course combine such time integrators with the present spatial/interacting approximation, but this lies beyond the scope of the paper.

\begin{remark}\label{rem_H_Z_N}
It follows from \ref{H_discretization} that the discrete embedding $Z_N\hookrightarrow X_N$ holds with the same constant as in \ref{H_Z}: $\Vert u\Vert_{X_N} = \Vert R_Nu\Vert_X \leq C_1 \Vert R_Nu\Vert_Z = C_1 \Vert u\Vert_{Z_N}$ for any $u\in V_N$.
\end{remark}

\begin{remark}\label{rem_no_uniform_X_to_Z}
For each fixed $N$, the operator $Q_N$ belongs to $L(X,Z)$.
However, one should not expect a uniform bound of $\Vert Q_N\Vert_{L(X,Z)}$ as $N\to+\infty$ when $Z\subsetneq X$.
This is an inverse-inequality phenomenon and is one of the reasons why the abstract error estimate below is formulated through the output approximation assumption \eqref{CV_X_Y}, and not through a uniform $X\to Z$ bound on $Q_N$.
\end{remark}

\subsection{Finite-dimensional approximation system}

Given any $\varepsilon\in(0,\varepsilon_0]$, any $N\in\N^*$ and any $(t,v)\in[0,T]\times Z_N$, we define the operator $A_\varepsilon^N[t,v]\in L(Z_N,X_N)$ and $f_\varepsilon^N[t,v]\in V_N$ by
\begin{equation}\label{def_A_f_eps_N}
A_\varepsilon^N[t,v] = P_N A_\varepsilon[t, R_N v] R_N ,
\qquad
f_\varepsilon^N[t,v] = P_N f_\varepsilon[t,R_Nv] ,
\end{equation}
as illustrated on Figure \ref{fig_def_A_eps_N}.
%
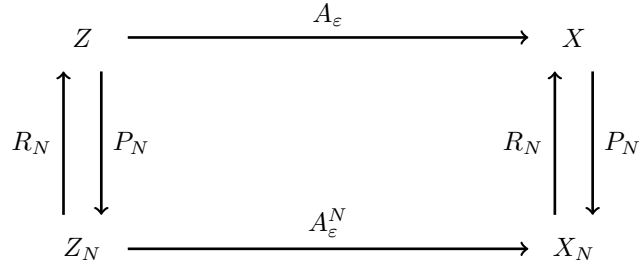
\begin{figure}[h]
\begin{center}
\begin{tikzpicture}[line width=1pt, every node/.style={inner sep=1pt}]
\node (Z)  at (0.5,2.8) {$Z$};
\node (X)  at (7.,2.8) {$X$};
\node (ZN) at (0.5,0)   {$Z_{N}$};
\node (XN) at (7,0) {$X_{N}$};
\draw[->] (1.1,2.8) -- (6.4,2.8)
  node[midway, above=0.1cm] {$A_\varepsilon$};
\draw[->] (1.1,0) -- (6.4,0)
  node[midway, above=0.1cm] {$A_\varepsilon^N$};
\draw[->] (0.25,0.45) -- (0.25,2.35)
  node[midway, left=0.1cm] {$R_N$};
\draw[->] (0.75,2.35) -- (0.75,0.45)
  node[midway, right=0.1cm] {$P_N$};
\draw[->] (6.75,0.45) -- (6.75,2.35)
  node[midway, left=0.1cm] {$R_N$};
\draw[->] (7.25,2.35) -- (7.25,0.45)
  node[midway, right=0.1cm] {$P_N$};
\end{tikzpicture}
\end{center}
\caption{Discretization diagram}\label{fig_def_A_eps_N}
\end{figure}

The discrete system on $V_N$ is
\begin{equation}\label{interacting_system_u}
\boxed{
\dot u_\varepsilon^N(t)=A_\varepsilon^N[t,u_\varepsilon^N(t)]u_\varepsilon^N(t)+f_\varepsilon^N[t,u_\varepsilon^N(t)]
}
\end{equation}
with initial condition $u_\varepsilon^N(0)=P_N y^0$.
Its lift in $Z$ is
\begin{equation}\label{def_yepsN}
\boxed{
y_\varepsilon^N(t)=R_Nu_\varepsilon^N(t)\in Z
}
\end{equation}
and satisfies
\begin{equation}\label{interacting_system_y}
\dot y_\varepsilon^N(t) = Q_N A_\varepsilon[t,y_\varepsilon^N(t)] y_\varepsilon^N(t) + Q_N f_\varepsilon[t,y_\varepsilon^N(t)]
\end{equation}
with initial condition $y_\varepsilon^N(0) = Q_N y^0$.
In particular, $y_\varepsilon^N(t)$ takes values in the finite-dimensional subspace $\mathrm{Ran}(R_N)\subset Z$.

\medskip

In addition to Assumptions \ref{H_Z} to \ref{H_discretization}, similarly to \ref{H_Aepsilon}, 
we assume that:
\begin{enumerate}[label=$(H_{\arabic*})$,resume] 
\setcounter{enumi}{10}
\item\label{H_Aepsilon_N} 
There exists $N_0\in\N^*$ such that, for all $t\in[0,T]$ and $v\in B_{Z_N}(P_Ny^0,r/2)$, the families of operators $A_\varepsilon^N[t,v]$ and of functions $f_\varepsilon^N[t,v]$ satisfy Assumptions \ref{H_Z} to \ref{H_f_Lip} (with $X$, $Z$, $S$, $A$, $B$, $B_Z(y^0,r)$ replaced by $X_N$, $Z_N$, $S_\varepsilon^N$, $A_\varepsilon^N$, $B_\varepsilon^N$, $B_{Z_N}(P_Ny^0,r/2)$) uniformly with respect to $N\geq N_0$ and $\varepsilon\in(0,\varepsilon_0]$, i.e., with constants $C_i,M,\omega$ that may be larger if necessary but do not depend on $\varepsilon$ and $N$.
\end{enumerate}


\begin{proposition}\label{prop_yepsilonN}
We make Assumptions \ref{H_Z} to \ref{H_Aepsilon_N}.
Given any $\varepsilon\in(0,\varepsilon_0]$, let $y_\varepsilon(\cdot)\in\mathscr{C}^0([0,T'],Z)\cap\mathscr{C}^1([0,T'],X)$ be the unique solution of \eqref{abstract_quasilinear_eps} such that $y_\varepsilon(0)=y^0$, as given by Proposition \ref{prop_yepsilon}.
Taking $N_0$ larger if necessary, for any $N\geq N_0$, there exists a unique solution $y_\varepsilon^N\in \mathscr{C}^0([0,T'],Z)\cap\mathscr{C}^1([0,T'],X)$ of \eqref{interacting_system_y} such that $y_\varepsilon^N(0)=Q_Ny^0$ (equivalently, a unique solution $u_\varepsilon^N\in \mathscr{C}^0([0,T'],Z_N)\cap\mathscr{C}^1([0,T'],X_N)$ of \eqref{interacting_system_u} such that $u_\varepsilon^N(0)=P_N y^0$, satisfying \eqref{def_yepsN}). Moreover, $y_\varepsilon^N(t)\in B_Z(y^0,r)$ for any $t\in[0,T']$ and 
\begin{equation}\label{CV_yeps_yepsN}
\Vert y_\varepsilon^N(t)-y_\varepsilon(t)\Vert_X \leq  \frac{b_\varepsilon(t)}{N^\gamma} 
\qquad\forall t\in[0,T'] \qquad\forall\varepsilon\in(0,\varepsilon_0],
\end{equation}
where 
\begin{equation}\label{def_beps}
b_\varepsilon(t) = M C_{10} e^{\beta t} \left( \Vert y^0\Vert_Z e^{\omega t} + L_\varepsilon \int_0^t e^{\omega s} \, ds \right) .
\end{equation}
\end{proposition}

\begin{proof}
By \eqref{CV_Z}, there exists $N_0\in\N^*$ such that $\Vert Q_N y^0-y^0\Vert_Z\leq \frac{r}{2}$ for any $N\geq N_0$.
Since Assumption \ref{H_Aepsilon_N} holds on the pair $(X_N,Z_N)$ with initial datum $P_N y^0$, Proposition \ref{prop_existence_uniqueness} and Remarks \ref{rem_prop_existence_uniqueness1} and \ref{rem_prop_existence_uniqueness_T'uniform} yield a unique solution
$u_\varepsilon^N\in \mathscr{C}^0([0,T'],Z_N)\cap \mathscr{C}^1([0,T'],X_N)$ of \eqref{interacting_system_u}, satisfying $u_\varepsilon^N(t)\in B_{Z_N}(P_N y^0,r/2)$ for every $t\in[0,T']$.
Since $R_N\in L(Z_N,Z)\cap L(X_N,X)$, we infer that $y_\varepsilon^N=R_Nu_\varepsilon^N \in \mathscr{C}^0([0,T'],Z)\cap \mathscr{C}^1([0,T'],X)$.
Moreover,
\begin{multline*}
\Vert y_\varepsilon^N(t)-y^0\Vert_Z \leq \Vert y_\varepsilon^N(t)-Q_N y^0\Vert_Z + \Vert Q_N y^0-y^0\Vert_Z \\
= \Vert u_\varepsilon^N(t)-P_N y^0\Vert_{Z_N} + \Vert Q_N y^0-y^0\Vert_Z 
\leq \frac{r}{2}+\frac{r}{2} = r,
\end{multline*}
hence $y_\varepsilon^N(t)\in B_Z(y^0,r)$ for every $t\in[0,T']$.
This gives the first part of the proposition. 

Let us prove \eqref{CV_yeps_yepsN}. The argument is similar to the one in the proof of Proposition \ref{prop_yepsilon}, but since there are some differences we give the complete detail.
Since $y_\varepsilon^N$ solves \eqref{interacting_system_y} and $y_\varepsilon$ solves \eqref{abstract_quasilinear_eps}, we have
\begin{multline*}
\frac{d}{dt} ( y_\varepsilon^N(t)-y_\varepsilon(t) ) = A_\varepsilon[t,y_\varepsilon^N(t)] ( y_\varepsilon^N(t)-y_\varepsilon(t) ) +  (A_\varepsilon[t,y_\varepsilon^N(t)]-A_\varepsilon[t,y_\varepsilon(t)]) y_\varepsilon(t)
\\
+ f_\varepsilon[t,y_\varepsilon^N(t)]-f_\varepsilon[t,y_\varepsilon(t)] + (Q_N-\mathrm{id}_X) \big(A_\varepsilon[t,y_\varepsilon^N(t)]y_\varepsilon^N(t)+f_\varepsilon[t,y_\varepsilon^N(t)]\big) .
\end{multline*}
Applying the Duhamel formula \eqref{duhamel_implicit} with the evolution system $U^\varepsilon_{y_\varepsilon^N}$ associated with
$t\mapsto A_\varepsilon[t,y_\varepsilon^N(t)]$, we obtain
\begin{multline*}
y_\varepsilon^N(t)-y_\varepsilon(t) = U^\varepsilon_{y_\varepsilon^N}(t,0)\big(Q_N y^0-y^0\big)
\\
+ \int_0^t U^\varepsilon_{y_\varepsilon^N}(t,s) \Big( \big(A_\varepsilon[s,y_\varepsilon^N(s)]-A_\varepsilon[s,y_\varepsilon(s)]\big)y_\varepsilon(s) + f_\varepsilon[s,y_\varepsilon^N(s)]-f_\varepsilon[s,y_\varepsilon(s)] \\
+ (Q_N-\mathrm{id}_X) \big(A_\varepsilon[s,y_\varepsilon^N(s)]y_\varepsilon^N(s)+f_\varepsilon[s,y_\varepsilon^N(s)]\big) \Big)\,ds.
\end{multline*}
Since both $y_\varepsilon$ and $y_\varepsilon^N$ take values in $B_Z(y^0,r)$ on $[0,T']$, we may use the same constants as in Proposition \ref{prop_yepsilon}.
Hence
\begin{multline*}
\Vert y_\varepsilon^N(t)-y_\varepsilon(t)\Vert_X \leq M e^{\omega t}\Vert Q_N y^0-y^0\Vert_X 
+ \int_0^t e^{\omega(t-s)} \Big( \beta \Vert y_\varepsilon^N(s)-y_\varepsilon(s)\Vert_X \\
+ M \big\Vert (Q_N-\mathrm{id}_X) \big(A_\varepsilon[s,y_\varepsilon^N(s)]y_\varepsilon^N(s)+f_\varepsilon[s,y_\varepsilon^N(s)]\big) \big\Vert_X \Big) \, ds.
\end{multline*}
By \eqref{CV_X_Z}, $\Vert Q_N y^0-y^0\Vert_X\leq \frac{C_{10}}{N^\gamma}\Vert y^0\Vert_Z$, 
and by \eqref{CV_X_Y} and Assumption \ref{H_Y},
$$
\big\Vert (Q_N-\mathrm{id}_X) \big(A_\varepsilon[s,y_\varepsilon^N(s)]y_\varepsilon^N(s) + f_\varepsilon[s,y_\varepsilon^N(s)]\big) \big\Vert_X \leq \frac{C_{10}}{N^\gamma} L_\varepsilon .
$$
As in the proof of Proposition \ref{prop_yepsilon}, applying the Gronwall lemma to $e^{-\omega t}\Vert y_\varepsilon^N(t)-y_\varepsilon(t)\Vert_X$ yields \eqref{CV_yeps_yepsN}. 
\end{proof}

\begin{remark}\label{rem_prop_yepsilonN_scope}
Proposition \ref{prop_yepsilonN} gives an abstract criterion for the existence of a lifted interacting solution, under the (strong) sufficient Assumption \ref{H_Aepsilon_N}.
In applications, one may prove the existence of $y_\varepsilon^N$ by other, more direct arguments; this is important to note because, when this is the case then one obtains the estimate \eqref{CV_yeps_yepsN} without Assumption \ref{H_Aepsilon_N}: deriving this estimate only requires the existence of such a lifted solution on $[0,T']$.
\end{remark}

\paragraph{Comments on Assumption \ref{H_Aepsilon_N}.}
Within Assumption \ref{H_Aepsilon_N}, the genuinely nontrivial parts are the existence of a suitable discrete intertwining operator, the semigroup stability, and the intertwining relation.
By contrast, the discrete analogues of the regularity and forcing estimates follow directly from the discretization assumptions, as shown by the next lemma.

\begin{lemma}\label{lem_H_regZ_N}
Under Assumptions \ref{H_Z} to \ref{H_discretization}, there exists $N_0\in\N^*$ such that, for all $t\in[0,T]$ and $v\in B_{Z_N}(P_Ny^0,r/2)$, the families of operators $A_\varepsilon^N[t,v]$ and of functions $f_\varepsilon^N[t,v]$ satisfy Assumptions \ref{H_regZ} 
and \ref{H_f_Lip} on $(X_N,Z_N)$ (with constants $C_{\mathrm{stab}} C_4$, $C_{\mathrm{stab}} C'_4$, $C_{\mathrm{stab}} C_6$ and $C_{\mathrm{stab}} C'_6$, respectively) uniformly with respect to $N\geq N_0$ and $\varepsilon\in(0,\varepsilon_0]$
\end{lemma}

\begin{proof}
Recalling that $\Vert Q_N y^0-y^0\Vert_Z\leq\frac{r}{2}$ for any $N\geq N_0$,
for any $v\in B_{Z_N}(P_Ny^0,r/2)$, we have 
$$
\Vert R_Nv-y^0\Vert_Z \leq \Vert R_N(v-P_N y^0)\Vert_Z + \Vert Q_N y^0-y^0\Vert_Z
= \Vert v-P_N y^0\Vert_{Z_N} + \Vert Q_N y^0-y^0\Vert_Z \leq r,
$$
and thus $R_Nv\in B_Z(y^0,r)$.
Then, using \eqref{stab_PX_PZ} and Assumption \ref{H_Aepsilon} (more precisely, \ref{H_regZ} for $A_\varepsilon$),
\begin{multline*}
\Vert A_\varepsilon^N[t,v]u\Vert_{X_N} 
= \Vert Q_N A_\varepsilon[t,R_Nv]R_Nu\Vert_X
\leq C_{\mathrm{stab}} \Vert A_\varepsilon[t,R_Nv]R_Nu\Vert_X \\
\leq C_{\mathrm{stab}} C_4 \Vert R_Nu\Vert_Z = C_{\mathrm{stab}} C_4 \Vert u\Vert_{Z_N},
\end{multline*}
and the Lipschitz property in $v$ is obtained similarly, which proves \ref{H_regZ} for $A_\varepsilon^N$ with constant $C_{\mathrm{stab}} C'_4$. Regarding \ref{H_f_Lip}, the Lipschitzness for $f_\varepsilon^N$ is proved in the same way; 
for the $Z$-bound of $f_\varepsilon^N$, we proceed similarly, using that $\Vert Q_N\Vert_{L(Z)} \leq C_{\mathrm{stab}}$.
\end{proof}

The following two lemmas give sufficient conditions under which the remaining assumptions \ref{H_stab} and \ref{H_intertwining} transfer to the discrete level.

\begin{lemma}\label{lem_discrete_stab}
Under Assumptions \ref{H_Z} to \ref{H_discretization}, assume moreover that $X$ is a Hilbert space, that $Q_N$ is the orthogonal projector of $X$ onto $\mathrm{Ran}(R_N)$ for every $N$, and that $A_\varepsilon[t,z]$ is uniformly $\omega$-dissipative in $X$, i.e., $\langle A_\varepsilon[t,z]y,y\rangle_X\leq \omega \Vert y\Vert_X^2$ for all $y\in Z$, $t\in[0,T]$,  $z\in B_Z(y^0,r)$ and $\varepsilon\in(0,\varepsilon_0]$.
Then $A_\varepsilon^N[t,v]$ is uniformly $\omega$-dissipative in $X_N$, i.e., 
$\langle A_\varepsilon^N[t,v]u,u\rangle_{X_N} \leq \omega \Vert u\Vert_{X_N}^2$
for all $u\in V_N$, $t\in[0,T]$, $v\in B_{Z_N}(P_Ny^0,r/2)$ and $\varepsilon\in(0,\varepsilon_0]$,
where the inner product on $X_N$ is transported from $X$ by $\langle u_1,u_2\rangle_{X_N} = \langle R_Nu_1,R_Nu_2\rangle_X$.

As a consequence, Assumption \ref{H_stab} holds for $A_\varepsilon^N$ with $M=1$ and with the same $\omega$, uniformly in $\varepsilon$ and $N$.
\end{lemma}

\begin{proof}
Let $u\in X_N$ and set $y=R_Nu\in \mathrm{Ran}(R_N)\subset Z$.
Since $Q_N$ is the orthogonal projector onto $\mathrm{Ran}(R_N)$ and $y\in \mathrm{Ran}(R_N)$, we have
$\langle Q_N x,y\rangle_X=\langle x,y\rangle_X$ for every $x\in X$.
Then
$$
\langle A_\varepsilon^N[t,v]u,u\rangle_{X_N}
= \langle Q_N A_\varepsilon[t,R_Nv]y,y\rangle_X
= \langle A_\varepsilon[t,R_Nv]y,y\rangle_X 
\leq \omega \Vert y\Vert_X^2
= \omega \Vert u\Vert_{X_N}^2,
$$
which proves the uniform dissipativity property.
The semigroup stability estimate then follows by a standard energy argument.
\end{proof}

\begin{lemma}\label{lem_discrete_intertwining}
Under Assumptions \ref{H_Z} to \ref{H_discretization}, denoting by $S_\varepsilon$ the intertwining operator and by  $B_\varepsilon$ the correction term for $A_\varepsilon$, we assume moreover that
$Q_N S_\varepsilon = S_\varepsilon Q_N$ and $S_\varepsilon(\mathrm{Ran}(R_N))\subset \mathrm{Ran}(R_N)$ for every $N\in\N^*$.

Then, defining $S_\varepsilon^N=P_N S_\varepsilon R_N$ and $B_\varepsilon^N[t,v] = P_N B_\varepsilon[t,R_Nv]R_N$, we have $\Vert B_\varepsilon^N[t,v]\Vert_{L(X_N)} \leq C_{\mathrm{stab}} C_5$ and
\begin{equation}\label{disc_intertwining}
S_\varepsilon^N A_\varepsilon^N[t,v]u = A_\varepsilon^N[t,v]S_\varepsilon^Nu + B_\varepsilon^N[t,v]S_\varepsilon^Nu.
\end{equation}
for all $u\in Z_N$, $t\in[0,T]$, $v\in B_{Z_N}(P_N y^0,r/2)$, $\varepsilon\in(0,\varepsilon_0]$ and $N\geq N_0$.
Hence Assumption \ref{H_intertwining} holds for $A_\varepsilon^N$ (with $C_5$ replaced by $C_{\mathrm{stab}} C_5$) uniformly in $\varepsilon$ and $N$.
\end{lemma}

\begin{proof}
Set $z=R_Nv$. Given any $w\in X_N$, we have
\begin{multline*}
\Vert B_\varepsilon^N[t,v]w\Vert_{X_N}
= \Vert Q_N B_\varepsilon[t,z]R_Nw\Vert_X
\leq C_{\mathrm{stab}} \Vert B_\varepsilon[t,z]R_Nw\Vert_X \\
\leq C_{\mathrm{stab}} C_5 \Vert R_Nw\Vert_X
= C_{\mathrm{stab}} C_5 \Vert w\Vert_{X_N},
\end{multline*}
which gives $\Vert B_\varepsilon^N[t,v]\Vert_{L(X_N)} \leq C_{\mathrm{stab}} C_5$.

Let $u\in Z_N$, and set $y=R_Nu$.
Using that $Q_N S_\varepsilon = S_\varepsilon Q_N$ by assumption,
$$
S_\varepsilon^N A_\varepsilon^N[t,v]u
= P_N S_\varepsilon Q_N A_\varepsilon[t,z]y
= P_N Q_N S_\varepsilon A_\varepsilon[t,z]y
= P_N S_\varepsilon A_\varepsilon[t,z]y,
$$
because $P_NQ_N=P_N$. 
Now, since $S_\varepsilon A_\varepsilon[t,z]y = A_\varepsilon[t,z] S_\varepsilon y + B_\varepsilon[t,z]S_\varepsilon y$ by the intertwining assumption on $A_\varepsilon[t,z]$, we infer that
$S_\varepsilon^N A_\varepsilon^N[t,v]u = P_N A_\varepsilon[t,z]S_\varepsilon y + P_N B_\varepsilon[t,z]S_\varepsilon y$.

Since $y\in \mathrm{Ran}(R_N)\subset Z$ and by assumption $S_\varepsilon y\in \mathrm{Ran}(R_N)$, we have $Q_N S_\varepsilon y=S_\varepsilon y$, hence
$$
P_N A_\varepsilon[t,z]S_\varepsilon y
= P_N A_\varepsilon[t,z]Q_N S_\varepsilon y
= P_N A_\varepsilon[t,z]R_NP_N S_\varepsilon y
= A_\varepsilon^N[t,v]S_\varepsilon^Nu.
$$
Likewise, $P_N B_\varepsilon[t,z]S_\varepsilon y = P_N B_\varepsilon[t,z]R_NP_N S_\varepsilon y = B_\varepsilon^N[t,v]S_\varepsilon^Nu$.
This proves \eqref{disc_intertwining}.
\end{proof}

\begin{remark}\label{rem_intertwining_not_generic}
The two previous lemmas are sufficient transfer criteria, not generic facts.
They are natural in spectral or orthogonal Galerkin settings, where $\mathrm{Ran}(R_N)$ is chosen invariant under the intertwining operator and where $Q_N$ commutes with it.
They usually fail for general finite volume schemes, blob reconstructions, and most finite element quasi-interpolations.
This is why the main error estimate in Section \ref{sec_abstract_main_result} hereafter only assumes the existence of a lifted discrete solution, and not the full discrete Kato framework.
\end{remark}

\subsection{Final estimate}\label{sec_abstract_main_result}

We now quantify the error between the exact solution $y$ of \eqref{abstract_quasilinear} and a lifted finite-dimensional approximation $y_\varepsilon^N$.
The estimate separates the regularization error $y-y_\varepsilon$, controlled by Proposition \ref{prop_yepsilon}, and the discretization error $y_\varepsilon-y_\varepsilon^N$, controlled by the consistency of $Q_N$ on the output class $Y$ (see Proposition \ref{prop_yepsilonN}).

\begin{theorem}\label{thm_approx_abstract}
We make Assumptions \ref{H_Z} to \ref{H_Aepsilon_N}.
Let $y(\cdot)\in\mathscr{C}^0([0,T'],Z)\cap\mathscr{C}^1([0,T'],X)$ be the unique solution of \eqref{abstract_quasilinear} such that $y(0)=y^0$, as given by Proposition \ref{prop_existence_uniqueness}. 
Besides, for any $\varepsilon\in(0,\varepsilon_0]$ and any $N\geq N_0$, let $u_\varepsilon^N\in \mathscr{C}^0([0,T'],Z_N)\cap\mathscr{C}^1([0,T'],X_N)$ be the unique solution of the interacting system \eqref{interacting_system_u} such that $u_\varepsilon^N(0)=P_N y^0$, as given by Proposition \ref{prop_yepsilonN}, and let
$y_\varepsilon^N=R_N u_\varepsilon^N\in \mathscr{C}^0([0,T'],Z)\cap\mathscr{C}^1([0,T'],X)$ be its lift \eqref{def_yepsN}, which solves \eqref{interacting_system_y} with $y_\varepsilon^N(0)=Q_N y^0$.
Then 
\begin{equation}\label{CV_y_yepsN}
\boxed{
\Vert y_\varepsilon^N(t)-y(t) \Vert_X \leq a_1(t)\chi(\varepsilon) + \frac{b_\varepsilon(t)}{N^\gamma} 
\leq \mathrm{Cst}\left( \chi(\varepsilon) + \frac{1+L_\varepsilon}{N^\gamma} \right)
\qquad \forall t\in[0,T'] 
}
\end{equation}
where $a_1(t)$ is defined by \eqref{def_a1a2} and $b_\varepsilon(t)$ by \eqref{def_beps}.
\end{theorem}

\begin{proof}
Let $y_\varepsilon(\cdot)\in\mathscr{C}^0([0,T'],Z)\cap\mathscr{C}^1([0,T'],X)$ be the unique solution of \eqref{abstract_quasilinear_eps} such that $y_\varepsilon(0)=y^0$ (see Proposition \ref{prop_yepsilon}).
Applying the triangle inequality $\Vert y(t)-y_\varepsilon^N(t)\Vert_X \leq \Vert y(t)-y_\varepsilon(t)\Vert_X + \Vert y_\varepsilon(t)-y_\varepsilon^N(t)\Vert_X$ and using the estimates \eqref{CV_yeps_abstract} and \eqref{CV_yeps_yepsN} yields \eqref{CV_y_yepsN}.
\end{proof}

\begin{remark}\label{rem_thm_approx_abstract}
Following Remark \ref{rem_prop_yepsilonN_scope}, in Theorem \ref{thm_approx_abstract} one can replace Assumption \ref{H_Aepsilon_N} by the weaker assumption that, for any $\varepsilon\in(0,\varepsilon_0]$ and any $N\in\N^*$, there exists a solution $y_\varepsilon^N\in\mathscr{C}^0([0,T'],Z)\cap\mathscr{C}^1([0,T'],X)$ of \eqref{interacting_system_y} such that $y_\varepsilon^N(0)=Q_N y^0$ and $y_\varepsilon^N(t)\in B_Z(y^0,r)$ for every $t\in[0,T']$ (Assumption \ref{H_Aepsilon_N} is only a sufficient criterion ensuring that such a discrete trajectory exists uniformly in $N$ and $\varepsilon$).
This distinction may be essential in applications, because the well-posedness of the discrete system may be obtained by arguments that are specific to the chosen discretization.
\end{remark}

\paragraph{Consequences and typical choices of $\varepsilon$ versus $N$.}
The estimate \eqref{CV_y_yepsN} yields convergence in $\mathscr{C}^0([0,T'],X)$ as soon as $\varepsilon\to 0$ and $N\to+\infty$.
A typical regime in applications is $\chi(\varepsilon)\leq \mathrm{Cst}\, \varepsilon^\alpha$ for some $\alpha>0$ and $L_\varepsilon \leq \mathrm{Cst}\, \varepsilon^{-k}$ for some $k\geq 0$.
Then, choosing $\varepsilon_N=N^{-\gamma/(\alpha+k)}$ yields
$$
\boxed{
\sup_{t\in[0,T']}\Vert y_\varepsilon^N(t)-y(t)\Vert_X \leq \frac{\mathrm{Cst}}{N^{\alpha\gamma/(\alpha+k)}} 
}
$$
where $\mathrm{Cst}$ is independent of $N$.

In the running PDE example based on variable-step mollifiers, one has $\chi(\varepsilon)=\mathrm{O}(\varepsilon)$ and $L_\varepsilon\leq L_0$, so that, in the notation above, one is in the regime $\alpha=1$ and $k=0$. Hence the natural balance is simply $\varepsilon_N\simeq N^{-\gamma}$, which yields an algebraic rate of order $N^{-\gamma}$. 
This is the content of Corollary \ref{cor_approx_PDE} in Section \ref{subsec_PDE_discrete} and of the estimate \eqref{estim_finale}.

If $L_\varepsilon$ grows exponentially as $\varepsilon\to 0$, for instance $L_\varepsilon\leq \mathrm{Cst}\,e^{\mathrm{Cst}/\varepsilon^k}$, then the optimal balance between the regularization error and the consistency defect leads at best to logarithmic rates. This logarithmic regime is not expected to be generic, but it may occur in some singular regularization procedures and explains why the abstract theorem is formulated in terms of the quantity $L_\varepsilon$.

\begin{remark}\label{rem_comparison_literature}
The estimate \eqref{CV_y_yepsN} has a universal structure.
Once the continuous quasilinear problem has been regularized and once a family of projectors $Q_N$ approximating the identity on the output class $Y$ is available, the final rate depends only on three ingredients:
the regularization defect $\chi(\varepsilon)$, the approximation order $N^{-\gamma}$ of $Q_N$, and the output size $L_\varepsilon$.
The particular form of the PDE and of the discretization enters only through these three quantities.

The estimate \eqref{CV_y_yepsN} is intentionally abstract and should be viewed as complementary to the numerical analysis of specific time integrators for quasilinear evolution equations.

The refined Kato frameworks recalled in Remark \ref{rem_Muller_refinement} underlie several works on quasilinear wave and Maxwell equations, including
\cite{DorichHochbruck_SINUM2022, Dorich_FoCM2025, HochbruckPazur_NM2017, HochbruckPazurSchnaubelt_NM2018, KovacsLubich_NM2018}.
These papers optimize concrete time discretizations and derive error bounds tailored to the PDE structure and to the chosen state-dependent norms.
By contrast, Theorem \ref{thm_approx_abstract} isolates the interplay between a regularization scale $\varepsilon$, a reconstruction scale $N$, and an output norm $Y$.
In the formal case $A_\varepsilon=A$, it becomes a purely spatial approximation result, whereas in the boundary-compatible kernelization framework of the present paper the two-step route is essential because it produces an explicit interacting representation of the dynamics on a bounded domain with boundary conditions (see Section \ref{subsec_discussion_two_step}).
\end{remark}

\subsection{Examples of discretizations}\label{sec_examples_discretizations}
This subsection explains the scope of Assumptions \ref{H_discretization} and \ref{H_Aepsilon_N}.

Assumption \ref{H_discretization} concerns the approximation properties of the reconstruction projector $Q_N=R_NP_N$ and is meant to cover a broad class of deterministic discretizations.
In most practical discretizations, one first constructs a finite-dimensional reconstruction space $V_N\subset Z$ and a projector or quasi-interpolation operator $Q_N:X\to V_N$ such that $Q_N$ acts as the identity on $V_N$.
One then recovers the abstract pair $(P_N,R_N)$ by taking $P_N=Q_N$ as a map $X\to V_N$ and $R_N$ as the inclusion $V_N\hookrightarrow Z$.
Under this identification, Assumption \ref{H_discretization} is nothing but stability and approximation properties of the family $(Q_N)_{N\in\N^*}$.
This observation is made precise in Appendix \ref{app_discretization_technical}, Lemma \ref{lem_projector_reduction}.
The role of this subsection is therefore not to enumerate all possible schemes, but to identify the structural patterns behind the abstract assumptions.
Assumption \ref{H_discretization} packages, under a common sampling-reconstruction language, several standard traditions: it encompasses finite element quasi-interpolation (including Scott-Zhang and related operators), smooth finite-volume reconstructions, partition-of-unity and meshfree reconstructions, and spectral or orthogonal Galerkin projections.

Assumption \ref{H_Aepsilon_N} is more restrictive: it is a strong sufficient criterion for discrete well-posedness (as said in Remarks \ref{rem_prop_yepsilonN_scope} and \ref{rem_thm_approx_abstract}), natural mainly in Galerkin settings or in schemes endowed with an independent stability mechanism, but it is not a generic consequence of sampling and reconstruction alone.

Thus, \ref{H_discretization} should be viewed as the generic approximation hypothesis, whereas \ref{H_Aepsilon_N} is a convenient but non-generic abstract criterion ensuring the existence of the discrete trajectory used later in Theorem \ref{thm_approx_abstract} and in Proposition \ref{prop_yepsilonN}.

Below, we discuss some representative classes of discretizations.
For detailed constructions and proofs, see Appendix \ref{app_discretization_technical}.

\paragraph{Spectral and orthogonal Galerkin discretizations.}
This is the cleanest class for the strong discrete hypothesis \ref{H_Aepsilon_N}.
One chooses $V_N$ as a finite-dimensional subspace of $Z$, typically generated by the first $N$ modes of a spectral basis, and $Q_N$ as an orthogonal or uniformly stable projector onto $V_N$.
Then Assumption \ref{H_discretization} follows from the usual projector bounds together with a Jackson-type approximation estimate.
Moreover, when the spaces $V_N$ are chosen invariant under the intertwining operator and the projectors commute with it, Lemmas \ref{lem_H_regZ_N}, \ref{lem_discrete_stab} and \ref{lem_discrete_intertwining} show that \ref{H_Aepsilon_N} is satisfied as well.
This is the paradigm in which the discrete well-posedness assumption is the most natural.

\paragraph{Finite volumes with smooth blob reconstruction.}
A second important class is obtained from cell averages and a smooth reconstruction built from localized blobs or corrected partition-of-unity functions.
The sampling operator $P_N$ is then a local averaging operator on the pivot space $X=L^2(\Omega,\R^d)$, while $R_N$ reconstructs a smooth field from the discrete values.
Under the usual shape-regularity and bounded-overlap assumptions, the resulting projector $Q_N$ is uniformly bounded on $L^2$ and on $H^s$, and it approximates Lipschitz functions in $L^2$ with order $h_N$.
Hence Assumption \ref{H_discretization} holds with $Y=W^{1,\infty}(\Omega,\R^d)\cap L^2(\Omega,\R^d)$ or any equivalent Lipschitz-type space.
This class is particularly relevant for the PDE setting of Sections \ref{sec_general_class_quasilinear} and \ref{sec_main_result}.
By contrast, the strong discrete hypothesis \ref{H_Aepsilon_N}, especially its intertwining component, is not automatic for such schemes.

\paragraph{Finite elements, splines, and meshfree quasi-interpolation.}
Stable quasi-interpolation operators from finite elements, spline spaces, or meshfree partition-of-unity / moving least squares constructions also fit naturally into the present framework.
The common pattern is the existence of a projector or quasi-interpolant $Q_N$ that is bounded on both the pivot space $X$ and the strong space $Z$, and that approximates the class $Y$ in the $X$-norm with order $h_N$.
Whenever such a projector is available, Assumption \ref{H_discretization} follows directly.
Again, however, the stronger well-posedness assumption \ref{H_Aepsilon_N} is usually not automatic unless the discrete space is chosen compatibly with the operator structure.

\paragraph{A limitation: point sampling on $L^2$.}
If $X=L^2(\Omega,\R^d)$, the map $y\mapsto y(x_i)$ is not continuous.
Therefore pure nodal sampling is excluded from Assumption \ref{H_discretization} when the pivot space is $L^2$.
This is not a weakness of the framework but a real obstruction: one must replace point sampling by bounded functionals such as local averages, moments, or stable projections.
A precise counterexample is recalled in Appendix \ref{app_discretization_technical}, Lemma \ref{lem_point_sampling_not_bounded}.

%

\subsection{Running example: discretization of the kernelized PDE and interacting system}\label{subsec_PDE_discrete}

We continue the running PDE class of Section \ref{sec_general_class_quasilinear} and its variable-step regularization of Section \ref{sec_main_result}. We introduce the finite-dimensional operator $A_\varepsilon^N$ and the associated interacting ODE on $V_N$, and we derive the quantitative PDE rate as a direct consequence of Theorem \ref{thm_approx_abstract}.

Let $(P_N,R_N)$ be any discretization satisfying Assumption \ref{H_discretization} for the output space
$Y=W^{1,\infty}(\Omega,\R^d)\cap L^2(\Omega,\R^d)$
as introduced in Section \ref{sec_verif}.
For $\varepsilon\in(0,\varepsilon_0]$ and $N\in\N^*$, we define $A_\varepsilon^N[t,v]$ and $f_\varepsilon^N[t,v]$ by \eqref{def_A_f_eps_N}, for $(t,v)\in[0,T]\times Z_N$, and we consider the corresponding finite-dimensional interacting system \eqref{interacting_system_u} and the lifted field $y_\varepsilon^N(t)=R_Nu_\varepsilon^N(t)$.

\begin{proposition}\label{prop_AepsN_hypotheses_PDE}
Assume the standing assumptions of Section \ref{sec_general_class_quasilinear}, together with Assumption \ref{H_discretization} for the output space $Y=W^{1,\infty}(\Omega,\R^d)\cap L^2(\Omega,\R^d)$.
Then the following assertions hold for the discrete family $(A_\varepsilon^N,f_\varepsilon^N)$:
\begin{enumerate}[parsep=0.7mm, itemsep=0.7mm, topsep=0.7mm]
\item[(i)] The analogues of Assumptions \ref{H_Z}, \ref{H_regZ}, and \ref{H_f_Lip} on $(X_N,Z_N)$ hold uniformly with respect to $\varepsilon$ and $N$.
\item[(ii)] If, in addition, $Q_N$ is the orthogonal projector of $X$ onto $\mathrm{Ran}(R_N)$, then the analogue of Assumption \ref{H_stab} holds uniformly with respect to $\varepsilon$ and $N$.
\item[(iii)] If, in addition, the commutation and invariance hypotheses of Lemma \ref{lem_discrete_intertwining} are satisfied, then the analogue of Assumption \ref{H_intertwining} holds as well.
\end{enumerate}
Consequently, Assumption \ref{H_Aepsilon_N} is fully verified for orthogonal Galerkin-type discretizations that are compatible with the intertwining operator. Without this extra compatibility, all parts of Assumption \ref{H_Aepsilon_N} are verified except the intertwining one.
\end{proposition}

\begin{proof}
Item (i) follows from Lemma \ref{lem_H_regZ_N}, because Proposition \ref{prop_Aeps_hypotheses_PDE} has already verified the corresponding continuous hypotheses for $(A_\varepsilon,f_\varepsilon)$.
Item (ii) follows from Lemma \ref{lem_discrete_stab}, combined with the dissipativity estimate of Lemma \ref{lem_Aeps_dissip}.
Item (iii) is exactly Lemma \ref{lem_discrete_intertwining}.
\end{proof}

\begin{remark}\label{rem_PDE_discrete_wellposedness}
Proposition \ref{prop_AepsN_hypotheses_PDE} is a strong sufficient criterion for the well posedness of the interacting system \eqref{interacting_system_u}. The main approximation theorem below does not require this full criterion. It only requires the existence, on the time interval under consideration, of a lifted solution $y_\varepsilon^N$ remaining in the same $Z$-ball as the continuous regularized solution (see Remark \ref{rem_thm_approx_abstract}). This weaker viewpoint is the one that will be used in the most flexible PDE applications.
\end{remark}

\begin{corollary}\label{cor_approx_PDE}
Assume the standing hypotheses of Section \ref{sec_general_class_quasilinear}, together with \eqref{coeff_W1inf_PDE}, and choose the explicit regularization \eqref{def_Aeps_Heps}.
Let $Y=W^{1,\infty}(\Omega,\R^d)\cap L^2(\Omega,\R^d)$, 
and let $(P_N,R_N)$ be any discretization satisfying Assumption \ref{H_discretization} with this output space $Y$.
Fix $T''\in(0,T']$.

Assume that, for every $\varepsilon\in(0,\varepsilon_0]$, the regularized problem admits a solution
$y_\varepsilon\in \mathscr{C}^0([0,T''],Z)\cap \mathscr{C}^1([0,T''],X)$
with $y_\varepsilon(t)\in B_Z(y^0,r)$ on $[0,T'']$, and that, for every $\varepsilon\in(0,\varepsilon_0]$ and every $N\in\N^*$, the interacting system \eqref{interacting_system_u} admits a lifted solution
$y_\varepsilon^N\in \mathscr{C}^0([0,T''],Z)\cap \mathscr{C}^1([0,T''],X)$
with $y_\varepsilon^N(t)\in B_Z(y^0,r)$ on $[0,T'']$.
This is in particular the case if the semigroup-generation part of Assumption \ref{H_Aepsilon} holds for $A_\varepsilon$ and if Assumption \ref{H_Aepsilon_N} holds for $(A_\varepsilon^N,f_\varepsilon^N)$.

Then there exists $C>0$, independent of $\varepsilon$ and $N$, such that
\begin{equation*}
\Vert y_\varepsilon^N(t)-y(t)\Vert_{L^2(\Omega,\R^d)}
\leq C \left( \varepsilon + \frac{1}{N^\gamma} \right)
\qquad
\forall t\in[0,T''].
\end{equation*}
In particular, if one chooses $\varepsilon_N=N^{-\gamma}$ then
\begin{equation}\label{estim_finale}
\boxed{
\Vert y_{\varepsilon_N}^N(t)-y(t)\Vert_{L^2(\Omega,\R^d)} \leq  \frac{2C}{N^\gamma}
\qquad \forall t\in[0,T'']
}
\end{equation}
For local quasi-uniform discretizations on an $n$-dimensional domain, one typically has $\gamma=1/n$.
\end{corollary}

\begin{proof}
By Proposition \ref{prop_Aeps_hypotheses_PDE}, the explicit regularized family $(A_\varepsilon,f_\varepsilon)$ satisfies Assumption \ref{H_Aepsilon_CV} with $\chi(\varepsilon)=C_{A,\eta,\rho}\,\varepsilon$ and Assumption \ref{H_Y} with the output space $Y=W^{1,\infty}(\Omega,\R^d)\cap L^2(\Omega,\R^d)$ and with a constant $L_\varepsilon$ that is bounded uniformly on $(0,\varepsilon_0]$.
Since $(P_N,R_N)$ satisfies Assumption \ref{H_discretization} for this same space $Y$, and since the existence assumptions on $y_\varepsilon$ and $y_\varepsilon^N$ are part of the statement of the corollary, we can apply Theorem \ref{thm_approx_abstract} together with Remark \ref{rem_thm_approx_abstract} (which shows that the full assumption \ref{H_Aepsilon_N} is not needed).
This gives the corollary.
\end{proof}

\begin{remark}\label{rem_explicit_coeffs_deferred}
In the running PDE example, the interaction coefficients of the discrete system can be written explicitly by combining the kernel formula \eqref{def_A_eps_sigma_eps} with the chosen discretization operator $(P_N,R_N)$. Since these coefficients depend strongly on that choice, we do not write them out at this abstract stage; they are made explicit, in the sampling-reconstruction form \eqref{kernel_discrete_general}-\eqref{kernel_discrete_coeff}, in the worked examples of Section \ref{subsec_examples}, and the underlying discretization templates are discussed in Section \ref{sec_examples_discretizations} and Appendix \ref{app_discretization_technical}.
\end{remark}

\begin{remark}\label{rem_PDE_rate_generality}
The estimate \eqref{estim_finale} should be read as a generic spatial/interacting approximation result for the broad class \eqref{general_quasilinear_PDE} of quasilinear PDEs on bounded domains with boundary conditions, once the boundary-compatible kernelization has been verified.
Unlike the time-discretization results cited in Remark \ref{rem_comparison_literature}, the rate is not tied to one specific numerical scheme; it depends only on the approximation order of $Q_N$ in the pivot norm and on the behavior of $L_\varepsilon$ as $\varepsilon\to 0$.
In the present PDE example, the uniform bound $L_\varepsilon\leq L_0$ yields the algebraic choice $\varepsilon_N\simeq N^{-\gamma}$ and the final rate \eqref{estim_finale}.
\end{remark}

\subsection{Worked examples and verification of the assumptions}\label{subsec_examples}

We illustrate Theorem \ref{thm_approx_abstract} and Corollary \ref{cor_approx_PDE} on several emblematic equations on the interval $\Omega = (0,1)$ (so that $n = d = 1$), writing out the corresponding interacting systems and explaining, in each case, how the assumptions of our main results are met. By the running-example machinery of Section \ref{sec_verif} (Proposition \ref{prop_Aeps_hypotheses_PDE}), the regularity \ref{H_regZ}, the output control \ref{H_Y}, the convergence \ref{H_Aepsilon_CV} with $\chi(\varepsilon) = \mathrm{O}(\varepsilon)$, and the dissipativity part of \ref{H_stab} hold for the mollifier regularization $A_\varepsilon = \mathcal{H}_\varepsilon^* A \mathcal{H}_\varepsilon$ of every operator in the running class. The two assumptions that deserve a case-by-case discussion are the intertwining \ref{H_intertwining} and the semigroup-generation part of \ref{H_Aepsilon}. Accordingly, the examples are organized in three tiers: boundary-compatible interacting systems that are rigorously of the abstract form \eqref{kernel_discrete_general}, for which the conclusion of Theorem \ref{thm_approx_abstract} holds on any interval where the relevant trajectories stay in the common $Z$-ball (Section \ref{subsubsec_boundary}); two fully closed cases, periodic conditions and orthogonal spectral truncation, where the generation and intertwining of $A_\varepsilon$ hold uniformly and the conclusion is unconditional (Section \ref{subsubsec_closed}); and higher-order or nonlocal equations, such as Korteweg-de Vries and Benjamin-Ono, which fall outside the present running class but are reachable by the same kernelization once the operator $S$ and the regularization are adapted to their energy structure (Remark \ref{rem_kdv_bo}).

It is useful here to separate two levels. At the level of the \emph{original} operator $A$, the intertwining \ref{H_intertwining} is Kato's classical hypothesis $S A = A S + B S$, i.e., the boundedness on $X$ of $B = [S,A] S^{-1}$, and it holds for all the examples below once $S$ is chosen appropriately. For a first-order operator and $S = \Lambda^s$, $[S,A] S^{-1}$ is of order zero, hence bounded (an elementary pseudodifferential computation). For an operator built from the Dirichlet Laplacian $\Delta$, with $S = (\mathrm{id} - \Delta)^{s/2}$, it even holds with $B = 0$, since $A$ and $S$ are then both functions of $\Delta$ and commute. For higher-order equations such as Korteweg-de Vries (Remark \ref{rem_kdv_bo}), the naive Sobolev choice $S = \Lambda^s$ does not give a bounded $B$, but a problem-adapted $S$ tied to the conserved energies does, which is exactly how such equations are placed in Kato's quasilinear framework (see \cite[Chapter 8]{Pazy} and \cite{Kato_1975}). The first- and second-order examples below all belong to the running class of Section \ref{sec_general_class_quasilinear}, with the fixed Sobolev scale, and are locally well posed.

At the level of the \emph{regularized} family $A_\varepsilon = \mathcal{H}_\varepsilon^* A \mathcal{H}_\varepsilon$, what is specific to a domain with boundary is the generation and the intertwining of $A_\varepsilon$ itself (Remarks \ref{rem_Aeps_semigroup_issue} and \ref{rem_intertwining_continuous_not_generic}). Writing $[S, A_\varepsilon] = [S, \mathcal{H}_\varepsilon^*] A \mathcal{H}_\varepsilon + \mathcal{H}_\varepsilon^* [S,A] \mathcal{H}_\varepsilon + \mathcal{H}_\varepsilon^* A [S, \mathcal{H}_\varepsilon]$, the middle term is controlled by the harmless commutator $[S,A]$ above, but the outer terms involve the commutator $[S, \mathcal{H}_\varepsilon]$ of $S$ with the variable-step mollifier. Near $\partial\Omega$ the mollification radius $\varepsilon\rho(x)$ shrinks, so the relevant frequencies are of size $1/(\varepsilon\rho(x))$, and inserting the derivative $A$ between two mollifiers produces, after dividing by $S$, an amplitude of size $\rho'(x)/\rho(x)$ for a first-order operator (and a further factor $1/\varepsilon$ for a second-order one), which is not uniformly bounded up to the boundary. This is why \ref{H_intertwining} and the generation of $A_\varepsilon$ are kept structural in general, and why they become automatic when the regularization commutes with $S$. We exhibit two such fully closed cases in Section \ref{subsubsec_closed}: periodic boundary conditions, where $\mathcal{H}_\varepsilon$ is a Fourier multiplier, and orthogonal spectral truncation, where the regularizing projector commutes with $S$.

\subsubsection{Boundary-compatible interacting systems}\label{subsubsec_boundary}

In every example below, the regularized operator $A_\varepsilon = \mathcal{H}_\varepsilon^* A \mathcal{H}_\varepsilon$ has the explicit kernel \eqref{def_A_eps_sigma_eps}, supported in $\{\vert x - x'\vert \leq 2\varepsilon\Vert\rho\Vert_{L^\infty}\}$, and the variable-step mollifier preserves the boundary traces (Corollary \ref{cor_preserv_deriv_boundary}); thus the interaction stays inside $\Omega$ and respects the boundary condition.

Let us write the finite-dimensional system \eqref{interacting_system_u} in a way that is consistent with the abstract sampling-reconstruction framework, and not as a pointwise collocation system (recall that point evaluation is excluded by Assumption \ref{H_discretization}, see Section \ref{sec_examples_discretizations}). Choose a basis $(e_j)_{1\leq j\leq N}$ of $V_N$, and set $\theta_j^N = R_N e_j\in Z$. Let $\ell_i^N : X\to\R^d$ denote the $i$-th coordinate functional of $P_N$, so that $P_N g = (\ell_i^N(g))_{1\leq i\leq N}$. Writing $u_\varepsilon^N(t) = \sum_{j=1}^N u_j(t) e_j$ and $y_\varepsilon^N(t) = R_N u_\varepsilon^N(t) = \sum_{j=1}^N u_j(t)\theta_j^N$, the discrete system \eqref{interacting_system_u} reads
\begin{equation}\label{kernel_discrete_general}
\dot u_i(t) = \sum_{j=1}^N K_{ij}^{\varepsilon,N}[t,u(t)]\, u_j(t) + F_i^{\varepsilon,N}[t,u(t)] ,
\end{equation}
where, using $A_\varepsilon^N[t,v] = P_N A_\varepsilon[t,R_Nv] R_N$ and the kernel \eqref{def_A_eps_sigma_eps},
\begin{equation}\label{kernel_discrete_coeff}
K_{ij}^{\varepsilon,N}[t,u] = \ell_i^N\left( x\mapsto \int_\Omega \sigma_\varepsilon[t,R_Nu](x,x')\, \theta_j^N(x')\, dx' \right) ,
\qquad
F_i^{\varepsilon,N}[t,u] = \ell_i^N\big(f_\varepsilon[t,R_Nu]\big) .
\end{equation}
This is the finite-dimensional interacting system covered by Theorem \ref{thm_approx_abstract}. Its structure is local in the following sense: if $P_N$ is a local cell-average sampler and the functions $\theta_j^N$ have local support of diameter $\mathrm{O}(h_N)$, then $K_{ij}^{\varepsilon,N}$ vanishes unless the cell of $\ell_i^N$ lies within distance $\mathrm{O}(\varepsilon + h_N)$ of the support of $\theta_j^N$, so that, for quasi-uniform partitions, each degree of freedom interacts with $\mathrm{O}((1 + \varepsilon/h_N)^n)$ neighbors; if $\varepsilon\simeq h_N$, the interaction graph has uniformly bounded degree. This is the price, and the meaning, of replacing a differential operator by an integral interaction: the stencil width is $\mathrm{O}(1 + \varepsilon/h_N)$ in mesh units, larger than the fixed-width stencil of a finite-difference scheme, but still local. Since the operators below belong to the running class, the estimate of Theorem \ref{thm_approx_abstract} and Corollary \ref{cor_approx_PDE} applies on any interval on which $y_\varepsilon$ and $y_\varepsilon^N$ exist and remain in $B_Z(y^0,r)$ (Remark \ref{rem_thm_approx_abstract}).

If one replaces the stable sampling $P_N$ by point evaluation at grid points $x_i$, then \eqref{kernel_discrete_general} reduces formally to the familiar collocation formula $\dot u_i(t)\simeq \frac{1}{N}\sum_j \sigma_\varepsilon[t,u(t)](x_i,x_j)\, u_j(t)$. This point-collocation formula is useful as intuition, but it is not the $L^2$-stable discretization covered by Assumption \ref{H_discretization}; we therefore keep the kernel-coefficient form \eqref{kernel_discrete_coeff} in what follows.

\begin{example}[Transport with inflow]\label{ex_transport}
For $\partial_t y + \partial_x y = 0$ on $(0,1)$ with inflow condition $y(t,0) = 0$, one has $A = -\partial_x$ on $D(A) = \{y \in H^1 \mid y(0) = 0\}$ and $f \equiv 0$, and the exact solution is $y(t,x) = y^0(x - t)$ for $x > t$ and $0$ for $x < t$. This is a first-order hyperbolic operator; in the usual Kato treatment, with an operator $S$ adapted to the boundary condition, the intertwining property \ref{H_intertwining} is verified through a bounded order-zero commutator (a standard structural input for the original transport family). The kernel is $\sigma_\varepsilon(x,x') = -\int_\Omega H_\varepsilon(x'',x)\, \partial_{x''}H_\varepsilon(x'',x')\, dx''$, which converges to the Schwartz kernel $-\delta'(x - x')$ of $-\partial_x$ as $\varepsilon \to 0$. Away from the boundary, $\rho$ is locally almost constant, so the kernel is almost antisymmetric and the interaction is centered and almost energy-conserving; near $x = 0$ the variable step bends the interaction so as to preserve the inflow trace. The interacting system \eqref{kernel_discrete_general} is linear and finite dimensional, hence the discrete trajectory is globally defined. The conclusion of Theorem \ref{thm_approx_abstract} applies on every interval on which the corresponding regularized continuous trajectory $y_\varepsilon$ exists and remains in the prescribed $Z$-ball; the remaining structural inputs for $A_\varepsilon$, namely its generation and intertwining, are those discussed in Section \ref{subsubsec_closed}.
\end{example}

\begin{example}[Heat equation]\label{ex_heat}
For $\partial_t y = \partial_x^2 y$ on $(0,1)$ with Dirichlet conditions $y(t,0) = y(t,1) = 0$, one has $A = \partial_x^2$ on $D(A) = H^2(0,1) \cap H_0^1(0,1)$, which is dissipative. With $S = (\mathrm{id} - \Delta)^{s/2}$, both $A$ and $S$ are functions of the Dirichlet Laplacian, so they commute and the original \ref{H_intertwining} holds with $B = 0$. Integrating by parts in \eqref{def_A_eps_sigma_eps} gives the symmetric kernel $\sigma_\varepsilon(x,x') = -\int_\Omega \partial_{x''}H_\varepsilon(x'',x)\, \partial_{x''}H_\varepsilon(x'',x')\, dx''$, so that $A_\varepsilon = -(\partial_x \mathcal{H}_\varepsilon)^* (\partial_x \mathcal{H}_\varepsilon) \leq 0$ as a quadratic form on $X$. In an $L^2$-Galerkin realization, the associated matrix $(K_{ij}^{\varepsilon,N})$ in \eqref{kernel_discrete_coeff} is symmetric negative semidefinite and can be interpreted as a mollified graph Laplacian. For a general sampling-reconstruction pair $(P_N,R_N)$, the continuous quadratic form remains dissipative, but the coordinate matrix need not be symmetric; in all cases the interacting system is a dissipative flow.
\end{example}

\begin{example}[Schr\"odinger and wave equations]\label{ex_schrodinger}
For the Schr\"odinger equation $\mathrm{i}\, \partial_t y = -\partial_x^2 y$ on $(0,1)$ with Dirichlet boundary conditions, written as $\partial_t y = \mathrm{i}\, \partial_x^2 y = A y$, the original operator $A$ is skew-adjoint on $X = L^2(0,1; \C)$ and commutes with $S = (\mathrm{id} - \Delta)^{s/2}$, so the original \ref{H_intertwining} holds with $B = 0$. For the variable-step regularized family $A_\varepsilon = \mathcal{H}_\varepsilon^* A \mathcal{H}_\varepsilon$, the symmetric conjugation preserves the formal skew-symmetry at the quadratic-form level. However, as emphasized in Remarks \ref{rem_Aeps_semigroup_issue} and \ref{rem_intertwining_continuous_not_generic}, full skew-adjointness and intertwining of $A_\varepsilon$ are not automatic in the variable-step boundary setting; they hold in the fully commuting cases of Section \ref{subsubsec_closed} (periodic conditions or orthogonal spectral truncation), where the regularized flow is unitary and conserves the $L^2$ charge exactly. The wave equation $\partial_{tt} u = \partial_{xx} u$ with Dirichlet conditions can be treated similarly in first-order form $y = (u, \partial_t u)$ with $d = 2$: the original operator is skew-adjoint for the energy inner product, and the kernelization produces a Hamiltonian-type interacting system, provided the corresponding generation and intertwining properties are verified.
\end{example}

\begin{example}[Generalized Hopf and Burgers equations]\label{ex_hopf}
A quasilinear case is the generalized Hopf equation $\partial_t y + y^k\, \partial_x y = 0$ on $(0,1)$ with inflow condition $y(t,0) = 0$. It corresponds to $A[t,z]y = -z^k \partial_x y$ and $f \equiv 0$, and one may take $Z = H^s(0,1)$ with $s > 5/2$, so that $z \mapsto z^k$ is controlled in $W^{1,\infty}$. At the level of the original operator, this is again a first-order Kato-type quasilinear equation; with a boundary-adapted choice of $S$, the commutator $[S,A[t,z]] S^{-1}$ is an order-zero operator controlled by the usual coefficient bounds, which is the standard structural input needed to place the equation in Kato's framework before shock formation. For the interacting approximation, the coefficient $z^k$ is evaluated on the reconstructed field $R_N u$, not merely at isolated nodal values: in the notation of \eqref{kernel_discrete_general}, the interaction coefficients are
$$
K_{ij}^{\varepsilon,N}[u] = -\ell_i^N\left( x\mapsto \int_\Omega\int_\Omega H_\varepsilon(x'',x)\, \big(R_N u(x'')\big)^k\, \partial_{x''}H_\varepsilon(x'',x')\, \theta_j^N(x')\, dx'\, dx'' \right) ,
$$
so that the scheme $\dot u_i(t) = \sum_{j=1}^N K_{ij}^{\varepsilon,N}[u(t)]\, u_j(t)$ is a quasilinear interacting system. For $k = 1$, this corresponds to inviscid Burgers before shock formation, and for $y^0\in\mathscr{C}^1$ the exact solution is defined up to the shock time $T(y^0) = 1/\max_x \max(0, -(y^0)'(x))$; the abstract convergence estimate applies on any time interval on which the exact, regularized, and lifted discrete solutions remain in the common $Z$-ball, in particular on $[0,T'']$ for any $T''<T(y^0)$. A telling structural feature is the degeneracy of the coefficient $(R_N u(x''))^k$, which reflects the finite-propagation character of the underlying transport dynamics. Because the kernelized operator is nonlocal at scale $\varepsilon$, one should not expect exact support preservation at the discrete level; rather, the influence of the reconstructed field is confined to an $\mathrm{O}(\varepsilon)$ neighborhood of its support, and mass enters a region only as the characteristics carry it there, up to the shock time $T(y^0)$.
\end{example}

\begin{remark}[Higher-order and nonlocal equations]\label{rem_kdv_bo}
Higher-order and nonlocal equations are reached by the kernelization viewpoint as well, although they require a problem-adapted choice of the operator $S$ rather than the fixed Sobolev scale used in our running class. The Korteweg-de Vries equation $\partial_t y + \partial_x^3 y + 6\, y\, \partial_x y = 0$ is a classical instance: it is treated by Kato's quasilinear theory, precisely as the closing application of the abstract framework in \cite[Chapter 8]{Pazy} (see also \cite{Kato_1975}), where the intertwining \ref{H_intertwining} is verified for a suitably chosen $S$ adapted to the conserved energies of the equation. 
At the level of the kernelization, the third-order part produces an antisymmetric dispersive kernel and the quadratic term a nonlinear interaction. Similarly, for the Benjamin-Ono equation $\partial_t y + y\, \partial_x y + H \partial_x^2 y = 0$, where $H$ is the Hilbert transform, the differential part is localized by the mollifier while the nonlocal operator $H$ is inherited by the kernel, producing a nonlocal interacting system. These equations are not part of the running class of Section \ref{sec_general_class_quasilinear}, which is built on the fixed pair $(Z,X)$ and the variable-step mollifier; adapting the kernelization to their natural energy structure is a natural direction, and the abstract Theorem \ref{thm_approx_abstract} already applies once the corresponding $S$, regularization, and discretization are set up.
\end{remark}

\subsubsection{Two fully closed cases: periodic conditions and spectral truncation}\label{subsubsec_closed}

In the boundary examples above, the rate of Corollary \ref{cor_approx_PDE} is conditional on the existence of $y_\varepsilon$ and $y_\varepsilon^N$ on a common interval, since the generation and the intertwining of the mollified operator $A_\varepsilon$ are not granted by the construction alone. We now describe two settings in which the regularization commutes with $S$, so that these structural inputs hold with $B = 0$ and the framework closes unconditionally.

\paragraph{Periodic boundary conditions.}
We work in the running class of Section \ref{sec_general_class_quasilinear} on the flat torus $\Omega = \mathbb{T}^n = (\R/\Z)^n$, for which $\partial\Omega = \emptyset$ and $\vert\Omega\vert = 1$. We use complex Fourier modes, working on the complexification of $X$ and recovering real-valued solutions by the conjugacy symmetry of the Fourier coefficients. Since there is no boundary, we take $\rho \equiv 1$ (as allowed in Appendix \ref{app_convolution}) and we choose the kernel $\eta$ even. The variable-step mollifier then reduces to the classical convolution $\mathcal{H}_\varepsilon g = \eta_\varepsilon * g$, with $\mathcal{H}_\varepsilon^* = \mathcal{H}_\varepsilon$, which is the Fourier multiplier $\widehat{\mathcal{H}_\varepsilon g}(k) = m_\varepsilon(k)\, \hat g(k)$, with $m_\varepsilon(k) = \hat\eta(\varepsilon k) \in \R$ for $k \in \Z^n$, $m_\varepsilon(0) = 1$ and $\vert m_\varepsilon(k)\vert \leq 1$. We consider a linear constant-coefficient operator and no source, $A = \sum_{\vert\alpha\vert \leq p} a_\alpha D^\alpha$ with $a_\alpha \in \R^{d \times d}$ and $f \equiv 0$, and we assume that $A - \omega\, \mathrm{id}$ is dissipative on $X = L^2(\mathbb{T}^n, \R^d)$ for some $\omega \in \R$, i.e., the symbol $\hat A(k) = \sum_{\vert\alpha\vert \leq p} a_\alpha (2\pi \mathrm{i} k)^\alpha$ satisfies $\frac{1}{2}(\hat A(k) + \hat A(k)^*) \leq \omega\, \mathrm{Id}$ for every $k \in \Z^n$. We take $S = \Lambda^s$, the Fourier multiplier $\langle k\rangle^s = (1 + \vert k\vert^2)^{s/2}$, so that $D(S) = H^s$ and the graph norm is equivalent to the $H^s$-norm; here $Z = H^s(\mathbb{T}^n, \R^d)$ with $s > p + 1 + \frac{n}{2}$ (no boundary constraint, so $D = H^p$).

\begin{proposition}\label{prop_periodic_unconditional}
Under the above periodic assumptions, the regularized family $A_\varepsilon = \mathcal{H}_\varepsilon^* A \mathcal{H}_\varepsilon$ satisfies, uniformly in $\varepsilon \in (0,1]$:
\begin{enumerate}[label=(\roman*), parsep=0.4mm, itemsep=0.4mm, topsep=0.4mm]
\item $A_\varepsilon \in L(X)$ for each $\varepsilon$, and $A_\varepsilon - \omega_0\, \mathrm{id}$ is dissipative on $X$ with $\omega_0 = \max(0, \omega)$; in particular $A_\varepsilon$ generates a $\mathscr{C}^0$ semigroup with $\Vert\mathrm{e}^{t A_\varepsilon}\Vert_{L(X)} \leq \mathrm{e}^{\omega_0 t}$;
\item the intertwining \ref{H_intertwining} holds with $S_\varepsilon = S$ and $B_\varepsilon = 0$ (hence $C_5 = 0$);
\item Assumption \ref{H_Aepsilon} holds uniformly in $\varepsilon$.
\end{enumerate}
Consequently, the regularized solution $y_\varepsilon$ of \eqref{abstract_quasilinear_eps} exists globally. Moreover, after choosing the working radius $r$ large enough, depending on $T$ and on the growth bound, the trajectory remains in $B_Z(y^0,r)$ on $[0,T]$. In the contractive or unitary cases, $r\geq 2\Vert y^0\Vert_Z$ suffices; alternatively one may restrict the time interval as in Proposition \ref{prop_existence_uniqueness}. Let moreover $Q_N$ be the $L^2$-orthogonal projection onto the trigonometric space $V_N = \mathrm{span}\{x \mapsto \mathrm{e}^{2\pi \mathrm{i} k\cdot x} c \mid \vert k\vert_\infty \leq K,\ c \in \R^d\}$ of cutoff $K$, so that $N = \dim V_N = d(2K+1)^n$. Then Assumptions \ref{H_discretization} and \ref{H_Aepsilon_N} hold uniformly, with $\gamma = 1/n$, and Theorem \ref{thm_approx_abstract} and Corollary \ref{cor_approx_PDE} apply unconditionally. The choice $\varepsilon_N \simeq N^{-1/n}$ yields
$$
\sup_{t \in [0,T]} \Vert y_{\varepsilon_N}^N(t) - y(t)\Vert_{L^2(\mathbb{T}^n, \R^d)} = \mathrm{O}(N^{-1/n}) .
$$
\end{proposition}

\begin{proof}
Since $\mathcal{H}_\varepsilon$, $A$ and $S$ are Fourier multipliers, they commute pairwise. The multiplier of $A_\varepsilon$ is $m_\varepsilon(k)^2 \hat A(k)$, which is bounded in $k$ for each fixed $\varepsilon$ because $\hat\eta$ is rapidly decreasing while $\hat A(k)$ has polynomial growth; hence $A_\varepsilon \in L(X)$. The dissipativity in (i) is Lemma \ref{lem_Aeps_dissip} together with $\Vert\mathcal{H}_\varepsilon\Vert_{L(X)} \leq 1$, and generation follows from Remark \ref{rem_Aeps_semigroup_issue}, $A_\varepsilon$ being bounded. For (ii), $S A_\varepsilon = A_\varepsilon S$ since the scalar multiplier $\langle k\rangle^s$ commutes with the matrix multiplier $m_\varepsilon(k)^2 \hat A(k)$ modewise; thus $B_\varepsilon = 0$, and the resolvent invariance of $Z$ is immediate. The remaining parts of \ref{H_Aepsilon}, namely \ref{H_regZ}, \ref{H_Aepsilon_CV} and \ref{H_Y}, follow from the verification of Section \ref{sec_verif} (Proposition \ref{prop_Aeps_hypotheses_PDE}), which applies verbatim with $\rho \equiv 1$ (so that $L_\rho = 0$ and all the estimates of Appendix \ref{app_convolution} hold for every $\varepsilon \in (0,1]$); this gives (iii). Existence of $y_\varepsilon$ then follows from Proposition \ref{prop_yepsilon}, with $T' = T$ by linearity. For the discretization, $Q_N$ is the $L^2$-orthogonal projection onto $V_N$, hence $\Vert Q_N\Vert_{L(X)} = \Vert Q_N\Vert_{L(Z)} = 1$; for $g \in Y \subset H^1$ one has $\Vert Q_N g - g\Vert_X \leq C(K+1)^{-1}\Vert g\Vert_{H^1} \leq C(K+1)^{-1}\Vert g\Vert_Y$ and $\Vert Q_N y - y\Vert_X \leq (K+1)^{-s}\Vert y\Vert_{H^s}$, so that \ref{H_discretization} holds with order $K^{-1}$ in the cutoff, that is $\gamma = 1/n$ in the number $N$ of degrees of freedom. Since $Q_N$ is $L^2$-orthogonal, commutes with $S = S_\varepsilon$, and leaves $V_N$ invariant, Lemmas \ref{lem_H_regZ_N}, \ref{lem_discrete_stab} and \ref{lem_discrete_intertwining} yield \ref{H_Aepsilon_N}. Theorem \ref{thm_approx_abstract} applies, and the balance $\varepsilon_N \simeq N^{-1/n}$ in Corollary \ref{cor_approx_PDE} (here $\chi(\varepsilon) = \mathrm{O}(\varepsilon)$ and $L_\varepsilon \leq L_0$) gives the announced rate.
\end{proof}

In the periodic setting the kernel \eqref{def_A_eps_sigma_eps} is translation invariant: with $\rho \equiv 1$ and $\eta$ even, $H_\varepsilon(x,x') = \eta_\varepsilon(x - x')$ and $\sigma_\varepsilon(x,x') = \sum_{\vert\alpha\vert \leq p} a_\alpha\, (\eta_\varepsilon * D^\alpha\eta_\varepsilon)(x - x')$, supported in $\{\vert x - x'\vert \leq 2\varepsilon\}$. In the Fourier-Galerkin discretization of Proposition \ref{prop_periodic_unconditional}, the exact finite-dimensional matrix is the matrix of the projected multiplier $m_\varepsilon(k)^2 \hat A(k)$ for $\vert k\vert_\infty \leq K$. Equivalently, in nodal variables on the uniform periodic grid $x_j = j/N$, it is a circulant interaction matrix with kernel $\sigma_{\varepsilon,N}(x) = \sum_{\vert k\vert_\infty \leq K} m_\varepsilon(k)^2 \hat A(k)\, \mathrm{e}^{2\pi \mathrm{i} kx}$, so that $\dot y_i = \frac{1}{N} \sum_j \sigma_{\varepsilon,N}(x_i - x_j)\, y_j$. This projected kernel $\sigma_{\varepsilon,N}$ should be distinguished from the unprojected collocation formula using the full kernel $\sigma_\varepsilon$. For transport ($A = -c\, \partial_x$), $\sigma_{\varepsilon,N}$ is odd and the matrix is skew-symmetric, the translation-invariant counterpart of Example \ref{ex_transport}; for the heat equation ($A = \partial_x^2$), $\sigma_{\varepsilon,N}$ is even with zero mean and the matrix is a symmetric negative semidefinite graph Laplacian, the counterpart of Example \ref{ex_heat}. The Schr\"odinger and wave equations are covered in the same way, the constant-coefficient skew-adjoint symbol giving an energy-preserving circulant system.

\medskip
Figure \ref{fig_convergence} gives a numerical illustration in the periodic transport case. It illustrates the sharpness of the abstract rate in a simple closed setting. The experiment uses the periodic translation-invariant interacting system in nodal Fourier variables, with the balanced choice $\varepsilon_N \simeq N^{-1}$ and a smooth periodic datum, the time integration being performed by exact exponentiation of the (small, circulant) interaction matrix so that the spatial error is isolated. The observed order is $\mathrm{O}(N^{-2})$, better than the guaranteed $\mathrm{O}(N^{-1})$ bound, reflecting the second-order consistency of the symmetric kernelization on smooth data. The interaction matrix is antisymmetric, so the scheme conserves the discrete energy, in agreement with the skew-symmetry of $A_\varepsilon$. The reported experiment uses the standard normalized bump mollifier $\eta(z) = c_\eta\exp(1/(\vert z\vert^2-1))\mathbf 1_{\vert z\vert<1}$, the smooth datum $y^0(x) = \sin(2\pi x) + \frac{1}{2}\sin(4\pi x + 0.7)$ on the torus, transport speed $c = 1$, final time $t = \frac{1}{4}$, the balance $\varepsilon_N = 4/N$, and exact time integration by matrix exponentiation, with $N\in\{32,64,128,256,512\}$.

\begin{figure}[ht]
\centering
\includegraphics[width=9cm]{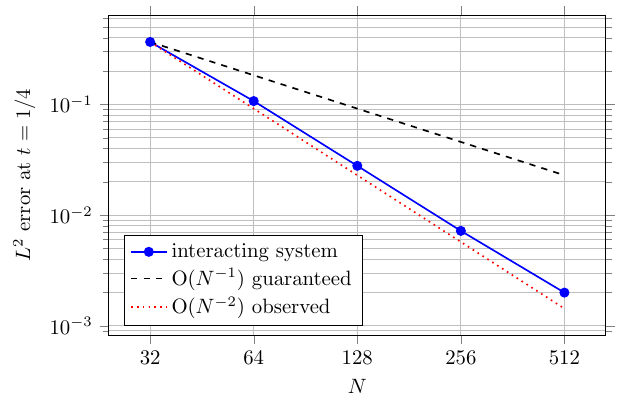}
\caption{Periodic transport interacting system (translation-invariant counterpart of Example \ref{ex_transport}): $L^2$ error at $t = \frac{1}{4}$ versus the number $N$ of degrees of freedom, for $\varepsilon_N \simeq N^{-1}$ and a smooth datum. The error remains below the rate $\mathrm{O}(N^{-1})$ guaranteed by Proposition \ref{prop_periodic_unconditional}, and is in fact of order $\mathrm{O}(N^{-2})$ for smooth data.}
\label{fig_convergence}
\end{figure}

\paragraph{Orthogonal spectral truncation.}
A second fully closed case is available on a bounded domain, with a boundary condition, when $A$ is built from a self-adjoint operator with compact resolvent and the regularization is an orthogonal spectral truncation rather than a mollifier. Let $L$ be a nonnegative self-adjoint operator on $X = L^2(\Omega, \mathbb{K})$ ($\mathbb{K} = \R$ or $\C$) with compact resolvent, and let $(\lambda_k, \phi_k)_{k\geq 1}$ be its eigenpairs, with $0\leq\lambda_1\leq\lambda_2\leq\cdots\to+\infty$ and $(\phi_k)$ an orthonormal basis of $X$ incorporating the boundary condition (for instance $L = -\Delta$ on $(0,1)$ with Dirichlet conditions, $\phi_k(x) = \sqrt{2}\sin(k\pi x)$, $\lambda_k = (k\pi)^2$). We take $S = (\mathrm{id} + L)^{s/2}$, so that $Z = D(S) = D(L^{s/2})$, with $s > 2$. For $\varepsilon\in(0,1]$, let $\pi_\varepsilon$ be the orthogonal projector onto the eigenspaces with $\lambda_k\leq\varepsilon^{-2}$, and set $A_\varepsilon = \pi_\varepsilon A \pi_\varepsilon$.

\begin{proposition}\label{prop_spectral_unconditional}
Let $A = -L$ (heat equation) or $A = -\mathrm{i} L$ (Schr\"odinger equation), with $f\equiv 0$, and let $s\geq 3$ be an integer. Then the regularized family $A_\varepsilon = \pi_\varepsilon A\pi_\varepsilon$ satisfies, uniformly in $\varepsilon\in(0,1]$:
\begin{enumerate}[label=(\roman*), parsep=0.4mm, itemsep=0.4mm, topsep=0.4mm]
\item $A_\varepsilon\in L(X)$ with $\Vert A_\varepsilon\Vert_{L(X)}\leq\varepsilon^{-2}$; moreover $A_\varepsilon$ is dissipative (resp., skew-adjoint) and generates on $X$ a $\mathscr{C}^0$ semigroup of contractions (resp., a unitary group) if $A = -L$ (resp., if $A = -\mathrm{i} L$);
\item the intertwining \ref{H_intertwining} holds with $S_\varepsilon = S$ and $B = 0$, since $A_\varepsilon$ and $S$ are both functions of $L$ and commute;
\item the convergence \ref{H_Aepsilon_CV} holds with $\chi(\varepsilon) = \mathrm{O}(\varepsilon^{s-2})$, and the output bound \ref{H_Y} holds with the choice $Y = D(L^{1/2})$ and a constant $L_\varepsilon\leq L_0$ uniform in $\varepsilon$ (this is where $s\geq 3$ is used).
\end{enumerate}
In particular, Assumption \ref{H_Aepsilon} holds uniformly in $\varepsilon$ and the regularized solution $y_\varepsilon$ exists globally. 
After choosing the working radius $r$ large enough (for the contractive or unitary flow here, $r\geq 2\Vert y^0\Vert_Z$ suffices) or after restricting the time interval as in Proposition \ref{prop_existence_uniqueness}, the trajectory remains in $B_Z(y^0,r)$ on $[0,T]$. 

Let now $V_N = \mathrm{span}(\phi_1, \ldots, \phi_N)$, let $P_N$ be the analysis map $P_N g = (\langle g,\phi_k\rangle_X)_{1\leq k\leq N}$ and $R_N$ the synthesis map $R_N(c_k) = \sum_k c_k\phi_k$, and let $Q_N$ be the $L^2$-orthogonal projection onto $V_N$. Then the approximation estimates hold in spectral form,
$$
\Vert Q_N g - g\Vert_X\leq\lambda_{N+1}^{-1/2}\Vert g\Vert_Y
\qquad
\forall g\in Y = D(L^{1/2}).
$$
In particular, if the eigenvalues satisfy a polynomial lower bound $\lambda_{N+1}^{-1/2}\leq C N^{-\gamma_L}$ for some $\gamma_L>0$, then Assumption \ref{H_discretization} holds with $\gamma = \gamma_L$ (together with \ref{H_Aepsilon_N}), and Theorem \ref{thm_approx_abstract} applies unconditionally. 

Assuming for simplicity that $\lambda_N<\lambda_{N+1}$ (otherwise one takes $N$ at the end of a spectral cluster, or indexes the discretization by the spectral cutoff rather than by $N$), the choice $\varepsilon_N = \lambda_N^{-1/2}$ gives $\pi_{\varepsilon_N} = Q_N$, so that $A_{\varepsilon_N}^N$ is the spectral Galerkin compression of $A$ to $V_N$, and one obtains
$$
\sup_{t\in[0,T]} \Vert y_{\varepsilon_N}^N(t) - y(t)\Vert_{L^2(\Omega, \mathbb{K})} = \mathrm{O}(\lambda_N^{-1/2}) .
$$
For $L = -\Delta$ on a regular $n$-dimensional domain, Weyl's law gives $\lambda_N\simeq c N^{2/n}$, hence $\gamma_L = 1/n$ and the algebraic rate $\mathrm{O}(N^{-1/n})$.
\end{proposition}

\begin{proof}
Since $\pi_\varepsilon$, $A$ and $S$ are functions of $L$ in the spectral calculus, they commute pairwise and are simultaneously diagonalized by $(\phi_k)$. The multiplier of $A_\varepsilon$ is $a(\lambda_k)\mathbf 1_{\lambda_k\leq\varepsilon^{-2}}$, with $a(\lambda) = -\lambda$ or $-\mathrm{i}\lambda$, so $\Vert A_\varepsilon\Vert_{L(X)} = \sup_{\lambda_k\leq\varepsilon^{-2}}\vert a(\lambda_k)\vert\leq\varepsilon^{-2}$; dissipativity (respectively skew-adjointness) is read off the sign (respectively the purely imaginary value) of $a(\lambda_k)$, and generation follows from Remark \ref{rem_Aeps_semigroup_issue}, $A_\varepsilon$ being bounded. This gives (i), and (ii) is immediate from the modewise commutation. For (iii), since $(A - A_\varepsilon)y = a(L)(\mathrm{id} - \pi_\varepsilon)y$ and $\vert a(\lambda)\vert = \lambda$, one has, for $y\in Z = D(L^{s/2})$,
$$
\Vert (A - A_\varepsilon)y\Vert_X^2 = \sum_{\lambda_k>\varepsilon^{-2}}\lambda_k^2\,\vert\langle y,\phi_k\rangle\vert^2 \leq \varepsilon^{2(s-2)}\sum_{\lambda_k>\varepsilon^{-2}}\lambda_k^s\,\vert\langle y,\phi_k\rangle\vert^2 \leq \varepsilon^{2(s-2)}\Vert y\Vert_Z^2 ,
$$
using $\lambda_k^{2-s}\leq\varepsilon^{2(s-2)}$ for $\lambda_k>\varepsilon^{-2}$ and $s>2$; hence $\chi(\varepsilon) = \mathrm{O}(\varepsilon^{s-2})$. For the output bound, with $Y = D(L^{1/2})$ and $z\in B_Z(y^0,r)$, the contraction $\pi_\varepsilon$ commutes with $L$, so $\Vert A_\varepsilon z\Vert_Y = \Vert(\mathrm{id}+L)^{1/2}L\pi_\varepsilon z\Vert_X\leq\Vert(\mathrm{id}+L)^{1/2}Lz\Vert_X\leq C\Vert z\Vert_{D(L^{3/2})}\leq C\Vert z\Vert_Z$ as soon as $s\geq 3$, which gives $L_\varepsilon\leq L_0$ uniformly. Assumption \ref{H_Aepsilon} thus holds uniformly, and existence of $y_\varepsilon$ on $[0,T]$ follows from Proposition \ref{prop_yepsilon} (with $T' = T$ by linearity).

For the discretization, $Q_N$ is the $L^2$-orthogonal projector onto $V_N$, hence $\Vert Q_N\Vert_{L(X)} = \Vert Q_N\Vert_{L(Z)} = 1$ since $V_N$ is spanned by eigenfunctions; for $g\in Y = D(L^{1/2})$ one has 
$$
\Vert Q_N g - g\Vert_X^2 = \sum_{k>N}\vert\langle g,\phi_k\rangle\vert^2\leq\lambda_{N+1}^{-1}\sum_{k>N}\lambda_k\vert\langle g,\phi_k\rangle\vert^2\leq\lambda_{N+1}^{-1}\Vert g\Vert_Y^2 ,
$$
and similarly $\Vert Q_N y - y\Vert_X\leq\lambda_{N+1}^{-s/2}\Vert y\Vert_Z$. Since $Q_N$ is $L^2$-orthogonal, commutes with $S = S_\varepsilon$, and leaves $V_N$ invariant, Lemma \ref{lem_spectral_scheme_verification} of Appendix \ref{app_discretization_technical} yields \ref{H_discretization} and \ref{H_Aepsilon_N}. Theorem \ref{thm_approx_abstract} then applies, and, when $N$ is chosen at the end of a spectral cluster, the choice $\varepsilon_N = \lambda_N^{-1/2}$ gives $\pi_{\varepsilon_N} = Q_N$, so that $A_{\varepsilon_N}^N = Q_N A Q_N$ is exactly the spectral Galerkin compression; this balances the regularization error $\chi(\varepsilon_N) = \mathrm{O}(\lambda_N^{-(s-2)/2})$ against the discretization error $\mathrm{O}(\lambda_N^{-1/2})$, and since $s\geq 3$, the latter dominates, giving the announced rate. The Weyl asymptotics $\lambda_N\simeq c N^{2/n}$ for $L = -\Delta$ turns it into $\mathrm{O}(N^{-1/n})$.
\end{proof}

\begin{remark}\label{rem_spectral_rate}
The rate in Proposition \ref{prop_spectral_unconditional} is stated for the minimal output space $Y = D(L^{1/2})$, for which $\gamma = 1/n$, in agreement with Corollary \ref{cor_approx_PDE}. Choosing instead $Y = D(L^{m/2})$ with $2\leq m\leq s - 2$ (allowed since $s$ may be taken arbitrarily large) yields the faster rate $\mathrm{O}(\lambda_N^{-m/2}) = \mathrm{O}(N^{-m/n})$ for data of the corresponding smoothness, which is the usual gain of spectral discretizations on smooth data within the Sobolev scale. The price of the spectral construction, compared with the local mollifier interaction of Section \ref{subsubsec_boundary}, is that the interaction becomes global: the kernel of $A_\varepsilon$ is the spectrally truncated kernel of $A$, supported on all of $\Omega\times\Omega$. The heat and Schr\"odinger equations of Examples \ref{ex_heat} and \ref{ex_schrodinger} on $(0,1)$ with Dirichlet conditions thus provide boundary cases in which all the assumptions of Theorem \ref{thm_approx_abstract} hold and the conclusion is unconditional. The wave equation enters an analogous framework after replacing the pivot space by the energy space $X = D(L^{1/2})\times L^2$ and rewriting the equation in first-order form $y = (u,\partial_t u)$, with $L$ the Dirichlet Laplacian and $A$ skew-adjoint for the energy inner product.
\end{remark}

\section{Conclusion, discussion and perspectives}\label{sec_discussion}
The main quantitative results of the paper are contained in Theorem \ref{thm_approx_abstract} and, for the running PDE class, in Corollary \ref{cor_approx_PDE}. The purpose of the present section is to make explicit what these results mean conceptually, numerically, and structurally.
We first interpret the two-step route as a systematic way to build interacting approximations of quasilinear PDEs on bounded domains. We then discuss locality and sparsity, the balance between the regularization scale $\varepsilon$ and the discretization scale $N$, the scope of the strong well-posedness assumption \ref{H_Aepsilon_N}, and several natural directions for extension.

\subsection{From kernelization to an interacting system on \texorpdfstring{$V_N$}{VN}}\label{subsec_examples_interacting}

A key point of the two-step route developed in this paper is that, once the regularized operator $A_\varepsilon[t,z]$ admits a genuine kernel $\sigma_\varepsilon[t,z](x,x')$ (as in \eqref{def_A_eps_sigma_eps} in the running PDE example of Section \ref{sec_main_result}), every discretization based on a projector or quasi-interpolation operator $Q_N$ produces a finite-dimensional interacting system on $V_N$.
Indeed, after reconstruction, the quantity $Q_N(A_\varepsilon[t,z]z)$ depends only on finitely many sampled degrees of freedom and can therefore be encoded as an interaction law on $V_N$.
In this sense, the present framework systematically turns a regularized PDE into an \emph{interacting approximation system}, rather than merely into an abstract finite-dimensional ODE.

This point is particularly transparent in the running PDE example of Sections \ref{sec_general_class_quasilinear} and \ref{sec_main_result}.
The original differential operator has only a distributional Schwartz kernel supported on the diagonal. The boundary-compatible regularization by variable-step mollifiers replaces it with an explicit localized kernel $\sigma_\varepsilon[t,z]$, supported inside $\Omega\times\Omega$, and preserving the boundary traces of sufficiently regular fields.
The interacting structure therefore does not come from the discretization alone; it is already latent in the kernelized operator.

This provides a general procedure to construct interacting approximations of quasilinear PDEs on bounded domains with boundary conditions, together with a quantitative error estimate. In that sense, the present article is not merely about finite-dimensional approximation: it is about a systematic way to pass from quasilinear evolution equations to explicit interacting systems.

At a heuristic level, two emblematic examples may be kept in mind.
For a one-dimensional transport equation, the symmetric construction $A_\varepsilon=\mathcal H_\varepsilon^*A\mathcal H_\varepsilon$ with an even kernel yields a formally skew-symmetric interaction law; in symmetric or Galerkin realizations this gives a centered, energy-conserving stencil once $\varepsilon$ is matched with the meshsize (see Example \ref{ex_transport}). For a heat equation, the continuous quadratic form is dissipative, and in an $L^2$-Galerkin realization the associated matrix is symmetric negative semidefinite, resembling a graph Laplacian (see Example \ref{ex_heat}); for a general sampling-reconstruction pair, the coordinate matrix need not be symmetric, although the continuous regularized form remains dissipative.
These examples illustrate the two main features of the present framework: explicit interaction coefficients and a direct control of the locality pattern through the kernelization scale.

\subsection{Locality of the interactions and sparsity regimes}\label{subsec_examples_locality}
In the running PDE example, the kernel $\sigma_\varepsilon[t,z](x,x')$ is localized near the diagonal: it vanishes whenever $\Vert x-x'\Vert_{\R^n}$ is larger than a constant multiple of $\varepsilon$ (see Section \ref{sec_regularized_operator}).
As a consequence, if the sampling points or cells underlying the discretization are quasi-uniform\footnote{This means that their diameters are all comparable to a common meshsize $h_N$ and, in the Euclidean setting, their volumes are all comparable to $h_N^n$ uniformly in $N$, see for instance \cite[Chapter 4]{BrennerScott} in the finite element setting and \cite[Chapter 3]{EymardGallouetHerbin} in the finite volume setting.}
then each degree of freedom interacts only with those lying within distance $\mathrm{O}(\varepsilon)$.
The number of effective neighbors is therefore of order $\big(\frac{\varepsilon}{h_N}\big)^n$.

This has a clear numerical interpretation.
If $\varepsilon$ is chosen of the same order as $h_N$, then the interaction graph has uniformly bounded degree: the discrete operator is sparse and the cost per time step remains essentially linear in the number of degrees of freedom.
If, by contrast, $\varepsilon$ is much larger than $h_N$, then the interaction matrix becomes denser, which may improve robustness but increases the computational cost.

Compared with classical local discretizations such as finite differences or standard finite volumes, the present approach is usually not the most economical one if one only wants a numerical solution of the PDE.
Its gain lies elsewhere: it provides an explicit interacting representation of the regularized dynamics, compatible with the geometry of $\Omega$ and its boundary conditions.
This makes it conceptually closer to particle, blob, or interacting-agent descriptions (see for instance \cite{Chertock_HNA2017, CottetKoumoutsakos, HockneyEastwood, Monaghan1992, Wendland2005}), while still remaining connected to the original PDE through a rigorous error estimate.

\subsection{Scale balance suggested by the error estimate}\label{subsec_examples_balance}
Theorem \ref{thm_approx_abstract} shows that the total error is the sum of two independent mechanisms:
the regularization error, measured by $\chi(\varepsilon)$, and the discretization error, measured by $(1+L_\varepsilon)N^{-\gamma}$.

In the running PDE example of Sections \ref{sec_general_class_quasilinear} and \ref{sec_main_result}, one has
$\chi(\varepsilon)=\mathrm{O}(\varepsilon)$ and $L_\varepsilon\leq L_0$, so that the natural balance is simply
$\varepsilon_N\simeq N^{-\gamma}$ .
In the usual quasi-uniform sense described above, this means that $\varepsilon_N$ is of the same order as the meshsize $h_N$.
Thus the kernelization scale and the discretization scale are matched, and one obtains simultaneously a local interaction graph of bounded degree, an algebraic convergence rate in $N$, and a transparent interpretation of the approximation as a sparse interacting system.

More generally, when $L_\varepsilon$ grows as $\varepsilon\to 0$, the theorem makes explicit the trade-off between regularization and discretization.
This is one of the main conceptual outputs of the paper: the rate is not attached to a particular numerical method, but to the balance between a kernelization scale and a reconstruction scale.

\subsection{Why a two-step route?}\label{subsec_discussion_two_step}
The abstract framework separates two distinct mechanisms:
\begin{itemize}[parsep=0.7mm, itemsep=0.7mm, topsep=0.7mm]
\item The regularization step, indexed by $\varepsilon$, is designed to \emph{``kernelize"} the dynamics: it converts operators with distributional kernels and boundary constraints into operators that can be interpreted as integral interactions inside $\Omega$. On bounded domains, boundary compatibility is a structural constraint, not merely a technicality.
\item The discretization step, indexed by $N$, turns a regularized infinite-dimensional dynamics into a finite-dimensional interacting system. This step is deliberately flexible and covers classical discretizations (finite elements, finite volumes) as well as meshfree and blob-type reconstructions.
\end{itemize}
If one only aims at a finite-dimensional approximation of an evolution problem, direct discretizations of the unbounded operator $A$ are of course available and often preferable. The interest of the present route is that it produces, through kernelization, an explicit interaction mechanism inside $\Omega$ that preserves boundary traces. This is the key feature needed to interpret the approximation as an interacting system on a bounded domain with boundary conditions. Compared with deterministic particle methods for diffusion, where the diffusive mechanism is encoded in a Lagrangian transport velocity (see \cite{Russo_CPAM1990}), here the corresponding step is the boundary-compatible kernelization $A\mapsto A_\varepsilon=\mathcal{H}_\varepsilon^*A\mathcal{H}_\varepsilon$, which keeps the analytic structure operator-theoretic and produces an explicit kernel on $\Omega\times\Omega$.

This should also be compared with the observation of \cite{HochbruckPazurSchnaubelt_NM2018} that the operator-theoretic Kato framework (and its refinement \cite{Muller}) is less flexible for certain boundary value problems, especially in Maxwell equations with standard boundary conditions.
One motivation of the present two-step route is precisely to replace a distributional diagonal kernel by a boundary-compatible integral kernel supported inside $\Omega\times\Omega$.

\paragraph{When can one skip the regularization?}
If the original operator $A[t,z]$ already admits a discretization with good stability and consistency properties on the scale $(Z,X)$, one may set formally $A_\varepsilon=A$ and $\chi(\varepsilon)=0$, or interpret the analysis as a one-step argument with a single parameter $N$. This corresponds to classical Galerkin or finite element approximations. The price is that the discrete operator is then built directly from the differential operator, rather than from an explicit kernel supported in $\Omega\times\Omega$. In contrast, the PDE construction of Sections \ref{sec_general_class_quasilinear} and \ref{sec_main_result} provides a canonical kernelization on bounded domains, which can then be discretized by many different sampling-reconstruction choices.

\subsection{Discrete well-posedness versus a posteriori error estimates}\label{subsec_discussion_wellposedness}
Assumption \ref{H_Aepsilon_N} is a strong sufficient criterion ensuring that, for each $(\varepsilon,N)$, the discrete dynamics is well posed in $V_N$ and enjoys stability estimates uniform in $N$. It is natural in Galerkin or spectral settings, and also in discretizations equipped with an independent stability mechanism.
Typical examples include:
\begin{itemize}[parsep=0.7mm, itemsep=0.7mm, topsep=0.7mm]
\item orthogonal Galerkin approximations, where discrete dissipativity follows from the projector structure (see for instance \cite{Thomee});
\item monotone finite-volume schemes for conservation laws, where stability follows from discrete maximum principles (see \cite{LeVeque, GodlewskiRaviart, EymardGallouetHerbin});
\item explicit schemes under CFL conditions, where well-posedness is controlled by the time step and the mesh (see again \cite{LeVeque, GodlewskiRaviart}).
\end{itemize}
Under Remark \ref{rem_thm_approx_abstract}, Theorem \ref{thm_approx_abstract} itself is logically weaker, as it only requires the existence of a discrete trajectory staying in the relevant ball, and consistency of $Q_N$ on the output class $Y$.
This distinction may be important in practice: the well-posedness of the discrete scheme is often proved by arguments specific to the chosen discretization, whereas the error estimate itself only needs a stable lifted trajectory.

For wave-type problems with state-dependent positivity constraints, even discrete well-posedness may require stronger, mesh-dependent norms and a bootstrap argument on the numerical solution.
This is the strategy used, for instance, in \cite{Dorich_FoCM2025}, and in a different time-discretization setting in \cite{DorichHochbruck_SINUM2022}.
This observation supports the present separation between the strong sufficient criterion \ref{H_Aepsilon_N} and the logically weaker viewpoint emphasized in Remark \ref{rem_thm_approx_abstract}.

\subsection{On the intertwining assumption}\label{subsec_discussion_intertwining}
The intertwining assumption \ref{H_intertwining} encodes the propagation of $Z$-regularity along the flow and is standard in Kato's theory for quasilinear evolution equations. For the \emph{original} operator $A$, it amounts to the boundedness on $X$ of $B = [S,A] S^{-1}$, which holds in all the examples of Section \ref{subsec_examples}: it is the order-zero commutator for first-order operators, and it even vanishes for operators that are functions of the reference operator defining $S$ (such as the Dirichlet Laplacian). For the \emph{regularized} family $A_\varepsilon = \mathcal{H}_\varepsilon^* A \mathcal{H}_\varepsilon$, the relevant object is instead $[S, A_\varepsilon]$, whose delicate part is not $[S,A]$ but the commutator $[S, \mathcal{H}_\varepsilon]$ of $S$ with the variable-step mollifier; on a domain with boundary, this commutator is not uniformly bounded on $X$, which is why we keep \ref{H_intertwining} structural for the mollified family. It does become automatic whenever the regularization commutes with $S$, as in the two fully closed cases of Section \ref{subsubsec_closed} (periodic conditions, where $\mathcal{H}_\varepsilon$ is a Fourier multiplier, and orthogonal spectral truncation, where the regularizing projector commutes with $S$). Developing more systematic verification criteria for the mollified family, especially for higher-order operators with boundary conditions, is a natural direction for future work.

\subsection{Beyond fixed sampling grids}\label{subsec_discussion_moving_particles}
The discretization framework is formulated for fixed operators $(P_N,R_N)$. Extending it to time-dependent sampling (moving particles, adaptive grids) would require a refined analysis of time-dependent reconstruction errors and stability in the spirit of nonautonomous evolution systems. Such extensions would connect the present setting to genuinely Lagrangian particle methods and to adaptive meshfree approximations, see e.g. \cite{HockneyEastwood,Monaghan1992,Wendland2005} for classical paradigms.

\subsection{Beyond semigroups}
The semigroup property enters the paper only through the first approximation step, namely through Proposition \ref{prop_yepsilon}, where it is used to obtain the Duhamel formula and the quantitative control of $y_\varepsilon-y$.
If this estimate can be obtained by other means for a given problem, then the rest of the framework still applies.
For instance, on very restricted classes of data, backward heat evolutions may be handled by explicit spectral or analytic arguments, but such situations are strongly ill-posed in general and lie outside the robust semigroup setting considered here.
Likewise, some equations that are nonlinear in derivatives can be recast in a quasilinear form after an algebraic rewriting or an integration-by-parts argument, as in Burgers-type models. However, genuinely nonlinear dependence on higher-order derivatives usually falls outside the present abstract framework and would require a different structural analysis.

\paragraph{Acknowledgment.}
We are indebted to Claude Bardos, Julien Barr\'e, Arnaud Debussche, Nicolas Fournier, Isabelle Gallagher, Thierry Gallay, Alain Joye, Beno\^{\i}t Perthame and Laure Saint-Raymond for useful discussions.

\appendix

\section{Appendix: variable-step mollifiers}\label{app_convolution}
As in Sections \ref{sec_general_class_quasilinear} and \ref{sec_main_result}, let $\Omega$ be the compact closure of a bounded open subset of $\R^n$ with a Lipschitz boundary, endowed with the induced Euclidean distance, of Lebesgue volume $\vert\Omega\vert=1$. We denote by $\mathring{\Omega}$ the interior of $\Omega$.
The assumption $\vert\Omega\vert=1$ is not a limitation; it simplifies the writing by avoiding to have the factor $\vert\Omega\vert$ in the estimates.

In this appendix, we recall the concept of variable-step mollifier, developed by Burenkov in \cite{Burenkov_1982, Burenkov_1998} and studied recently in \cite{Hintermuller}. The main feature of this modification of the usual convolution is to provide an operator that smoothens functions $g\in L^1(\Omega)$ and preserves the values of $g$ and of its derivatives at the boundary when $g$ is sufficiently regular.

\subsection{Setting and definition}\label{app_convolution_setting}
Let $q\in\N^*$ be fixed.
Let $\eta\in \mathscr{C}^\infty_c(\R^n)$ be a nonnegative smooth real-valued function on $\R^n$, supported in the unit Euclidean ball $\mathsf{B}_1=B_{\R^n}(0,1)$ of $\R^n$, and such that $\int_{\R^n}\eta(x)\, dx=1$.%
\footnote{For example, 
$$
\eta(x) = \left\{\begin{array}{ll} 
c\, e^{1/(\Vert x\Vert^2-1)} & \textrm{if}\ \Vert x\Vert < 1, \\
0 & \textrm{otherwise} ,
\end{array}\right.
$$
where $c>0$ is a normalization constant. 
}
For any $\varepsilon\in(0,1]$, we define $\eta_\varepsilon\in \mathscr{C}^\infty_c(\R^n)$  by
$$
\eta_\varepsilon(x) = \frac{1}{\varepsilon^n}\eta\left(\frac{x}{\varepsilon}\right) \qquad\forall x\in\R^n .
$$
Let $\rho\in \mathscr{C}^q(\Omega)$ be a nonnegative real-valued function on $\Omega$ such that $\rho>0$ in $\mathring{\Omega}$ and
\begin{equation}\label{def_rho}
\rho(x)\leq \mathrm{d}_{\R^n}(x,\partial\Omega) \qquad \forall x\in\Omega ,
\end{equation}
and such that all derivatives of $\rho$ of order $\leq q-1$ vanish on $\partial\Omega$,%
\footnote{For example, one can take $\rho(x) = C\, \mathrm{d}_{\R^n}(\cdot,\partial\Omega)^q$ for some adequate $C\in(0,1]$ when $\partial\Omega$ is sufficiently regular, and when $\partial\Omega$ is only Lipschitz one can take $\rho(x)=C\delta(x)^q$ where $\delta\in \mathscr{C}^q(\Omega)$ is a smooth regularized distance satisfying $c_1\,\mathrm{d}_{\R^n}(x,\partial\Omega)\leq \delta(x)\leq c_2\,\mathrm{d}_{\R^n}(x,\partial\Omega)$ on $\Omega$. 
}
When $\partial\Omega=\emptyset$, one may simply take $\rho(x)=1$. 
We set $L_\rho=\Vert\nabla\rho\Vert_{L^\infty}$. 

\begin{definition}\label{def_H_eps}
For any $\varepsilon\in(0,1]$, any $g\in L^1(\Omega)$ and any $x\in\mathring{\Omega}$, we define
\begin{align}
\mathcal{H}_\varepsilon g(x) 
&= \int_{B(0,\varepsilon)} \eta_\varepsilon(x') g(x-\rho(x)x')\, dx' 
= \int_{\mathsf{B}_1} \eta(x') g(x-\varepsilon\rho(x)x')\, dx'  \label{def_Teps1} \\
&= \int_{B(x,\varepsilon\rho(x))} H_\varepsilon(x,x') g(x')\, dx' , \qquad H_\varepsilon(x,x') = \frac{1}{\rho(x)^n}\eta_\varepsilon\!\left( \frac{x-x'}{\rho(x)} \right)  \label{def_Teps2} .
\end{align}
\end{definition}

In \eqref{def_Teps1} and \eqref{def_Teps2}, the domain of integration can be replaced by $\Omega$ because $x-\rho(x)x'\in B(x,\rho(x))\subset\Omega$ (by using \eqref{def_rho}). 

Clearly, $\mathcal{H}_\varepsilon g\in\mathscr{C}^\infty(\mathring{\Omega})$ because $H_\varepsilon$ is smooth with respect to $x\in\mathring{\Omega}$, i.e., \eqref{def_Teps2} defines a linear operator $\mathcal{H}_\varepsilon:L^1(\Omega)\rightarrow\mathscr{C}^\infty(\mathring{\Omega})$ of kernel $H_\varepsilon$. Note that $\mathcal{H}_\varepsilon$ is not symmetric in general. 

When $g\in L^1(\Omega,\R^d)$ for $d\in\N^*$, one defines $\mathcal{H}_\varepsilon g$ by applying $\mathcal{H}_\varepsilon$ to each component of $g$, so we next keep $g$ with values in $\R$ to describe properties of $\mathcal{H}_\varepsilon$.


\begin{lemma}\label{lem_preserv_value_boundary}
If $g\in\mathscr{C}^0(\Omega)$ then the definition \eqref{def_Teps1} can be extended to $x\in\partial\Omega$, and we have $\mathcal{H}_\varepsilon g\in\mathscr{C}^0(\Omega)$ and $\mathcal{H}_\varepsilon g(x) = g(x)$ for every $x\in\partial\Omega$.
\end{lemma}

\begin{proof}
Indeed, if $x\in\partial\Omega$ then $\rho(x)=0$ and thus, by \eqref{def_Teps1}, $\mathcal{H}_\varepsilon g(x) = \int_{B(0,\varepsilon)} \eta_\varepsilon(x') g(x)\, dx' = \big(\int_{\R^n}\eta(x')\, dx' \big) g(x) = g(x)$.
\end{proof}

When $g$ only belongs to $L^1(\Omega)$, no boundary value is canonically defined and the identity of the above lemma is not asserted. We will generalize Lemma \ref{lem_preserv_value_boundary} further (see Corollary \ref{cor_preserv_deriv_boundary}).

\subsection{\texorpdfstring{$L^r$}{Lr} bounds}\label{app_Lr-bounds}

\begin{lemma}\label{lem_Lr-bounds}
For every $r\in[1,+\infty]$ and every $\varepsilon\in(0,\frac{1}{1+L_\rho}]$, the operator $\mathcal{H}_\varepsilon$ extends boundedly to $L(L^r(\Omega))$, 
and
\begin{equation}\label{eq:Lq-bound-Schur}
\Vert \mathcal{H}_\varepsilon\Vert_{L(L^r(\Omega))} \leq \left(\Vert\eta\Vert_{L^\infty}\,\vert \mathsf{B}_1\vert\,\left(\frac{1+\varepsilon L_\rho}{1-\varepsilon L_\rho}\right)^n\right)^{\frac{1}{r}},
\end{equation}
with the convention $1/r=0$ when $r=+\infty$.
As a consequence, assuming moreover that $\varepsilon L_\rho\leq \frac{1}{2}$, we have 
\begin{equation}\label{eq:Lq-bound-Schur_not_eps}
\Vert \mathcal{H}_\varepsilon\Vert_{L(L^r(\Omega))}\leq \left(\Vert\eta\Vert_{L^\infty}\,\vert \mathsf{B}_1\vert\,3^n \right)^{1/r} .
\end{equation}
\end{lemma}

\begin{proof}
Noting that $H_\varepsilon\geq 0$, the general Schur test recalled in Lemma \ref{lem:Schur-general} (Appendix \ref{app_Schur}) gives
\begin{equation*}
\Vert \mathcal{H}_\varepsilon\Vert_{L(L^r(\Omega))} \leq 
\left(\sup_{x}\int_\Omega H_\varepsilon(x,x')\,dx'\right)^{1-\frac{1}{r}}
\left(\sup_{x'}\int_\Omega H_\varepsilon(x,x')\,dx\right)^{\frac{1}{r}}.
\end{equation*}
The first term at the right-hand side is equal to $1$ because $\int_\Omega H_\varepsilon(x,x')\,dx'=1$ for any $x$.
To estimate the second term, we first note that, for any fixed $x'\in\Omega$, if $H_\varepsilon(x,x')\neq 0$ then $\Vert x-x'\Vert_{\R^n}\leq \varepsilon \rho(x)$ and
\begin{equation}\label{eq:rho-comparability}
\frac{1}{1+\varepsilon L_\rho}\,\rho(x') \leq \rho(x) \leq \frac{1}{1-\varepsilon L_\rho}\,\rho(x') .
\end{equation}
(Note that $\varepsilon L_\rho<\varepsilon(1+L_\rho)\leq 1$ by assumption). 
Indeed, by Lipschitz continuity of $\rho$, we have $\vert\rho(x)-\rho(x')\vert \leq L_\rho\Vert x-x'\Vert_{\R^n} \leq \varepsilon L_\rho\,\rho(x)$, hence $\rho(x')\leq \rho(x)+\varepsilon L_\rho\,\rho(x)=(1+\varepsilon L_\rho)\rho(x)$ and $\rho(x')\geq \rho(x)-\varepsilon L_\rho\,\rho(x)=(1-\varepsilon L_\rho)\rho(x)$, and \eqref{eq:rho-comparability} follows. 
In particular, using that $\Vert x-x'\Vert_{\R^n}\leq \varepsilon\rho(x)\leq \varepsilon\rho(x')/(1-\varepsilon L_\rho)$, it follows that
$$
\{x\in\Omega\ \mid\ H_\varepsilon(x,x')\neq 0\}\subset B\Big(x',\frac{\varepsilon\,\rho(x')}{1-\varepsilon L_\rho}\Big) \subset \Omega ,
$$
where, to obtain the latter inclusion, we have used that $\frac{\varepsilon}{1-\varepsilon L_\rho}\leq 1$ because $\varepsilon\leq\frac{1}{1+L_\rho}$.
Hence, using that $\rho(x)\geq \rho(x')/(1+\varepsilon L_\rho)$ by \eqref{eq:rho-comparability}, we get
\begin{multline*}
\int_\Omega H_\varepsilon(x,x')\,dx
\leq \Vert\eta\Vert_{L^\infty} \int_{B\big(x',\frac{\varepsilon\rho(x')}{1-\varepsilon L_\rho}\big)} \frac{1}{\big(\varepsilon\rho(x)\big)^n} \,dx \\
\leq \Vert\eta\Vert_{L^\infty}\,\frac{(1+\varepsilon L_\rho)^n}{\big(\varepsilon\rho(x')\big)^n}\,\left\vert B\big(x',\tfrac{\varepsilon\rho(x')}{1-\varepsilon L_\rho}\big)\right\vert
= \Vert\eta\Vert_{L^\infty}\,\vert \mathsf{B}_1\vert\,\left(\frac{1+\varepsilon L_\rho}{1-\varepsilon L_\rho}\right)^n.
\end{multline*}
The estimate \eqref{eq:Lq-bound-Schur} follows.
To obtain \eqref{eq:Lq-bound-Schur_not_eps} under the additional assumption $\varepsilon L_\rho\leq \frac{1}{2}$, it suffices to note that $\big(\frac{1+\varepsilon L_\rho}{1-\varepsilon L_\rho}\big)^n\leq 3^n$.
\end{proof}

\begin{remark}
The fact that $\mathcal{H}_\varepsilon\in L(L^r(\Omega))$ is proved in \cite{Hintermuller} by using the Hardy-Littlewood maximal inequality and the Marcinkiewicz theorem. Our proof above is much simpler.
\end{remark}

\subsection{Adjoint in \texorpdfstring{$L^2(\Omega)$}{L2(Omega)}}\label{app_adjoint_H_eps}
As a particular case of Lemma \ref{lem_Lr-bounds}, $\mathcal{H}_\varepsilon\in L(L^2(\Omega))$ for any $\varepsilon\in(0,\frac{1}{1+L_\rho}]$.
The lemma below then follows from \eqref{def_Teps2}.

\begin{lemma}\label{lem_adjoint_H_eps}
For any $\varepsilon\in(0,\frac{1}{1+L_\rho}]$, the adjoint of $\mathcal{H}_\varepsilon$ in $L^2(\Omega)$ is given by
$$
\mathcal{H}_\varepsilon^*g(x) = \int_\Omega H_\varepsilon(x',x) g(x')\, dx' 
=\int_\Omega \frac{1}{\rho(x')^n} \eta_\varepsilon\!\left(\frac{x'-x}{\rho(x')}\right) g(x')\,dx'
\qquad \forall g\in L^2(\Omega). 
$$
\end{lemma}

\begin{lemma}\label{lem_adjoint_Lip}
Assume $q\geq 2$, so that $\rho\in\mathscr{C}^2(\Omega)$. For any $\varepsilon\in(0,\inf(\frac{1}{1+L_\rho},\frac{1}{2L_\rho})]$, the adjoint operator $\mathcal{H}_\varepsilon^*$ maps $W^{1,\infty}(\Omega)$ into itself. Moreover, there exists $C>0$, depending only on $\eta$ and on $\rho$ through $\Vert \nabla\rho\Vert_{L^\infty}$ and $\Vert D^2\rho\Vert_{L^\infty}$, such that
\begin{equation}\label{Hepsstar_W1infty}
\Vert \mathcal{H}_\varepsilon^* g\Vert_{W^{1,\infty}}\leq C \Vert g\Vert_{W^{1,\infty}}
\qquad
\forall g\in W^{1,\infty}(\Omega).
\end{equation}
\end{lemma}

\begin{proof}
Fix $z\in \mathsf{B}_1$ and define $\Phi_z(x)=x-\varepsilon \rho(x) z$ for every $x\in\Omega$.
Since $D\Phi_z(x)=\mathrm{id}-\varepsilon z\otimes \nabla\rho(x)$, one has
$\Vert D\Phi_z(x)-\mathrm{id}\Vert\leq \varepsilon \Vert \nabla\rho\Vert_{L^\infty}\Vert z\Vert\leq \varepsilon L_\rho\leq \frac{1}{2}$.
Hence $\Phi_z$ is locally bi-Lipschitz on $\Omega$, with
$\Vert D\Phi_z^{-1}\Vert_{L^\infty}\leq (1-\varepsilon L_\rho)^{-1}\leq 2$.
Moreover $\Phi_z$ is a bi-Lipschitz homeomorphism of $\Omega$ onto $\Omega$: since $\rho$ vanishes on $\partial\Omega$, one has $\Phi_z=\mathrm{id}$ on $\partial\Omega$, and $\Phi_z$ is a $\frac{1}{2}$-Lipschitz perturbation of the identity, so it is injective and proper; being the identity on the boundary, its degree equals that of the identity, hence it is also surjective onto $\Omega$. Let $J_z(y)=\det(D\Phi_z^{-1}(y))$.
A change of variables in the expression of $(\mathcal{H}_\varepsilon^* g)(x)$ shows that, for almost every $y\in\Omega$,
\begin{equation}\label{Hepsstar_changevar}
(\mathcal{H}_\varepsilon^* g)(y)
=
\int_{\mathsf{B}_1} \eta(z) g\big(\Phi_z^{-1}(y)\big) J_z(y)\,dz.
\end{equation}
Since $\Vert J_z\Vert_{L^\infty}$ is uniformly bounded, \eqref{Hepsstar_changevar} immediately gives
$\Vert \mathcal{H}_\varepsilon^* g\Vert_{L^\infty}\leq C \Vert g\Vert_{L^\infty}$.
Let $y_1,y_2\in\Omega$. Using \eqref{Hepsstar_changevar}, we write
\begin{multline*}
\vert \mathcal{H}_\varepsilon^* g(y_1)-\mathcal{H}_\varepsilon^* g(y_2)\vert
\leq \int_{\mathsf{B}_1} \eta(z) \vert g\big(\Phi_z^{-1}(y_1)\big)-g\big(\Phi_z^{-1}(y_2)\big)\vert \vert J_z(y_1)\vert\,dz \\
+ \int_{\mathsf{B}_1} \eta(z) \vert g\big(\Phi_z^{-1}(y_2)\big)\vert \vert J_z(y_1)-J_z(y_2)\vert\,dz.
\end{multline*}
Because $\Phi_z^{-1}$ is uniformly Lipschitz, the first term is bounded by $C \Lip(g) \vert y_1-y_2\vert$.%
\footnote{We denote by $\Lip(\Omega)$ the space of Lipschitz functions $g:\Omega\to\R$, with
$$
\mathrm{Lip}(g) = \sup \left\{ \frac{\vert g(x)-g(x')\vert}{\Vert x-x'\Vert_{\R^n}} \ \mid\ x,x'\in\Omega, \ x\neq x' \right\} 
$$}
For the second term, note that $D\Phi_z(x)$ depends linearly on $\nabla\rho(x)$, hence $D\Phi_z^{-1}(y)$ depends smoothly on $D\Phi_z$, and therefore $J_z$ is Lipschitz with $\Lip(J_z)\leq C \varepsilon \Vert D^2\rho\Vert_{L^\infty}\leq C$, uniformly in $z\in\mathsf{B}_1$ and $\varepsilon\in(0,1]$. Thus the second term is bounded by $C \Vert g\Vert_{L^\infty} \vert y_1-y_2\vert$.
Hence $\Lip(\mathcal{H}_\varepsilon^* g) \leq C (\Vert g\Vert_{L^\infty}+\Lip(g)) \leq C \Vert g\Vert_{W^{1,\infty}}$, which gives \eqref{Hepsstar_W1infty}.
\end{proof}

\begin{lemma}\label{lem_CV_Hepsstar}
There exists $C_{\eta,\rho}>0$ such that, for every $\varepsilon\in(0,\inf(\frac{1}{1+L_\rho},\frac{1}{2L_\rho})]$ and every $g\in W^{1,2}(\Omega)$, 
\begin{equation}\label{Hepsstar_L2_rate}
\Vert \mathcal{H}_\varepsilon^* g-g\Vert_{L^2}
\leq C_{\eta,\rho}\, \varepsilon \Vert g\Vert_{W^{1,2}(\Omega)} .
\end{equation}
\end{lemma}

\begin{proof}
Using \eqref{Hepsstar_changevar}, we write $\mathcal{H}_\varepsilon^* g-g=I_1+I_2$ with
$$
I_1(y) = \int_{\mathsf{B}_1} \eta(z) \big(g(\Phi_z^{-1}(y))-g(y)\big) J_z(y)\,dz, \qquad
I_2(y) = g(y) \int_{\mathsf{B}_1} \eta(z) \big(J_z(y)-1\big)\,dz.
$$
Since $D\Phi_z^{-1}=\mathrm{id}+\mathrm{O}(\varepsilon)$ uniformly, one has
$\Vert J_z-1\Vert_{L^\infty}\leq C \varepsilon$ for every $z\in\mathsf{B}_1$.
Therefore $\Vert I_2\Vert_{L^2}\leq C \varepsilon \Vert g\Vert_{L^2}$.
Let us estimate $I_1$. By Minkowski's inequality and the uniform $L^\infty$ bound on $J_z$, it is enough to show that
\begin{equation}\label{comp_diff_Psi}
\Vert g\circ \Phi_z^{-1}-g\Vert_{L^2}
\leq C \varepsilon \Vert \nabla g\Vert_{L^2}
\qquad
\forall z\in\mathsf{B}_1.
\end{equation}
Fix such a $z$ and set $\Psi=\Phi_z^{-1}$ and $h=\Psi-\mathrm{id}$.
Since $\Psi$ is uniformly bi-Lipschitz and $\Psi(y)-y=\varepsilon \rho(\Psi(y)) z$, one has
$\Vert h\Vert_{L^\infty(\Omega)}\leq C \varepsilon$ and $\Vert D\Psi-\mathrm{id}\Vert_{L^\infty(\Omega)}\leq C \varepsilon$.
To handle a possibly non-convex Lipschitz domain, where the segment between $y$ and $\Psi(y)$ may leave $\Omega$, we first extend $g$ to a function $Eg\in W^{1,2}(\R^n)$ by a bounded Sobolev extension operator $E$, with $\Vert Eg\Vert_{W^{1,2}(\R^n)}\leq C \Vert g\Vert_{W^{1,2}(\Omega)}$, and we run the estimate for $Eg$ on $\R^n$ before restricting to $\Omega$.
For smooth $Eg$, writing $Eg(\Psi(y))-Eg(y) = \int_0^1 \langle \nabla Eg(y+\tau h(y)),h(y)\rangle_{\R^n}\,d\tau$ where $F_\tau=\mathrm{id}+\tau h=(1-\tau)\mathrm{id}+\tau \Psi$,
we get
$$
\Vert Eg\circ \Psi-Eg\Vert_{L^2(\Omega)}^2 \leq \Vert h\Vert_{L^\infty}^2 \int_0^1 \int_\Omega \big\vert \nabla Eg\big(F_\tau(y)\big)\big\vert^2\,dy\,d\tau.
$$
Because $\Vert D\Psi-\mathrm{id}\Vert_{L^\infty}\leq C \varepsilon$ and $\varepsilon$ is small, each $F_\tau$ is uniformly bi-Lipschitz, and its Jacobian determinant is bounded above and below away from $0$, uniformly in $\tau\in[0,1]$; moreover $F_\tau(\Omega)\subset\R^n$ remains in a fixed bounded neighborhood of $\Omega$. A change of variables therefore gives
$\int_\Omega \big\vert \nabla Eg\big(F_\tau(y)\big)\big\vert^2\,dy \leq C \Vert \nabla Eg\Vert_{L^2(\R^n)}^2\leq C \Vert g\Vert_{W^{1,2}(\Omega)}^2$.
Since $\Vert h\Vert_{L^\infty}\leq C \varepsilon$ and $g=Eg$ on $\Omega$, this proves \eqref{comp_diff_Psi} for smooth $g$. For a general $g\in W^{1,2}(\Omega)$, choose $g_j\in C^\infty(\overline\Omega)$ with $g_j\to g$ in $W^{1,2}(\Omega)$; the composition operators $u\mapsto u\circ\Psi$ are uniformly bounded on $L^2(\Omega)$ because $\Psi$ is uniformly bi-Lipschitz, so passing to the limit in the estimate for $g_j$ yields \eqref{comp_diff_Psi} for $g$.
We conclude that $\Vert I_1\Vert_{L^2}\leq C \varepsilon \Vert \nabla g\Vert_{L^2}$.
Combining the estimates on $I_1$ and $I_2$ proves \eqref{Hepsstar_L2_rate}.
\end{proof}

\subsection{Derivatives}
Assume that $g\in W^{1,1}(\Omega)$. 
Using \eqref{def_Teps1} and \eqref{def_Teps2}, we have
\begin{align}
\partial_i\mathcal{H}_\varepsilon g(x) 
&= \mathcal{H}_\varepsilon \partial_i g(x) - \varepsilon \int_{\mathsf{B}_1} \eta(x') \langle \nabla g(x-\varepsilon\rho(x)x'),x' \rangle\, dx'\ \partial_i\rho(x)  \label{grad_Teps1} \\
&= \int_\Omega \partial_i H_\varepsilon(x,x') g(x')\, dx'  \label{grad_Teps2}
\end{align}
for every $x\in\mathring{\Omega}$, every $i\in\{1,\ldots,n\}$ and every $\varepsilon\in(0,1]$. In \eqref{grad_Teps2}, $\partial_i$ acts on $x$, and $g\in L^1(\Omega)$ is enough to have \eqref{grad_Teps2}.
Using \eqref{grad_Teps1}, we see that if $\rho$ and $\nabla\rho$ vanish along $\partial\Omega$ and if $g\in\mathscr{C}^1(\Omega)$ (up to the boundary) then $\mathcal{H}_\varepsilon g\in\mathscr{C}^1(\Omega)$ (up to the boundary) and $\partial_i\mathcal{H}_\varepsilon g(x)=\partial_i g(x)$ for every $x\in\partial\Omega$. 
Moreover, $\partial_i\mathcal{H}_\varepsilon g = \mathcal{H}_\varepsilon\partial_i g + \mathrm{O}(\varepsilon) \Vert g\Vert_{W^{1,\infty}}$, where the $\mathrm{O}(\varepsilon)$ only depends on $\eta$ and $\rho$, if $g\in W^{1,\infty}(\Omega)$. 
More generally, we have the following lemma.

\begin{lemma}\label{lem_commutator}
Let $\alpha=(\alpha_1,\ldots,\alpha_n)\in\N^n$ with $\vert\alpha\vert=m\in\{1,\ldots,q\}$.
For any $\varepsilon\in(0,1]$ and any $g\in L^1(\Omega)$, we have $D^\alpha\mathcal{H}_\varepsilon g\in\mathscr{C}^\infty(\mathring{\Omega})$ and
$$
D^\alpha\mathcal{H}_\varepsilon g(x) = \int_\Omega D^\alpha_x H_\varepsilon(x,x') g(x')\, dx'  \qquad\forall x\in\mathring{\Omega}.
$$
If moreover $g\in W^{m,r}(\Omega)$ for some $r\in[1,+\infty]$, then, for any $\varepsilon\in(0,\frac{1}{1+L_\rho}]$, $D^\alpha\mathcal{H}_\varepsilon g\in L^r(\Omega)$ and
\begin{equation}\label{crochetDalphaHeps}
D^\alpha\mathcal{H}_\varepsilon g = \mathcal{H}_\varepsilon D^\alpha g + [D^\alpha,\mathcal{H}_\varepsilon]g ,
\end{equation}
where the function $[D^\alpha,\mathcal{H}_\varepsilon]g \in L^r(\Omega)$ is given by%
\footnote{Alternatively, we have $[D^\alpha,\mathcal{H}_\varepsilon]g(x)=\int_\Omega \big(D_x^\alpha H_\varepsilon(x,x') - (-1)^{\vert\alpha\vert}D_{x'}^\alpha H_\varepsilon(x,x')\big)\,g(x')\,dx'$, but this is not useful. 
}
\begin{equation}\label{eq:comm-structure}
[D^\alpha,\mathcal{H}_\varepsilon]g(x)
= \sum_{k=1}^{m}\varepsilon^k \sum_{\vert\gamma\vert\leq m} P_{\alpha,k}\big(D\rho(x),\ldots,D^m\rho(x)\big) \int_{\mathsf{B}_1}\eta(x')\,\Vert x'\Vert^k (D^\gamma g) (x-\varepsilon\rho(x)x')\,dx'
\end{equation}
for every $x\in\Omega$, where the $P_{\alpha,k}$ are polynomials in the derivatives $D^\beta\rho(x)$ ($1\leq\vert\beta\vert\leq m$) and moments of $\eta$, satisfying $P_{\alpha,k}(0,\ldots,0)=0$.
Assuming moreover that $\varepsilon L_\rho\leq \frac{1}{2}$, we have
\begin{equation}\label{eq:comm-Lp}
\Vert [D^\alpha,\mathcal{H}_\varepsilon]g\Vert_{L^r(\Omega)}
\leq C_{\alpha,\eta} \sum_{k=1}^{m}\varepsilon^k \sum_{\vert\gamma\vert\leq m}\Vert D^\gamma g\Vert_{L^r(\Omega)}
\leq C_{\alpha,\eta,\rho}\,\varepsilon\,\Vert g\Vert_{W^{m,r}(\Omega)}
\end{equation}
with $C_{\alpha,\eta,\rho}$ depending on $n$, on the moments $\int_{\mathsf{B}_1}\eta(x')\Vert x'\Vert^k\,dx'$ for $1\leq k\leq m$, on $\Vert\eta\Vert_{L^\infty}$, and on $\sup_{1\leq\vert\beta\vert\leq m}\Vert D^\beta\rho\Vert_{L^\infty}$, but not depending on $\varepsilon$ nor on $g$. 
\end{lemma}


\begin{proof}
Let us first prove \eqref{eq:comm-structure} by induction on $m=\vert\alpha\vert$.
The case $m=1$ follows from \eqref{grad_Teps1}.
Now, assume that \eqref{eq:comm-structure} holds for all multi-indices of length $\leq m-1$. Fix $\alpha$ with $\vert\alpha\vert=m$ and write $D^\alpha=\partial_j D^{\alpha-e_j}$ for some $j$, where $e_j=(0,\ldots,1,\ldots,0)$ (with the $1$ at the $j^\mathrm{th}$ position). Then
\begin{equation}\label{eq:three-terms}
[D^\alpha,\mathcal{H}_\varepsilon]g = [\partial_j,\mathcal{H}_\varepsilon]\,D^{\alpha-e_j}g + \partial_j\left([D^{\alpha-e_j},\mathcal{H}_\varepsilon]g\right)-[D^{\alpha-e_j},\mathcal{H}_\varepsilon]\partial_j g.
\end{equation}
For the first term, apply the base case with $g$ replaced by $D^{\alpha-e_j}g$:
$$
[\partial_j,\mathcal{H}_\varepsilon]\,D^{\alpha-e_j}g(x) = - \varepsilon(\partial_j\rho)(x)\int_{\mathsf{B}_1}\eta(x')\big\langle \nabla D^{\alpha-e_j}g\big(x-\varepsilon\rho(x)x'\big),x'\big\rangle \,dx',
$$
which matches \eqref{eq:comm-structure} with $k=1$ and $\vert\gamma\vert=m$.
For the second and third terms, we apply the induction hypothesis to $[D^{\alpha-e_j},\mathcal{H}_\varepsilon]h$ with $h=g$ and with $h=\partial_j g$. 
Differentiating the first with respect to $x_j$, the derivative hits either the coefficient
$P_{\alpha-e_j,k}\big(D\rho(x),\ldots,D^{m-1}\rho(x)\big)$, producing the same structure with derivatives of $\rho$ up to order $m$, or the integrand. When it hits the integrand, the chain rule applied to
$x\mapsto (D^\gamma g) (x-\varepsilon\rho(x)x')$ yields
$$
\partial_j\big( D^\gamma g (x-\varepsilon\rho(x)x')\big)
= D^{\gamma+e_j}g (x-\varepsilon\rho(x)x') - \varepsilon (\partial_j\rho)(x) \big\langle \nabla D^\gamma g (x-\varepsilon\rho(x)x'), x'\big\rangle .
$$
The piece with $D^{\gamma+e_j}g$ cancels exactly with the corresponding term in $[D^{\alpha-e_j},\mathcal{H}_\varepsilon]\partial_j g$ in \eqref{eq:three-terms}. The remaining pieces all carry at least one explicit factor $\varepsilon$, and have the form
$$
\varepsilon^k\,\widetilde P_{\alpha,k}\big(D\rho(x),\ldots,D^{m}\rho(x)\big)\,
\int_{\mathsf{B}_1}\eta(x') \Vert x'\Vert^k (D^\gamma g) (x-\varepsilon\rho(x)x')\,dx',
$$
with $1\leq k\leq m$ and $\vert\gamma\vert\leq m$. Collecting all contributions gives \eqref{eq:comm-structure}.

Let us now prove \eqref{eq:comm-Lp}. Each integral on the right-hand side of \eqref{eq:comm-structure} can be written as $T_{\varepsilon,k}(D^\gamma g)$, where $T_{\varepsilon,k}$ is defined by \eqref{def_Ts} in Lemma \ref{lem:schur-weight} hereafter, and $T_{\varepsilon,k}\in L(L^r(\Omega))$ with a norm bounded uniformly for $\varepsilon$ small enough. Bounding the polynomial coefficients $P_{\alpha,k}$ by $\sup_{1\leq \vert\beta\vert\leq m}\Vert D^\beta\rho\Vert_{L^\infty}$ 
gives the result.
\end{proof}

\begin{lemma}\label{lem:schur-weight}
For any $\varepsilon\in(0,1]$, any $s\geq 0$ and any $h\in L^1(\Omega)$, define
\begin{equation}\label{def_Ts}
(T_{\varepsilon,s}h)(x) = \int_{\mathsf{B}_1}\eta(x') \Vert x'\Vert^s h (x-\varepsilon\rho(x)x')\,dx'.
\end{equation}
Then $T_{\varepsilon,s}\in L(L^r(\Omega))$ for any $r\in[1,+\infty]$ and any $\varepsilon\in(0,\frac{1}{1+L_\rho}]$, and we have
\begin{equation}\label{eq:TsLp}
\Vert T_{\varepsilon,s}\Vert_{L(L^r(\Omega))}
\leq \left(\int_{\mathsf{B}_1}\eta(x') \Vert x'\Vert^s\,dx'\right)^{1-\frac{1}{r}}
\left(\Vert\eta\Vert_{L^\infty} \vert \mathsf{B}_1\vert \Big(\frac{1+\varepsilon L_\rho}{1-\varepsilon L_\rho}\Big)^{n+s}\right)^{\frac{1}{r}},
\end{equation}
with the convention $1/r=0$ for $r=+\infty$. As a consequence, assuming moreover that $\varepsilon L_\rho\leq \frac{1}{2}$, 
\begin{equation}\label{schur-weight_no_eps}
\Vert T_{\varepsilon,s}\Vert_{L(L^r(\Omega))}\leq \left(\int_{\mathsf{B}_1}\eta(x') \Vert x'\Vert^s\,dx' \right)^{1-\frac{1}{r}} \left( \Vert\eta\Vert_{L^\infty} \vert \mathsf{B}_1\vert 3^{n+s} \right)^{\frac{1}{r}} .
\end{equation}
\end{lemma}

\begin{proof}
We write $T_{\varepsilon,s}$ as an operator with kernel
$$
K_{\varepsilon,s}(x,x')=\frac{1}{\big(\varepsilon\rho(x)\big)^n}\,\eta\Big(\frac{x-x'}{\varepsilon\rho(x)}\Big)\,\Big\Vert\frac{x-x'}{\varepsilon\rho(x)}\Big\Vert^s \geq 0 .
$$
For $x$ fixed, the change of variables $x'=x-\varepsilon\rho(x)z$ gives $\int_\Omega K_{\varepsilon,s}(x,x')\,dx'=\int_{\mathsf{B}_1}\eta(x') \Vert x'\Vert^s\,dx'$.
For $x'$ fixed, we follow the argument developed in the proof of Lemma \ref{lem_Lr-bounds}: if $K_{\varepsilon,s}(x,x')\neq 0$ then $\Vert x-x'\Vert_{\R^n}\leq\varepsilon\rho(x)$ and we have \eqref{eq:rho-comparability}, hence $x$ ranges in $B\big(x',\frac{\varepsilon\rho(x')}{1-\varepsilon L_\rho}\big)$ and $\rho(x)\geq \rho(x')/(1+\varepsilon L_\rho)$ there. Therefore
$$
\int_\Omega K_{\varepsilon,s}(x,x')\,dx \leq \Vert\eta\Vert_{L^\infty}\,\frac{(1+\varepsilon L_\rho)^{n+s}}{\big(\varepsilon\rho(x')\big)^n}\,
\big\vert B\big(x',\tfrac{\varepsilon\rho(x')}{1-\varepsilon L_\rho}\big)\big\vert
=\Vert\eta\Vert_{L^\infty} \vert \mathsf{B}_1\vert \Big(\frac{1+\varepsilon L_\rho}{1-\varepsilon L_\rho}\Big)^{n+s}.
$$
Lemma \ref{lem:Schur-general} in Appendix \ref{app_Schur} yields \eqref{eq:TsLp}.
Assuming moreover that $\varepsilon L_\rho\leq \frac{1}{2}$, we have $\big(\frac{1+\varepsilon L_\rho}{1-\varepsilon L_\rho}\big)^{n+s}\leq 3^{n+s}$, and \eqref{schur-weight_no_eps} follows.
\end{proof}

The corollary below follows from Lemmas \ref{lem_Lr-bounds} and \ref{lem_commutator}.

\begin{corollary}\label{H_eps_Wmr}
For any $m\in\{0,\ldots,q\}$ and any $r\in[1,+\infty]$, there exists $C_{m,r}>0$ such that $\Vert \mathcal{H}_\varepsilon g\Vert_{W^{m,r}} \leq C_{m,r} \Vert g\Vert_{W^{m,r}}$ for any $g\in W^{m,r}(\Omega)$ and any $\varepsilon\in(0,\inf(\frac{1}{1+L_\rho},\frac{1}{2L_\rho})]$.
\end{corollary}

We also have the following corollary, which extends Lemma \ref{lem_preserv_value_boundary}.

\begin{corollary}\label{cor_preserv_deriv_boundary}
For any $m\in\{1,\ldots,q-1\}$, any $\varepsilon\in(0,\frac{1}{1+L_\rho}]$, any $g\in\mathscr{C}^m(\Omega)$ and any $\alpha\in\N^n$ such that $\vert\alpha\vert=m$, we have $D^\alpha\mathcal{H}_\varepsilon g\in \mathscr{C}^0(\Omega)$ and $D^\alpha\mathcal{H}_\varepsilon g(x) = D^\alpha g(x)$ for every $x\in\partial\Omega$.
\end{corollary}

\begin{proof}
Recall that $\rho\in \mathscr{C}^q(\Omega)$ is a nonnegative real-valued function on $\Omega$, satisfying $\rho>0$ in $\mathring{\Omega}$ and \eqref{def_rho}, and such that all derivatives of $\rho$ of order $\leq q-1$ vanish along $\partial\Omega$.
It follows from \eqref{eq:comm-structure} and from the properties of the polynomials $P_{\alpha,k}$ that, for any $m\in\{1,\ldots,q-1\}$, 
any $g\in\mathscr{C}^m(\Omega)$ and any $\alpha\in\N^n$ such that $\vert\alpha\vert=m$, 
we have $[D^\alpha,\mathcal{H}_\varepsilon]g\in\mathscr{C}^0(\Omega)$ and $[D^\alpha,\mathcal{H}_\varepsilon]g(x)=0$ for every $x\in\partial\Omega$.
Since $D^\alpha g\in\mathscr{C}^0(\Omega)$, Lemma \ref{lem_preserv_value_boundary} implies that $\mathcal{H}_\varepsilon D^\alpha g\in\mathscr{C}^0(\Omega)$ and $\mathcal{H}_\varepsilon D^\alpha g(x)=D^\alpha g(x)$ for every $x\in\partial\Omega$.
The result follows, using \eqref{crochetDalphaHeps}.
\end{proof}

\subsection{Convergence properties}

\begin{lemma}\label{lem_CV_Heps}
For any $r\in[1,+\infty]$, there exists $C_{r,\eta,\rho}>0$ such that
\begin{equation}\label{eq:Heps-rate}
\Vert \mathcal{H}_\varepsilon g-g\Vert_{L^r} \leq C_{r,\eta,\rho}\,\varepsilon\,\Vert \nabla g\Vert_{L^r}  
\qquad\forall g\in W^{1,r}(\Omega) \qquad \forall \varepsilon\in\big(0,\inf\big(\tfrac{1}{1+L_\rho},\tfrac{1}{2L_\rho}\big)\big].
\end{equation}
In particular, if $g\in W^{1,\infty}(\Omega)$, $\mathcal{H}_\varepsilon g\to g$ uniformly on $\Omega$ as $\varepsilon\to 0$. 

In addition, for $r\in[1,+\infty)$ and $g\in L^r(\Omega)$, one has $\mathcal{H}_\varepsilon g \to g$ in $L^r(\Omega)$ as $\varepsilon\to 0$.
\end{lemma}

\begin{proof}
For $x\in\Omega$, by \eqref{def_Teps1}, we have
$$
\mathcal{H}_\varepsilon g(x)-g(x)
=\int_{\mathsf{B}_1}\eta(x') \big(g(x-\varepsilon\rho(x)x')-g(x)\big)\,dx' ,
$$
and writing $g(x-\varepsilon\rho(x)x')-g(x) = -\int_0^{\varepsilon\rho(x)} \langle\nabla g(x-tx'),x'\rangle\,dt$, since $\Vert x'\Vert_{\R^n}\leq 1$, we obtain, with the change of variables $t=\varepsilon\rho(x)\tau$ with $\tau\in[0,1]$,
$$
\vert \mathcal{H}_\varepsilon g(x)-g(x)\vert
\leq 
\varepsilon\rho(x) \int_0^1 \int_{\mathsf{B}_1}\eta(x') \Vert \nabla g(x-\varepsilon\tau\rho(x) x')\Vert_{\R^n} \,dx' \,d\tau 
= \varepsilon\rho(x) \int_0^1 (T_{\varepsilon\tau,0}\Vert\nabla g\Vert_{\R^n} )(x)\, d\tau
$$
where $T_{\varepsilon\tau,0}$ is defined by \eqref{def_Ts}. We infer \eqref{eq:Heps-rate} from \eqref{schur-weight_no_eps} in Lemma \ref{lem:schur-weight}.

Let us prove the last statement, by a density argument. By \eqref{eq:Lq-bound-Schur_not_eps} in Lemma \ref{lem_Lr-bounds}, $\mathcal{H}_\varepsilon$ is bounded on $L^r(\Omega)$, uniformly with respect to $\varepsilon$ small enough. Let $g\in L^r(\Omega)$. Let $(g_k)_{k\in\N}\subset W^{1,r}(\Omega)$ such that $g_k\to g$ in $L^r(\Omega)$ as $k\to +\infty$. Then
\begin{align*}
\Vert \mathcal{H}_\varepsilon g-g\Vert_{L^r}
&\leq \Vert \mathcal{H}_\varepsilon(g-g_k)\Vert_{L^r} + \Vert \mathcal{H}_\varepsilon g_k-g_k\Vert_{L^r} + \Vert g_k-g\Vert_{L^r}\\
&\leq \mathrm{Cst}\,\Vert g-g_k\Vert_{L^r} + \Vert \mathcal{H}_\varepsilon g_k-g_k\Vert_{L^r} + \Vert g_k-g\Vert_{L^r}.
\end{align*}
Fix $\delta>0$. First choose $k$ large enough so that $\Vert g_k-g\Vert_{L^r}\leq \delta$; then choose $\varepsilon>0$ small enough so that $\Vert \mathcal{H}_\varepsilon g_k-g_k\Vert_{L^r}\leq \delta$ by \eqref{eq:Heps-rate}. For such $k$ and $\varepsilon$, the right-hand side above is bounded by $(\mathrm{Cst}+1)\delta+\delta$. As $\delta>0$ is arbitrary, this implies $\Vert \mathcal{H}_\varepsilon g-g\Vert_{L^r}\to 0$ as $\varepsilon\to 0$. 
\end{proof}

\begin{lemma}\label{lem_CV_deriv_Heps}
For any $r\in[1,+\infty]$ and any $\alpha\in\N^n$ with $\vert\alpha\vert=m\in\{1,\ldots,q\}$, there exists $C_{\alpha,r,\eta,\rho}>0$ such that
\begin{equation}\label{eq:Heps-deriv-rate}
\Vert D^\alpha\mathcal{H}_\varepsilon g-D^\alpha g\Vert_{L^r}
\leq C_{\alpha,r,\eta,\rho}\,\varepsilon\,\Vert g\Vert_{W^{m+1,r}} 
\qquad \forall g\in W^{m+1,r}(\Omega) \quad \forall \varepsilon\in\big(0,\inf\big(\tfrac{1}{1+L_\rho},\tfrac{1}{2L_\rho}\big)\big].
\end{equation}
%
In addition, given any $r\in[1,+\infty)$ and $g\in W^{m,r}(\Omega)$, one has $D^\alpha\mathcal{H}_\varepsilon g \to D^\alpha g$ in $L^r(\Omega)$ as $\varepsilon\to 0$.
\end{lemma}

\begin{proof}
Writing $D^\alpha\mathcal{H}_\varepsilon g - D^\alpha g = [D^\alpha,\mathcal{H}_\varepsilon]g + \big(\mathcal{H}_\varepsilon D^\alpha g - D^\alpha g\big)$, the first term at the right-hand side is tackled by the estimate \eqref{eq:comm-Lp} of Lemma \ref{lem_commutator}. The second term is estimated by applying Lemma \ref{lem_CV_Heps} to $D^\alpha g$. We obtain \eqref{eq:Heps-deriv-rate}. 
The last statement of the lemma is proved by density, as for Lemma \ref{lem_CV_Heps}.
\end{proof}

\section{A general Schur test}\label{app_Schur}
In this appendix, we recall a general Schur-type lemma for integral operators on $L^r$, $1\leq r\leq +\infty$, which is used several times in Appendix \ref{app_convolution} and in Sections \ref{sec_general_class_quasilinear} and \ref{sec_main_result}.
This lemma is sometimes called Young's inequality for integral operators (see, e.g., \cite[Theorem 0.3.1]{Sogge_book}). We recall a proof for completeness.

\begin{lemma}\label{lem:Schur-general}
Let $(E,\mu)$ be a $\sigma$-finite measure space and let $K:E\times E\to\R$ be a measurable kernel. Define the integral operator $T$ by
$$
(Tf)(x)=\int_E K(x,y)\,f(y)\,d\mu(y)
$$
whenever the integral is well defined. Assume that
\begin{equation}\label{eq:Schur-row-col}
M_1 = \sup_{x\in E}\int_E \vert K(x,y)\vert\,d\mu(y) < +\infty,\qquad
M_2 = \sup_{y\in E}\int_E \vert K(x,y)\vert\,d\mu(x) < +\infty.
\end{equation}
Then, for every $r\in[1,+\infty]$, the operator $T$ extends boundedly to $L^r(E,\mu)$ and
\begin{equation}\label{eq:Schur-norm}
\Vert T\Vert_{L(L^r(E))} \leq M_1^{1-\frac{1}{r}}\,M_2^{\frac{1}{r}},
\end{equation}
with the convention $1/r=0$ when $r=+\infty$.
Actually, for $r=1$ (resp., for $r=+\infty$) only the assumption $M_2<+\infty$ (resp., $M_1<+\infty$) is required.
\end{lemma}

\begin{proof}
Without loss of generality, we assume that $K\geq 0$ (otherwise, replace $K$ by $\vert K\vert$). 

Let us first treat the case $r=1$.
For $f\in L^1(E,\mu)$,
\begin{align*}
\Vert Tf\Vert_{L^1}
&= \int_E \left\vert\int_E K(x,y)\,f(y)\,d\mu(y)\right\vert\,d\mu(x)
\leq \int_E\int_E K(x,y)\,\vert f(y)\vert\,d\mu(y)\,d\mu(x)\\
&= \int_E\left(\int_E K(x,y)\,d\mu(x)\right)\vert f(y)\vert\,d\mu(y)
\leq M_2\,\Vert f\Vert_{L^1}.
\end{align*}
Thus $\Vert T\Vert_{L(L^1)}\leq M_2$.

Let us now treat the case $r=+\infty$. For $f\in L^\infty(E,\mu)$ and $\mu$-almost every $x\in E$,
$$
\vert Tf(x)\vert \leq \int_E K(x,y)\,\vert f(y)\vert\,d\mu(y)
\leq \Vert f\Vert_{L^\infty}\int_E K(x,y)\,d\mu(y)\leq M_1\,\Vert f\Vert_{L^\infty}.
$$
Hence $\Vert T\Vert_{L(L^\infty)}\leq M_1$.

Let us then treat the case $1<r<+\infty$.
For $\mu$-almost every $x\in E$ fixed, by H\"older's inequality,
\begin{equation*}
\vert Tf(x)\vert^r
=\left\vert \int_E K(x,y)\,f(y)\,d\mu(y)\right\vert^r
\leq \left(\int_E K(x,y)\,d\mu(y)\right)^{r-1}\int_E K(x,y)\,\vert f(y)\vert^r\,d\mu(y).
\end{equation*}
%
Integrating 
over $x\in E$ and using the bounds in \eqref{eq:Schur-row-col} gives
\begin{multline*}
\int_E \vert Tf(x)\vert^r\,d\mu(x)
\leq \int_E \left(\int_E K(x,y)\,d\mu(y)\right)^{r-1}
\int_E K(x,y)\,\vert f(y)\vert^r\,d\mu(y)\,d\mu(x)\\
\leq M_1^{r-1}\int_E\int_E K(x,y)\,\vert f(y)\vert^r\,d\mu(y)\,d\mu(x) 
\leq M_1^{r-1}M_2\int_E \vert f(y)\vert^r\,d\mu(y)
\end{multline*}
where we have used the Fubini theorem to obtain the latter inequality, which gives \eqref{eq:Schur-norm}.
\end{proof}

\section{Discretization hypotheses}\label{app_discretization_technical}

This appendix collects technical material related to Section \ref{sec_examples_discretizations}.
Its purpose is twofold: first, to explain how Assumption \ref{H_discretization} reduces to properties of a projector or quasi-interpolation operator; second, to record representative examples and limitations.

\subsection{Generalities}
\paragraph{A projector reduction principle.}
The following observation explains why most practical discretizations are naturally expressed in terms of a family of projectors.

\begin{lemma}\label{lem_projector_reduction}
Let $V_N\subset Z$ be a finite-dimensional subspace and let $Q_N:X\to V_N$ be a linear map such that $Q_N u=u$ for every $u\in V_N$.
Define $P_N=Q_N:X\to V_N$ and let $R_N:V_N\hookrightarrow Z$ be the inclusion map.
Then $P_N R_N=\mathrm{id}_{V_N}$.
Moreover, the norms induced on $V_N$ by Assumption \ref{H_discretization} are simply the restrictions of the continuous norms: $\Vert u\Vert_{X_N}=\Vert u\Vert_X$ and $\Vert u\Vert_{Z_N}=\Vert u\Vert_Z$ for every $u\in V_N$.
Therefore, Assumption \ref{H_discretization} is equivalent to the uniform boundedness of $Q_N$ on $X$ and on $Z$, together with the approximation estimate in $X$, the strong convergence in $Z$, and the approximation estimate of $Q_N$ on the output space $Y$.
\end{lemma}

\paragraph{A density-extension lemma.}
The next lemma is convenient when the strong convergence $Q_N z\to z$ in $Z$ is first proved only on a dense subspace.

\begin{lemma}\label{lem_dense_extension_QN}
Let $(Q_N)_{N\in\N^*}$ be a family of linear operators on $Z$ such that $\Vert Q_N\Vert_{L(Z)}$ is uniformly bounded in $N$. Assume that there exists a dense subspace $Z_0\subset Z$ such that $Q_N z\to z$ in $Z$ for every $z\in Z_0$.
Then $Q_N z\to z$ in $Z$ for every $z\in Z$.
\end{lemma}

\begin{proof}
Let $z\in Z$ and let $(z_m)_m\subset Z_0$ converge to $z$ in $Z$.
For every $N$ and every $m$,
\begin{align*}
\Vert Q_N z-z\Vert_Z
&\leq \Vert Q_N(z-z_m)\Vert_Z + \Vert Q_N z_m-z_m\Vert_Z + \Vert z_m-z\Vert_Z \\
&\leq \big(\sup_N \Vert Q_N\Vert_{L(Z)}+1\big)\Vert z-z_m\Vert_Z + \Vert Q_N z_m-z_m\Vert_Z.
\end{align*}
First choose $m$ so large that the first term is arbitrarily small, then let $N\to+\infty$ and use the convergence on $Z_0$.
\end{proof}

\paragraph{A local quasi-interpolation criterion.}
The next proposition isolates the properties that are really needed in local schemes such as finite volumes with smooth reconstruction, finite elements, splines, or meshfree partition-of-unity methods.

\begin{lemma}\label{lem_local_quasi_interpolation}
Assume $X=L^2(\Omega,\R^d)$ and $Z=H^s(\Omega,\R^d)$ with $s>\frac{n}{2}+1$, and let $Y=W^{1,\infty}(\Omega,\R^d)\cap L^2(\Omega,\R^d)$ endowed with its natural norm.
Let $V_N\subset Z$ be a finite-dimensional subspace and let $Q_N:X\to V_N$ be a projector.
Assume that there exist $C_{\mathrm{stab}},C_Y>0$ and $\gamma>0$, independent of $N$, such that
$\Vert Q_N\Vert_{L(X)}\leq C_{\mathrm{stab}}$ and $\Vert Q_N\Vert_{L(Z)}\leq C_{\mathrm{stab}}$, 
and
\begin{equation}\label{local_QN_Lip}
\Vert Q_N g-g\Vert_X\leq \frac{C_Y}{N^\gamma}\Vert g\Vert_Y
\qquad
\forall g\in Y.
\end{equation}
Assume moreover that there exists a dense subspace $Z_0\subset Z$ such that
$Q_N z\to z$ in $Z$ for every $z\in Z_0$.
Then the pair $(P_N,R_N)$ associated with $Q_N$ by Lemma \ref{lem_projector_reduction} satisfies Assumption \ref{H_discretization}. 
\end{lemma}

\begin{proof}
By Lemma \ref{lem_projector_reduction}, it remains to check \eqref{CV_X_Z}, \eqref{CV_X_Y} and \eqref{CV_Z}.
The estimate \eqref{CV_X_Y} is exactly the assumption \eqref{local_QN_Lip}.
Since $s>\frac{n}{2}+1$, the Sobolev embedding gives
$\Vert z\Vert_Y\leq C_{\mathrm{emb}}\Vert z\Vert_Z$ for every $z\in Z$.
Applying \eqref{local_QN_Lip} to $g=z\in Z$, we obtain
$\Vert Q_N z-z\Vert_X \leq \frac{C_Y C_{\mathrm{emb}}}{N^\gamma}\Vert z\Vert_Z$,
which is exactly \eqref{CV_X_Z}, up to redefining $C_{10}$.
The strong convergence \eqref{CV_Z} follows from Lemma \ref{lem_dense_extension_QN}.
\end{proof}

\begin{remark}\label{rem_local_quasi_interpolation}
Lemma \ref{lem_local_quasi_interpolation} shows that, in the PDE regime relevant to Sections \ref{sec_general_class_quasilinear} and \ref{sec_main_result}, the genuinely important estimate is the approximation of Lipschitz or $W^{1,\infty}$ outputs in the pivot norm $X$.
This is why Assumption \ref{H_Y} is the right one for the main approximation theorem.
\end{remark}

\subsection{Finite volumes with smooth blob reconstruction}
We now explain how the previous criterion applies to corrected finite-volume or particle-in-cell type reconstructions.

Let $(\Omega_i^N)_{1\leq i\leq N}$ be a shape-regular partition\footnote{Shape-regular means that the cells have uniformly comparable diameters and volumes, satisfy a uniform interior-ball condition, and have uniformly bounded overlap of neighboring patches after enlargement by a factor comparable to the meshsize. This is a standard notion for finite elements and finite volumes, see, e.g., \cite[Chapter 4]{BrennerScott} and \cite[Chapter 3]{EymardGallouetHerbin}.} of $\Omega$ with meshsize $h_N$ satisfying $h_N\leq C_\Omega N^{-\gamma}$.
Choose points $x_i^N\in\Omega_i^N$.
Assume that one has a family of smooth basis functions $(\theta_i^N)_{1\leq i\leq N}\subset C^\infty(\overline\Omega)$ such that $\mathrm{supp}(\theta_i^N)\subset B(\Omega_i^N,C h_N)$ and $\Vert D^\alpha \theta_i^N\Vert_{L^\infty(\Omega)}\leq C_\alpha h_N^{-\vert\alpha\vert}$ for every multi-index $\alpha$ with $\vert\alpha\vert\leq s$, the overlap is uniformly bounded in $N$, the family reproduces constants, i.e., $\sum_{i=1}^N\theta_i^N(x)=1$ for every $x\in\Omega$, and it is biorthogonal to cell averages, i.e., $\frac{1}{\vert\Omega_j^N\vert}\int_{\Omega_j^N}\theta_i^N(x)\,dx=\delta_{ij}$.
Define
$$
Q_N y=\sum_{i=1}^N \Big(\frac{1}{\vert\Omega_i^N\vert}\int_{\Omega_i^N} y(x)\,dx\Big)\theta_i^N.
$$
Then $Q_N$ is a projector onto the reconstruction space $V_N=\mathrm{span}\{\theta_1^N,\dots,\theta_N^N\}$, and we have the following result.

\begin{lemma}\label{lem_blob_scheme_verification}
The family $(Q_N)_N$ satisfies \eqref{local_QN_Lip} and is uniformly bounded on $L^2(\Omega)$. If, in addition, $(Q_N)_N$ is uniformly bounded on $H^s(\Omega)$ and converges strongly in $H^s(\Omega)$ on a dense subspace of $H^s(\Omega)$, then Assumption \ref{H_discretization} holds through Lemma \ref{lem_local_quasi_interpolation}.
\end{lemma}

\begin{proof}
Let $g\in W^{1,\infty}(\Omega,\R^d)$.
Using the reproduction of constants, for $x\in\Omega$ one has
$$
Q_N g(x)-g(x)
=
\sum_{i=1}^N \theta_i^N(x)
\Big(
\frac{1}{\vert\Omega_i^N\vert}\int_{\Omega_i^N} g(x')\,dx'-g(x)
\Big).
$$
Because of the support condition, only indices $i$ such that $\Omega_i^N$ is at distance $\mathrm{O}(h_N)$ from $x$ contribute to the sum.
Using the Lipschitz regularity of $g$ and the bounded overlap condition, we obtain the pointwise estimate
$\vert Q_N g(x)-g(x)\vert\leq C h_N \Vert g\Vert_{W^{1,\infty}}$.
Taking the $L^2$-norm gives
$$
\Vert Q_N g-g\Vert_{L^2}
\leq
C h_N \Vert g\Vert_{W^{1,\infty}}
\leq
\frac{C}{N^\gamma}\Vert g\Vert_Y.
$$
This proves \eqref{local_QN_Lip}.
The $L^2$-boundedness follows from the local support, bounded overlap, and volume comparability of the shape-regular partition: by Jensen's inequality,
$$
\bigg\vert \frac{1}{\vert\Omega_i^N\vert}\int_{\Omega_i^N}g(x')\,dx'\bigg\vert^2\leq\frac{1}{\vert\Omega_i^N\vert}\int_{\Omega_i^N}\vert g(x')\vert^2\,dx',
$$
and the support and overlap assumptions give
$$
\Vert Q_N g\Vert_{L^2}^2\leq C\sum_i\int_{\Omega_i^N}\vert g(x')\vert^2\,dx'=C\Vert g\Vert_{L^2}^2.
$$
The uniform $H^s$-stability is a standard local quasi-interpolation estimate, see for instance \cite{BabuskaMelenk, BrennerScott}.
Then Lemma \ref{lem_local_quasi_interpolation} applies.
\end{proof}

\begin{remark}
The blob construction above includes, after standard corrections, many smooth finite-volume or particle-in-cell type reconstructions.
It also covers a large class of meshfree partition-of-unity reconstructions and moving least squares schemes.
What is essential is not the exact formula of the basis, but the combination of locality, bounded overlap, reproduction of the sampled moments, and uniform stability in the strong norm.
\end{remark}

\subsection{Finite elements, splines, and meshfree quasi-interpolation}
The same abstract criterion applies to many classical Galerkin-type constructions, provided one chooses a projector or quasi-interpolation operator that is simultaneously stable on $X$ and on $Z$.

\begin{lemma}\label{lem_FE_spline_meshfree}
Assume $X=L^2(\Omega,\R^d)$, $Z=H^s(\Omega,\R^d)$ with $s>\frac{n}{2}+1$, and $Y=W^{1,\infty}(\Omega,\R^d)\cap L^2(\Omega,\R^d)$.
Suppose that $V_N\subset Z$ is either a conforming finite element space on a shape-regular quasi-uniform mesh, or a spline space on a quasi-uniform knot sequence, or a meshfree partition-of-unity or moving least squares reconstruction space.
Assume that there exists a projector or quasi-interpolation operator $Q_N:X\to V_N$ such that
$\Vert Q_N\Vert_{L(X)}+\Vert Q_N\Vert_{L(Z)}\leq C$
uniformly in $N$, and such that
$\Vert Q_N g-g\Vert_X\leq C h_N \Vert g\Vert_Y$ for every $g\in Y$.
Assume moreover that $Q_N z\to z$ in $Z$ on a dense subspace of $Z$.
Then Assumption \ref{H_discretization} holds.
\end{lemma}

\begin{proof}
This is a direct application of Lemma \ref{lem_local_quasi_interpolation}.
The existence of such quasi-interpolation operators is classical in finite elements and splines (see \cite{BrennerScott, ScottZhang}), and standard in partition-of-unity and meshfree approximations (see \cite{BabuskaMelenk}).
\end{proof}

\begin{remark}
This lemma is intentionally formulated at the level of the projector $Q_N$.
Indeed, for finite elements and splines, several natural choices of projector coexist, and not all of them are simultaneously stable on $L^2$ and on high Sobolev norms.
The abstract framework does not force a particular construction; it only requires the stability and approximation properties that are later used in the convergence proof.
\end{remark}

\subsection{Spectral Galerkin discretizations}
The strong discrete well-posedness assumption \ref{H_Aepsilon_N} is particularly natural in spectral Galerkin settings.

\begin{lemma}\label{lem_spectral_scheme_verification}
Assume that $X$ is a Hilbert space, that $Z\subset X$ is dense, and that $V_N\subset Z$ is a family of finite-dimensional subspaces.
Let $Q_N:X\to V_N$ be the orthogonal projector onto $V_N$.
Assume that $\Vert Q_N\Vert_{L(Z)}\leq C$ uniformly in $N$ and that there exist $C_{10}>0$ and $\gamma>0$ such that
\begin{align*}
\Vert Q_N z-z\Vert_X\leq \frac{C_{10}}{N^\gamma}\Vert z\Vert_Z \qquad \forall z\in Z, \\
\Vert Q_N g-g\Vert_X\leq \frac{C_{10}}{N^\gamma}\Vert g\Vert_Y \qquad \forall g\in Y.
\end{align*}
Assume moreover that $Q_N z\to z$ in $Z$ for every $z\in Z$ and that, for every $\varepsilon\in(0,\varepsilon_0]$, the subspace $V_N$ is invariant under the discrete intertwining structure used for $A_\varepsilon$, and $Q_N$ commutes with the corresponding intertwining operator. Assume finally that $S_\varepsilon^N = P_N S_\varepsilon R_N$ realizes the discrete graph norm, namely that $\Vert u\Vert_{Z_N}$ is uniformly equivalent to $\Vert u\Vert_{X_N} + \Vert S_\varepsilon^N u\Vert_{X_N}$ (which holds in the spectral case, where $V_N$ is spanned by eigenfunctions and $S_\varepsilon$ commutes with $Q_N$).

Then Assumption \ref{H_discretization} holds.
Moreover, if the continuous dissipativity hypothesis of Lemma \ref{lem_discrete_stab} is satisfied, then the strong discrete assumption \ref{H_Aepsilon_N} holds. 
\end{lemma}

\begin{proof}
The first part follows from Lemma \ref{lem_projector_reduction}.
The second part is exactly the content of Lemmas \ref{lem_H_regZ_N}, \ref{lem_discrete_stab} and \ref{lem_discrete_intertwining}.
\end{proof}

\begin{remark}
This lemma explains why the strong assumption \ref{H_Aepsilon_N} is natural in Galerkin and spectral settings: orthogonality gives the discrete stability, while invariance and commutation give the discrete intertwining relation.
This mechanism is absent in most local reconstruction schemes, which is why Proposition \ref{prop_yepsilonN} should be viewed as a convenient sufficient criterion rather than as the generic form of all discretizations.
\end{remark}

\subsection{Point sampling is not bounded on \texorpdfstring{$L^2$}{L2}}

\begin{lemma}\label{lem_point_sampling_not_bounded}
Let $X=L^2(\Omega)$, where $\Omega\subset\R^n$ has nonempty interior.
For any fixed $x_0\in\Omega$, the evaluation map $y\mapsto y(x_0)$ is not continuous on $L^2(\Omega)$.
In particular, point-sampling maps of the form $P_N y=(y(x_1^N),\dots,y(x_N^N))$ cannot satisfy Assumption \ref{H_discretization}.
\end{lemma}

\begin{proof}
Choose $\varphi\in C_c^\infty(\Omega)$, nonzero, supported in a small ball around $x_0$, and define $y_\delta(x)=\delta^{-\frac{n}{2}}\varphi\big(\frac{x-x_0}{\delta}\big)$.
Then $\Vert y_\delta\Vert_{L^2}$ remains bounded as $\delta\to 0$, while $y_\delta(x_0)=\delta^{-\frac{n}{2}}\varphi(0)\to +\infty$.
Therefore evaluation at $x_0$ is not a bounded linear functional on $L^2(\Omega)$.
\end{proof}

\begin{remark}\label{rem_related_particle_limits}
The examples discussed in this appendix should also be compared with several neighboring viewpoints in the literature.
The deterministic particle approximation of nonlinear diffusion in \cite{LionsMas-Gallic} may be seen as an early instance of a finite interacting approximation of a PDE.
The mean-field, hydrodynamic, and graph-limit perspective developed in \cite{PaulTrelat} is complementary to the present work: it analyzes passages from interacting finite-dimensional systems to continuum limits, whereas the present paper adds an upstream kernelization step from an unbounded PDE operator $A$ to a boundary-compatible interacting operator $A_\varepsilon$.
Finally, the graph-limit treatment of a nonlinear heat equation in \cite{Medvedev_SIMA2014} provides another example where interaction structures and continuum PDEs are linked through a limiting procedure.
\end{remark}

\let\OLDthebibliography\thebibliography
\renewcommand\thebibliography[1]{
  \OLDthebibliography{#1}
  \setlength{\parskip}{3pt}
  \setlength{\itemsep}{1pt plus 0.3ex}
}

\end{document}